\documentclass[11pt]{article}

\usepackage[letterpaper,margin=1in]{geometry}
\usepackage{setspace}
\usepackage{amsmath,amsfonts,amssymb,amsthm}
\usepackage{mathtools}
\allowdisplaybreaks
\usepackage{bm}
\usepackage[inline]{enumitem}
\usepackage[table]{xcolor}
\usepackage{graphicx}
\usepackage{multirow}
\usepackage{booktabs}
\usepackage{microtype}
\usepackage{algorithm}
\usepackage{algpseudocode}
\usepackage{tikz}
\usepackage{natbib}
\usepackage{bibunits}
\usepackage[colorlinks=true,linkcolor=blue,citecolor=blue,urlcolor=blue]{hyperref}
\hypersetup{
  pdftitle={From Frequentist to Bayesian Contextual Optimization},
  pdfauthor={Zhuojun Xie, Adam Abdin, and Yiping Fang},
  pdfkeywords={Prescriptive Analytics; Contextual Stochastic Optimization; Decision-Focused Learning; Generalized Bayesian Inference; Gibbs Posterior}
}

\bibpunct[, ]{(}{)}{,}{a}{}{,}
\newtheorem{theorem}{Theorem}
\newtheorem{lemma}{Lemma}
\newtheorem{proposition}{Proposition}
\newtheorem{corollary}{Corollary}
\newtheorem{assumption}{Assumption}
\newtheorem{condition}{Condition}
\theoremstyle{definition}
\newtheorem{definition}{Definition}
\newtheorem{example}{Example}
\theoremstyle{remark}
\newtheorem{remark}{Remark}

\makeatletter
\renewenvironment{proof}[1]{%
  \par\pushQED{\qed}\normalfont
  \topsep6\p@\@plus6\p@\relax
  \trivlist
  \item[\hskip\labelsep\itshape #1\@addpunct{.}]\ignorespaces
}{%
  \popQED\endtrivlist\@endpefalse
}
\makeatother
\newcommand{\Halmos}{}
\newcommand{\halmos}{}

\newcommand{\SingleSpacedXI}{\singlespacing}
\newcommand{\OneAndAHalfSpacedXI}{\onehalfspacing}
\newcommand{\ECHead}[1]{%
  \begin{center}\Large\bfseries #1\end{center}\medskip
}

\newcommand{\R}{\mathbb{R}}
\newcommand{\E}{\mathbb{E}}
\newcommand{\W}{\mathcal{W}}
\renewcommand{\P}{\mathbb{P}}

\newcommand{\Q}{\mathbb{Q}}
\newcommand{\Gibbs}{\rho_{n}}

\newcommand{\lra}[1]{\left( #1 \right)}
\newcommand{\lrb}[1]{\left[ #1 \right]}
\newcommand{\lrc}[1]{\left\{ #1 \right\}}

\newcommand{\norm}[1]{\left\lVert #1 \right\rVert}
\newcommand{\abs}[1]{\left\lvert #1 \right\rvert}
\newcommand{\KL}[2]{\mathbb{D}_{\mathrm{KL}}\lra{#1 \, \| \, #2}}
\newcommand{\bcohl}[1]{\cellcolor{gray!15}{#1}}
\DeclareMathOperator*{\argmin}{arg\,min}
\DeclareMathOperator*{\argmax}{arg\,max}

\title{From Frequentist to Bayesian Contextual Optimization}
\author{%
  Zhuojun Xie \quad Adam Abdin \quad Yiping Fang\thanks{Corresponding author.}\\[0.4em]
  \small Laboratoire G\'enie Industriel, CentraleSup\'elec, Universit\'e Paris-Saclay, France\\[-0.1em]
  \small \texttt{xie.zhuojun@centralesupelec.fr} \quad
  \texttt{adam.abdin@centralesupelec.fr}\\[-0.1em]
  \small \texttt{yiping.fang@centralesupelec.fr}
}
\date{}

\begin{document}
\maketitle

\begin{abstract}
In data-driven contextual stochastic optimization, existing approaches are predominantly frequentist: they commit to a single predictive model and treat it as ground truth, yielding prescriptions that are fragile to model uncertainty and sampling variability in small- and moderate-sample regimes.
We propose Bayesian contextual optimization (BCO), a framework that maintains a Gibbs posterior over the parameter space. This decision-focused posterior weights candidate models by their empirical decision quality rather than statistical fit, thereby avoiding commitment to a potentially misspecified likelihood while encoding the full distribution of plausible models consistent with the data.
A Bayesian contextual policy is then derived by minimizing the expected decision cost under the posterior predictive distribution, hedging prescriptions against model uncertainty by aggregating over the posterior.
We establish three theoretical guarantees: (i) the Gibbs posterior concentrates exponentially fast around the frequentist best-in-class parameter set; (ii) BCO can strictly improve over the frequentist best-in-class policy when model uncertainty is non-negligible; and (iii) BCO attains an $O(n^{-1/2})$ excess risk rate against the oracle over all probability measures up to a misspecification term and an oracle aggregate gap.
Computationally, we tailor a variational inference scheme that has the same per-iteration cost as frequentist alternatives and a gradient-free Metropolis--Hastings algorithm that handles nondifferentiable problems.
Numerical experiments on two-stage shipment planning, contextual newsvendor, and return-constrained portfolio problems confirm that BCO consistently reduces out-of-sample cost and variance relative to kernel estimators and decision-focused baselines, with the most pronounced gains in small- and moderate-sample regimes under substantial model uncertainty.
\end{abstract}

\noindent\textbf{Keywords:} Prescriptive Analytics; Contextual Stochastic Optimization; Decision-Focused Learning; Generalized Bayesian Inference; Gibbs Posterior

\begin{bibunit}[plainnat]
\section{Introduction}\label{sec:1}
Data-driven contextual stochastic optimization (CSO)~\citep{Bertsimas-2020-MS} prescribes decisions under uncertainty by exploiting the statistical relationship between observable context and uncertain outcomes. 
Two complementary paradigms have emerged: decision-rule optimization~\citep{Qi-2023-MS}, which directly learns a context-to-decision mapping, and estimate-then-optimize (ETO), which first constructs a predictive model of the uncertainty and then optimizes against it~\citep{EJORReview}. Both paradigms offer frameworks that demonstrate strong empirical performance on large samples and provide out-of-sample guarantees~\citep{Bertsimas-2020-MS, Ban-2019-OR}. This paper focuses on ETO, which supports explicit uncertainty modeling, enables risk-averse objectives, and naturally accommodates uncertain constraints.

Within ETO, the emerging \emph{integrated estimate-optimize}~\citep[IEO]{Elmachtoub-2023-ARXIV}, also referred to as \emph{decision-focused learning} (DFL) ~\citep{DFLreview} or \emph{smart predict-then-optimize} (SPO)~\citep{SPO}, replaces conventional statistical losses with downstream decision objectives, thereby improving consistency between the estimation and optimization~\citep{Ho-Nguyen-2022-MS}. Yet, a fundamental statistical limitation persists across all ETO variants, including IEO: existing frameworks are predominantly frequentist, treating the fitted predictive model as the truth once trained. As such, there is no principled mechanism for quantifying uncertainty about which estimated model best represents the data-generating process. As we show in Figure~\ref{fig:intro}, this limitation has significant implications for decision-making in finite-sample regimes that arise routinely in OR/MS practice.

The figure illustrates two concrete consequences of the frequentist ETO paradigms. The first is \emph{decision instability}: because frequentist estimators are random functions of the observed sample, they can deviate markedly from the population risk minimizer in small-to-moderate sample regimes~\citep{Spokoiny-2012-AOS}. For example, at sample size $n \in \lrc{256,\,512,\, 1024}$, the optimality gaps of IEO, a decision-focused method that trains the predictor by minimizing downstream decision errors, vary substantially across independent training sets. While regularization and cross-validation can reduce this variability, they offer no principled mechanism for propagating the remaining estimation uncertainty into downstream decisions. 
The second consequence is the \emph{neglect of model uncertainty}. By committing to a single fitted model, current ETO methods, ranging from nonparametric kernel methods to distributional estimators, discard the full set of models plausible under the data. Models not selected due to empirical suboptimality may still carry prescriptive information, and ignoring them prevents the induced policy from reaching its full potential. In Figure~\ref{fig:intro}, even at $n = 4096$, IEO exhibits a higher mean cost and greater variability than the Bayesian framework we propose. 
Beyond their frequentist nature, existing IEO methods are heavily tailored to specific problem classes. Incorporating decision feedback into the learning pipeline introduces problem-specific challenges~\citep{DFLreview}, confining their applicability to a narrow subset of OR/MS problems.

\begin{figure}[t]
    \centering
    \caption{A Comparison of Frequentist IEO and BCO in a Contextual Newsvendor Problem}\label{fig:intro}
    \includegraphics[width=1.\linewidth]{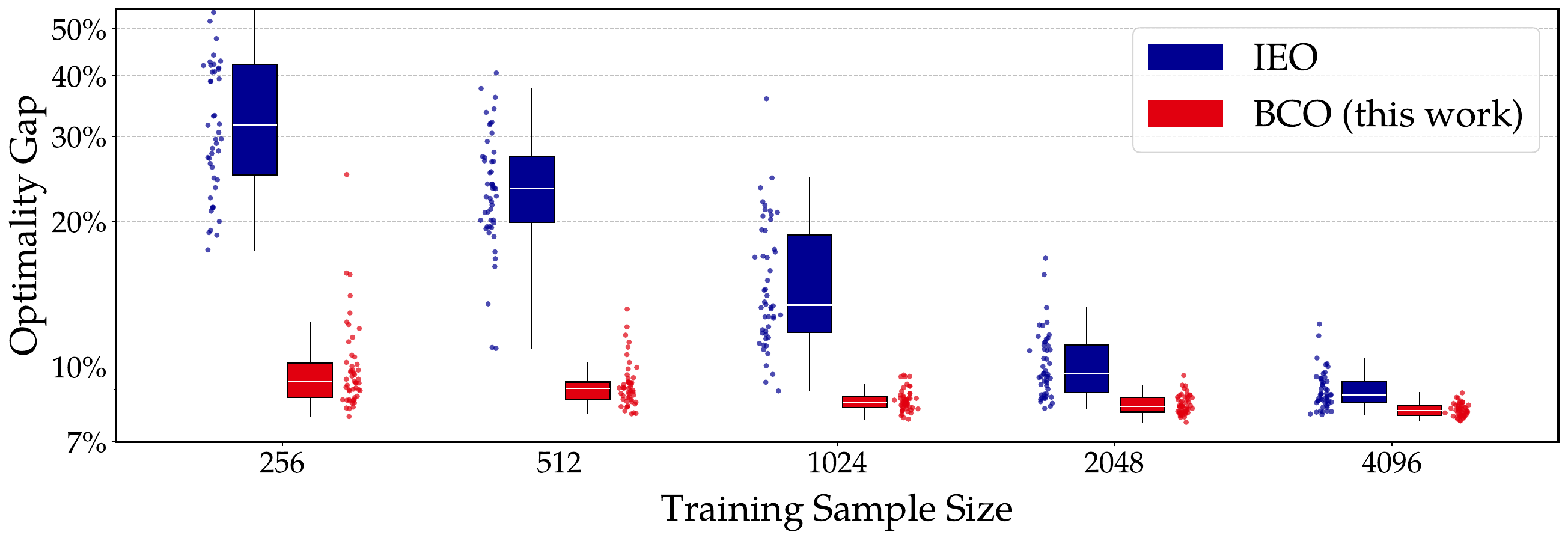}
    \begin{minipage}{\textwidth}
    \vskip 0.1cm
    \small
    \textit{Notes}. Data-generating process follows the second model in Section~\ref{sec:6_2} with $\tau = 0.75$. Optimality gap computed on a $2^{15}$-sample test set. Predictive model: fully-connected network with 8-dimensional input, two hidden layers of width 128 (ReLU), scalar output. \textbf{IEO}: integrated estimate-optimize (equivalent to decision-rule optimization here); \textbf{BCO}: Bayesian contextual optimization proposed in this paper. Results from 50 independent runs per sample size; each run uses a validation set for early stopping and model selection. See Section~\ref{sec:6_2} for details.
    \end{minipage}
\end{figure}


To bridge the gaps, we propose \emph{Bayesian contextual optimization} (BCO) for \emph{offline} data-driven CSO. 
Throughout this paper, \emph{model uncertainty} refers to the epistemic uncertainty about which parameter of the predictive model incurs the lowest population decision risk.
Rather than committing to a single estimator, BCO maintains a belief distribution over the parameter space, updated according to an \emph{operational likelihood} (see Equation~\eqref{oplike}) that scores each candidate model by the empirical decision quality incurred by its predictions.
The resulting posterior, known as the \textit{Gibbs posterior}~\citep{Wenxin-2008-AS}, reflects model uncertainty through the lens of decision quality rather than statistical fit, and therefore mitigates the impact of model misspecification. This construction unifies DFL and Bayesian statistics within a data-driven CSO framework. 

Building on this foundation, we study the theoretical properties and practical implementations of BCO. A structural advantage of the Gibbs posterior is that its density is fully specified by pointwise evaluation of the decision loss, requiring no problem-specific gradient or optimality conditions. Standard Bayesian inference routines can therefore be applied directly across the breadth of OR/MS problem classes, resolving the applicability  gap that limits existing decision-focused frameworks.

Our contributions are summarized as follows.
\begin{itemize}
    \item We propose a BCO framework built on a decision-focused Gibbs posterior: a data- and problem-dependent belief over the parameter space, updated in proportion to empirical decision quality rather than statistical likelihood. To the best of our knowledge, BCO is the first data-driven CSO framework that simultaneously \textup{(i)} quantifies model uncertainty through a probabilistic belief; \textup{(ii)} aligns belief updating with the downstream decision objective; and \textup{(iii)} hedges decisions against model uncertainty.
    
    \item We establish nonasymptotic concentration of the Gibbs posterior around the best-in-class parameter set, with an exponential rate in sample size under mild assumptions (Theorem~\ref{thm:1}), confirming that uncertainty quantification in finite-sample regimes does not sacrifice asymptotic consistency. Corollaries~\ref{cor:1} and~\ref{cor:2} instantiate this result for canonical OR/MS problems, e.g., linear programs, mixed-integer linear programs, newsvendor, and mean-variance portfolio models, with explicit, computable concentration rates. 
    
    \item Aligned with Bayes decision theory, we derive a Bayesian contextual policy (BCP) that minimizes the expected cost under the posterior predictive distribution, so model uncertainty propagates into the objective function against which decisions are selected. Two nonasymptotic out-of-sample guarantees are established for this policy. The first (Proposition~\ref{prop:1}) decomposes a Bayesian–frequentist gap into three interpretable factors (model misspecification, model uncertainty, and posterior concentration), and shows BCP can strictly outperform the frequentist best-in-class policy when model uncertainty is nonnegligible. The second (Theorem~\ref{thm:2}) establishes an oracle-type bound with an $O(n^{-1/2})$ excess risk rate against the oracle measure over \emph{all} probability measures up to a misspecification term and an oracle aggregate gap.

    \item We demonstrate that BCO is readily implementable with standard Bayesian inference algorithms for three OR/MS problems. Variational inference applies in differentiable settings with computational costs comparable to frequentist decision-focused alternatives, while gradient-free Metropolis–Hastings sampling extends BCO to nondifferentiable problems. Numerical experiments on two-stage shipment planning, contextual newsvendor, and return-constrained portfolio problems corroborate that BCO improves out-of-sample cost and decision stability relative to kernel-estimator and frequentist decision-focused baselines, with the most pronounced gains in small- and moderate-sample regimes under substantial model uncertainty.
    
\end{itemize}

\subsection{Related Literature}\label{sec:review}
We review three strands of the CSO literature most relevant to our framework: (i) statistical and machine learning approaches for conditional uncertainty estimation, (ii) Bayesian frameworks for decision-making under uncertainty, and (iii) DFL methods that incorporate optimization feedback. BCO sits at the intersection of all three; the review below identifies each strand's contributions and the gaps that motivate our framework. For a comprehensive overview of data-driven CSO, we refer readers to the survey by~\cite{EJORReview}.

\cite{Bertsimas-2020-MS} introduced the foundational idea of using local machine learning methods, including $k$-nearest neighbors ($k$NN), Nadaraya-Watson estimators, trees, and random forests, to estimate conditional distributions for unseen scenarios by exploiting contextual affinity between new and historical observations.
These local methods were subsequently extended to problems with uncertainty in the constraints~\citep{Rahimian-2023-SIAMOPT}.
Several remedies sharpen local methods in small-sample regimes: globally optimal weighting via a reproducing kernel Hilbert space approach~\citep{Bertsimas-2022-OR}, distributionally robust bootstrap~\citep{Bertsimas-2022-MP}, and variance-regularized formulations~\citep{Wang-2026-MSOM}.


A structural limitation shared by all nonparametric methods is their vulnerability to the curse of dimensionality: sample complexity grows rapidly with the dimension of the context, hampering applications with high-dimensional features~\citep{Wainwright-2019-Book}.
To circumvent this,~\cite{Kannan-2025-OR} proposed a residual-based approach that assumes a context-dependent heteroscedastic Gaussian noise structure, approximating the conditional distribution by scaling historical residuals of a fitted parametric regressor with a jackknife bias correction.~\cite{Sim-2025-OR} further combined residual-based estimation with a robust satisficing model to address decision \emph{fragility} under distributional ambiguity.
Despite their technical diversity, these approaches share two limitations that restrict prescriptive performance: overlooking model uncertainty~\citep{Rockova-2020-AOS} and decoupling from optimization~\citep{SPO}. Our framework addresses both limitations simultaneously through a decision-focused belief and a Bayesian aggregate policy.

Bayesian methodologies remain underutilized in data-driven CSO, primarily because of the tension between parametric density specification and the data-driven principle, and in part because of computational difficulties.
Belief updating under Bayes' rule requires explicit specification of a likelihood function, or at least a summary statistic.
To see this, consider the example of Bayesian inventory management, where parametric densities such as normal, lognormal, and Weibull are often assumed~\citep{Azoury-2009-MS, Harrison-2012-MS, Luo-2023-MSOM} to facilitate tractable belief updating.
Beyond inventory,~\cite{Wu-2018-SIAMOPT} studied the theoretical properties of Bayesian posteriors in data-driven optimization with general objective functions under various risk functionals, assuming the likelihood is well-specified; that is, the true data-generating distribution lies within the assumed likelihood family (the M-closed assumption).~\cite{Shapiro-2023-SIAMOPT} relaxed this M-closed assumption with a distributionally robust formulation whose ambiguity set is defined via Kullback–Leibler (KL) divergence to address model uncertainty under misspecification.

Manual specification of a likelihood is difficult to justify with finite samples in a purely data-driven scenario; under likelihood misspecification, the consequent posterior can yield unreliable conclusions~\citep{Nott-2023-ARSA}.
BCO addresses this limitation by substituting a parametric likelihood with an operational likelihood, analogous to the generalized Bayesian updating principle of~\cite{Bissiri-2016-JRSSB}. 
The resulting Gibbs posterior captures model uncertainty through prescriptive performance rather than statistical fit and is therefore robust to misspecification. 
As a special case, when the downstream problem is the classical newsvendor, our framework recovers Bayesian inference for the conditional $\tau$-quantile (critical fractile) under the asymmetric Laplace likelihood, connecting BCO to the classical Bayesian inventory and statistical literature.

A number of decision-focused frameworks have been proposed for data-driven CSO to incorporate downstream optimization feedback into the learning process. For contextual linear programs,~\cite{SPO} introduced the \texttt{SPO+} loss, a surrogate of decision regret, to train point-prediction models and established its advantage under considerable model misspecification.~\cite{Sun-2023-ICML} exploited the optimality margin of a linear program to learn a point predictor via inverse optimization;~\cite{Dias-2025-OR} leveraged the Karush–Kuhn–Tucker (KKT) conditions to learn an application-oriented forecaster for two-stage linear programs with recourse. Beyond linear programs,~\cite{Kallus-2023-MS} developed a decision-focused splitting criterion for random forests with a perturbation-based approximation to manage computational intractability;~\cite{Rafaela-2026-EJOR} further tailored decision-focused tree construction to two-stage linear programs.
On the parametric distributional side,~\cite{Qi-2025-OR} proposed an integrated estimate-optimize framework for problems with convex objective functions, constructing the conditional distribution estimator over a discrete support by directly minimizing decision errors. 

Most of this work is tailored to specific problem classes. Furthermore, none of these frameworks quantify finite-sample model uncertainty.
Our framework closes both gaps: the decision-focused posterior quantifies model uncertainty, and the induced policy propagates this uncertainty into the decision objective.
In addition, the Gibbs posterior depends on the data only through pointwise evaluations of the empirical decision risk.
This pointwise dependence further enables gradient-free MCMC samplers, extending DFL to nondifferentiable problems for which existing gradient-based and optimality-based methods cannot apply.

The remainder of the paper is organized as follows. Section~\ref{sec:2} introduces notation and the problem setting. Section~\ref{sec:3} presents the main components of BCO. Section~\ref{sec:4} studies its theoretical properties, including posterior concentration and nonasymptotic guarantees for our policy. Section~\ref{sec:5} develops approximate inference algorithms for the Gibbs posterior. Section~\ref{sec:6} reports numerical experiments on three OR/MS problems. Section~\ref{sec:7} concludes.
\section{Preliminaries}\label{sec:2}

\subsection{Contextual Stochastic Optimization}
Throughout this paper, let $\mathcal{X} \subseteq \R^{d_x}$ and $\mathcal{Y} \subseteq \R^{d_y}$ be Borel sets equipped with the Borel $\sigma$-algebras $\mathcal{B}(\mathcal{X})$ and $\mathcal{B}(\mathcal{Y})$ induced by their subspace topologies, and let $\mathcal{P}(\mathcal{Y})$ be the set of all probability measures on $(\mathcal{Y}, \mathcal{B}(\mathcal{Y}))$. The contextual information and uncertainty of interest are represented by random variables $X$ and $Y$, supported on $\mathcal{X}$ and $\mathcal{Y}$, respectively, with joint distribution $\P\in \mathcal{P}(\mathcal{X} \times \mathcal{Y})$. Since $\mathcal{X}$ and $\mathcal{Y}$ are standard Borel spaces, a regular conditional distribution $\P_{Y \mid X}$ exists, with $\P_{Y \mid X = x }\in \mathcal{P}(\mathcal{Y}) $ for $\P_X$-a.e. $x\in\mathcal{X}$; here $\P_X$ denotes the marginal distribution of $X$. Unless explicitly stated otherwise, every probability measure considered in this paper is assumed to have a finite first moment, and we use $\mathcal{W}_1(\mu_1, \mu_2)$ to denote the 1-Wasserstein distance between two probability measures $\mu_1$ and $\mu_2$.

The decision-maker observes the context $X=x$ prior to choosing a decision $z\in \mathcal{Z}$, where $\mathcal{Z} \subseteq \R^{d_z}$, while $Y$ is unobserved at decision time. For each $y\in\mathcal{Y}$, the objective function $c(\cdot;y): \mathcal{Z} \rightarrow \R$ is the cost incurred when $Y = y$. In the main analysis, we assume $Y$ only affects the objective, with $\mathcal{Z}$ fixed and independent of $Y$. In Section~\ref{sec:6_3}, we provide an empirical extension of our framework to problems with uncertainty in the constraints.

To leverage the statistical dependence between $X$ and $Y$ captured by $\P_{Y \mid X=x}$, CSO prescribes decisions by solving the following conditional stochastic optimization for each $x\in\mathcal{X}$~\citep{Bertsimas-2020-MS}:
\begin{equation}
    \pi^\star(x) \in \argmin_{z \in \mathcal{Z}} \,\mathbb{E}_{\P_{Y|X}}[c(z;Y) \mid X = x], \label{CSO}
\end{equation} 
where the expectation is taken with respect to the nominal conditional distribution $\P_{Y|X}$, hereafter referred to as the \textit{conditional}. Problem \eqref{CSO} prescribes the optimal decision for every realization $x\in\mathcal{X}$, yielding an optimal policy $\pi^\star:\mathcal{X}\mapsto \mathcal{Z}$.
We defer assumptions on the geometry of $\mathcal{X}, \, \mathcal{Y}$, and $\mathcal{Z}$ and the regularity of the function $c$ to Section~\ref{sec:4}.

\subsection{Data-Driven Frequentist Frameworks}\label{sec22}
In practice, the joint distribution $\P$ is unknown to the decision-maker, who instead has access to a logged dataset $\mathcal{D}_n = \lrc{ \lra{x_i, y_i} }^n_{i=1}$ of $n$ observations drawn from $\P$. The ultimate goal of data-driven CSO is to construct a data-driven policy $\pi$ that performs well out-of-sample using only $\mathcal{D}_n$. 
In this work, we adopt the parametric ETO approach. Compared to decision-rule optimization and nonparametric approaches, this approach ensures feasibility in constrained problems, offers flexibility for DFL, accommodates high-dimensional features, and integrates readily with Bayesian modeling. 

In parametric ETO, the decision-maker specifies a parametric model $m(\cdot;\theta):\mathcal{X}\mapsto \mathcal{P}(\mathcal{Y})$, drawn from a model class $\mathcal{M}\coloneqq \lrc{m(\cdot;\theta):\mathcal{X} \rightarrow \mathcal{P}(\mathcal{Y}) \mid \theta \in \Theta \subseteq \R^d}$, to approximate the conditional $\P_{Y|X}$. The codomain of $\mathcal{M}$ is defined as $\mathcal{P}(\mathcal{Y})$ rather than $\mathcal{Y}$ to include both point predictors, which produce Dirac measures on $\mathcal{B}(\mathcal{Y})$, and distributional estimators.

Given a parameter vector $\theta \in \Theta$, the frequentist ETO derives a policy for all $x$:
\begin{equation}
     \pi_\theta(x) \in \argmin_{z \in \mathcal{Z}} V(z, m(x;\theta)) = \argmin_{z \in \mathcal{Z}} \mathbb{E}_{m(x;\theta)}\lrb{ c(z;Y) },\label{freq:policy}
\end{equation}
where $V(z,m(x;\theta))$ is a function of the decision $z$ and the probability measure $m(x;\theta)$, specified as the expectation in this work. Note that although the original target of CSO is to minimize the expectation as defined in Problem~\eqref{CSO}, data-driven policies may incorporate additional regularization because of estimation errors. Throughout this work, all Bayesian and frequentist policies $\pi$ are assumed to apply a measurable selection, e.g., the lexicographical minimum, if multiple optimal decisions exist for the objective $V$.

The prescriptive performance of $\pi_\theta$ depends critically on the conditional estimate $m(x;\theta)$, and hence on the choice of $\theta$. Within ETO, two popular strategies have been studied for inferring $\theta$ from finite samples in a frequentist manner: maximum likelihood estimation (MLE) and integrated estimate-optimize (IEO). Both can be formulated as a regularized empirical risk minimization problem:
\begin{equation}
    \hat{\theta}_{M} \coloneqq \argmin_{\theta \in \Theta} \; \frac{1}{n}\sum_{i=1}^n \ell_M(y_i, m(x_i;\theta)) + \gamma \Vert \theta \Vert, \quad M \in \{\text{mle}, \text{ieo}\},
    \label{erm-g}
\end{equation}
where $\gamma \geq 0$ is the regularization strength, and the loss functions $\ell_M$ are defined as:
\begin{equation*}
    \begin{aligned}
        \ell_{\text{mle}}(y_i,m(x_i;\theta)) &= -\log f_Y(y_i \mid x_i,\theta), \\
        \ell_{\text{ieo}}(y_i,m(x_i;\theta)) &= -\log e^{-c(\pi_\theta(x_i);y_i)} = c\lra{\pi_\theta(x_i);y_i},
    \end{aligned}
\end{equation*}
where $f_Y(\cdot \mid x,\theta)$ denotes the density function of the parametric model $m(x;\theta)$. The MLE loss $\ell_{\text{mle}}$ measures the statistical fit of the model to the observed data, while the IEO loss $\ell_{\text{ieo}}$ directly measures the decision cost incurred by $\theta$ under the realized uncertainty $y_i$.
$\ell_\mathrm{ieo}$ is deliberately written in a negative-log form to mirror $\ell_\mathrm{mle}$: the realized decision cost can act as a log-density and will constitute our operational likelihood in Section~\ref{sec:31}.

While the IEO approach has been shown to be superior to MLE under model misspecification for generic optimization problems~\citep{Elmachtoub-2023-ARXIV, Elmachtoub-2025-AISTATS}, it inevitably introduces computational difficulties due to the bilevel structure of Problem~\eqref{erm-g}.~\cite{DFLreview} reviewed recent solution strategies for IEO.

\subsection{Bayesian Inference and Model Uncertainty}\label{sec23}

Both MLE and IEO produce a single point estimate of $\theta$, discarding uncertainty over the model parameter space.
We now recall the conventional Bayesian framework, which maintains a belief distribution over $\Theta$ and motivates the development of our BCO in Section~\ref{sec:3}. Within a prescribed parametric family $\mathcal M$, each parameter value $\theta$ indexes a candidate model. Hence, a probability measure over $\Theta$ induces a probability measure over the corresponding models in $\mathcal M$. In this sense, Bayesian parameter uncertainty provides a tractable within-family representation of model uncertainty if the model class $\mathcal{M}$ is sufficiently expressive, e.g., deep neural networks.

Formally, let $\mathcal P(\Theta)$ be the set of all probability measures on the measurable parameter space $(\Theta, \mathcal{B}(\Theta))$. Bayesian frameworks represent the decision-maker's belief on the unknown parameter through a prior measure $\rho_0\in\mathcal P(\Theta)$. Upon observing $\mathcal D_n$, this belief is updated via Bayes' rule:
\begin{equation}
        \rho_n(\text{d} \theta) = \frac{L(\mathcal{D}_n\mid \theta)\rho_0(\text{d}\theta)}{\int_{\Theta} L(\mathcal{D}_n\mid \nu)\rho_0(\text{d}\nu)} \propto L(\mathcal{D}_n\mid \theta)\rho_0(\text{d}\theta). \label{bayes}
\end{equation}
Given that $\mathcal{D}_n$ consists of i.i.d. observations and the parametric density $f_Y(\cdot \mid x,\theta)$ of model $m$, the likelihood $L(\mathcal{D}_n \mid \theta)$ decomposes as:
\begin{equation}
    L(\mathcal{D}_n \mid \theta) \coloneqq \prod_{i=1}^{n} f_Y(y_i \mid x_i,\theta).
    \label{likelihood}
\end{equation}
Conjugate priors with respect to the parametric densities are commonly adopted to derive a tractable closed-form posterior.
When no closed-form posterior exists, approximate Bayesian inference techniques may be employed; see~\cite{Chen-2020-OR} and references therein for applications in OR/MS.

\section{Bayesian Contextual Optimization Framework}\label{sec:3}

\subsection{Decision-Focused Belief Updating}\label{sec:31}
To address model misspecification in data-driven CSO while accounting for model uncertainty overlooked by frequentist approaches, we propose a decision-focused belief updating mechanism, which is a Bayesian generalization of the frequentist IEO approach.

We begin with the definition of the frequentist empirical decision risk associated with a policy $\pi_\theta$:
\begin{equation}
    R_n(\theta) \coloneqq \frac{1}{n} \sum_{i=1}^{n} c\lra{\pi_\theta(x_i);y_i}, \label{emprisk}
\end{equation}
and we denote by $R(\theta)$ its population risk. Throughout this belief-updating procedure, we adopt the policy $\pi_\theta$ defined in Equation~\eqref{freq:policy} which degenerates to a deterministic problem when $m(\cdot;\theta)$ is a Dirac measure, i.e., the model is a point predictor. 

Central to our framework is the notion of an \textit{operational likelihood}, which encodes the decision-maker's preference over candidate models without a parametric density. 
Following the rewriting of $\ell_{\mathrm{ieo}}$ in negative-log form (Section~\ref{sec22}), we define the operational likelihood as follows:
\begin{equation}
    L_\lambda(\mathcal{D}_n\mid \theta) \coloneqq 
    \prod_{i=1}^{n} e^{- \lambda c\lra{\pi_\theta(x_i);y_i}} = 
    \exp\lra{- \lambda n R_n(\theta)} \label{oplike}, 
\end{equation}
where $\lambda > 0$ is an inverse temperature parameter, also referred to as the \textit{learning rate} in Bayesian statistics, that controls the relative weight placed on the information carried by the data versus the prior. A larger $\lambda$ sharpens the likelihood around parameters with lower empirical decision risk. 
We justify a theoretical order $\lambda = C_\lambda n^{-0.5}$ for some constant $C_\lambda > 0$ in Theorem~\ref{thm:2}, and tune $C_\lambda$ for the practical performance of our policy derived in Section~\ref{sec:32}; see~\cite{Wu-2023-BA} for alternative $\lambda$-selection strategies for different purposes.

Replacing the parametric likelihood $L(\cdot \mid \theta)$ in Equation~\eqref{bayes} with the operational likelihood $L_\lambda(\cdot\mid \theta)$ in Equation~\eqref{oplike}, we derive the following decision-focused belief updating mechanism: 
\begin{equation}
        \Gibbs(\text{d} \theta) \coloneqq \frac{L_\lambda(\mathcal{D}_n\mid \theta)\rho_0(\text{d}\theta)}{\int_{\Theta} L_\lambda(\mathcal{D}_n\mid \nu)\rho_0(\text{d}\nu)} \propto L_\lambda(\mathcal{D}_n\mid \theta)\rho_0(\text{d}\theta). \label{da-bayes}
\end{equation}
Although $\Gibbs$ depends on $\lambda,\,\mathcal{D}_n$, and $\pi_\theta$, we suppress this dependence in the notation. This problem-dependent posterior is also known as the \textit{Gibbs} posterior. By construction, $\Gibbs$ assigns greater posterior density to parameter vectors that yield lower empirical decision risk. Parameters with higher (but finite) decision loss still receive nonzero weight, encoding the remaining model uncertainty. In the following Example~\ref{examp:newsvendor}, we demonstrate how our belief updating mechanism for a newsvendor problem recovers the behavior of Bayesian quantile regression using the asymmetric Laplace distribution.

\begin{example}\label{examp:newsvendor}
Consider a newsvendor problem $c(z;y) = c_o(z - y)^+ + c_u (y - z)^+$ with unit overage cost $c_o > 0$ and underage cost $c_u > 0$. Define the critical fractile $\tau = c_u / (c_u + c_o)$ and the pinball loss $\ell_\tau(y) = y(\tau - \mathbb{I}\{y < 0\})$, so that $c(z;y) = (c_u + c_o)\ell_\tau(y - z)$. Given a point predictor $f_\theta: \mathcal{X} \rightarrow \mathcal{Y}$, the Gibbs posterior becomes:
\begin{equation*}
    \Gibbs(\theta)\propto \exp\!\left(-\lambda(c_u+c_o)\sum_{i=1}^n \ell_\tau(y_i - f_\theta(x_i))\right)\rho_0(\theta).
\end{equation*}
This coincides exactly with the posterior arising from Bayesian quantile regression~\citep{Yu-2001-SPL} under an asymmetric Laplace likelihood. In particular, the Bayesian posterior under this likelihood is:
\begin{equation*}
    \rho_n^{\mathrm{AL}}(\theta)
    \;\propto\; \exp\!\left(-\frac{1}{\sigma}\sum_{i=1}^n \ell_\tau(y_i - f_\theta(x_i))\right)\rho_0(\theta).
\end{equation*}
The two posteriors $\Gibbs$ and $\rho_n^{\mathrm{AL}}$ coincide under the specification $\sigma = 1/(\lambda(c_u+c_o))$. The consistency of $\rho_n^{\mathrm{AL}}$ under misspecification (the true density is not asymmetric Laplace) has been established in~\cite{Sriram-2013-BA}.
\end{example}

The operational likelihood mirrors the MLE-to-IEO transition from the frequentist setting into the Bayesian setting: Equation \eqref{oplike} replaces a parametric density with empirical decision risk as the criterion for belief updating. 
This is particularly valuable in data-driven settings with multivariate uncertainty for which specifying a parametric density that faithfully represents the data-generating process is nontrivial. The theoretical properties of the decision-focused posterior $\Gibbs$ are established in Section~\ref{sec:4}.

\subsection{Bayesian Contextual Policy}\label{sec:32}
Our goal is not to infer the posterior itself but to translate it into a data-driven policy. Grounded in Bayes decision theory, the proposed Bayesian contextual policy minimizes expected decision loss under the posterior predictive distribution (PPD); see definition below.

\begin{definition}[Posterior Predictive Distribution, PPD]\label{def:1}
Let $m(\cdot;\theta):\mathcal X \to \mathcal P(\mathcal Y)$ be a predictive model that generates a probability measure for each context $x\in\mathcal X$ and parameter $\theta\in\Theta$. Let $\rho\in\mathcal P(\Theta)$ be a posterior measure. The PPD $m(\cdot;\rho):\mathcal X\to\mathcal P(\mathcal Y)$ is defined for any $x\in\mathcal X$ by:
\[
m(x;\rho)(\mathcal{A})
\coloneqq
\int_\Theta m(x;\theta)(\mathcal{A})\,\rho(\mathrm{d}\theta),
\quad \forall \mathcal{A} \in \mathcal B(\mathcal Y).
\]
That is, $m(x;\rho)$ is the mixture of measures $\{m(x;\theta)\}_{\theta\sim\rho}$ weighted by $\rho$. In the point predictor case where $m(x;\theta)=\delta(f_\theta(x))$ is a Dirac measure for some measurable $f_\theta:\mathcal{X} \rightarrow \mathcal{Y}$, the PPD corresponds to the pushforward of $\rho$ under the map $\theta \mapsto f_\theta(x)$.
\end{definition}

In this work, we focus on the case where the model $m(\cdot;\theta)$ is a point predictor for three reasons. First, point predictors are likelihood-free: no parametric family for the law of $Y|X$ must be specified. Second, they are more flexible than discrete-support modeling for high-dimensional support $\mathcal{Y}$. 
Most importantly, unlike frequentist point predictors, the PPD $m(x;\rho)$ is inherently a mixture measure on $\mathcal B(\mathcal Y)$, and is therefore compatible with risk-averse optimization models even when each $m(\cdot;\theta)$ is a point predictor. For example, if a linear model $m(x;\theta) = \theta x$ is specified and the posterior $\rho = \mathcal{N}(\mu, \sigma^2)$ is Gaussian, then $m(x;\rho) = \mathcal{N}(x\mu, x^2 \sigma^2)$.

\begin{definition}[Bayesian Contextual Policy, BCP]\label{def:2}
Let $\rho \in \mathcal{P}(\Theta)$ be a posterior measure and $m(x;\rho)$ be the posterior predictive distribution given by Definition~\ref{def:1}. The Bayesian contextual policy $\pi_\rho$ is defined for any $x \in \mathcal{X}$ as:
\begin{equation*}
    \pi_\rho(x) \in \argmin_{z \in \mathcal{Z}} \mathbb{E}_{m(x;\rho)}[c(z;Y)] 
    = \argmin_{z \in \mathcal{Z}} \int_{\mathcal{Y}} c(z;y) m(x;\rho)(\mathrm{d}y).
\end{equation*}
\end{definition}

BCP $\pi_\rho$ shares the same feasible region $\mathcal Z$ and objective structure as Problem~\eqref{CSO}, differing only in the distribution under which the expectation is taken: the unknown $\mathbb{P}_{Y|X=x}$ is replaced by the PPD $m(x;\rho)$. This is a deliberate design: by adding no regularization, the policy preserves the computational tractability of the original stochastic problem, so standard techniques, e.g., sample average approximation (SAA), can directly apply.

Our framework exhibits a \textit{separate-learning-aggregate-optimization} structure. During belief updating, each parameter $\theta$ participates individually to evaluate its operational likelihood $L_\lambda(\mathcal{D}_n \mid \theta)$. At the prescribing stage, the policy $\pi_\rho$ aggregates the individual predictions $m(x;\theta)$ into PPD $m(x;\rho)$, thereby propagating model uncertainty into the decision objective.
Figure~\ref{fig:pipeline-v3} contrasts the gradient-based IEO and BCO pipelines across their training and prescribing stages, highlighting how each handles model uncertainty and incorporates optimization feedback.

Beyond $\pi_\rho$, Bayes decision theory also admits policies that serve different purposes. To illustrate what distinguishes our policy from alternatives in handling model uncertainty, we introduce the Thompson policy, i.e., posterior sampling, and the maximum a posteriori (MAP) policy~\citep{Murphy-2012-book}.

\begin{definition}[Thompson and MAP policies]\label{def:3}
Let $\rho \in \mathcal{P}(\Theta)$ be a posterior measure and $m(x;\rho)$ be a posterior predictive distribution. The Thompson policy $\pi^{\mathrm{T}}_\rho$ and MAP policy $\pi^{\mathrm{M}}_\rho$ are defined for any $x \in \mathcal{X}$ as:
\begin{align*}
    \pi^{\mathrm{T}}_\rho(x)  &\in \argmin_{z \in \mathcal{Z}} \mathbb{E}_{m(x;\tilde{\theta})}[c(z;Y)], \quad \tilde{\theta} \sim \rho, \\
    \pi^{\mathrm{M}}_\rho(x) &\in \argmin_{z \in \mathcal{Z}} \mathbb{E}_{m(x;\theta^\dagger)}[c(z;Y)],  \quad \theta^\dagger \in \argmax_{\theta \in \Theta}  \rho(\theta),
\end{align*}
where $\tilde{\theta} \sim \rho$ denotes a random sample $\tilde{\theta}$ drawn from the posterior $\rho$, and $\argmax_{\theta \in \Theta} \rho(\theta)$ is the set of posterior modes.
\end{definition}

Both the Thompson and MAP policies discard the PPD and therefore ignore model uncertainty at the decision stage. The Thompson policy draws a single $\tilde{\theta} \sim \rho$ to construct the predictive distribution for each $x$, making it an inherently stochastic policy suited to balancing exploration and exploitation in online environments~\citep{Agrawal-2025-MOR}. 
The MAP policy uses a posterior mode as the parameter estimate, yielding a deterministic policy.
In contrast, $\pi_\rho$ integrates predictions over the entire posterior, providing a deterministic and more robust prescription when model uncertainty is substantial. The relative prescriptive performance of $\pi_\rho,\pi^{\text{T}}_\rho$, and $\pi^{\text{M}}_\rho$ derived from the same Gibbs posterior is examined in the numerical experiments.

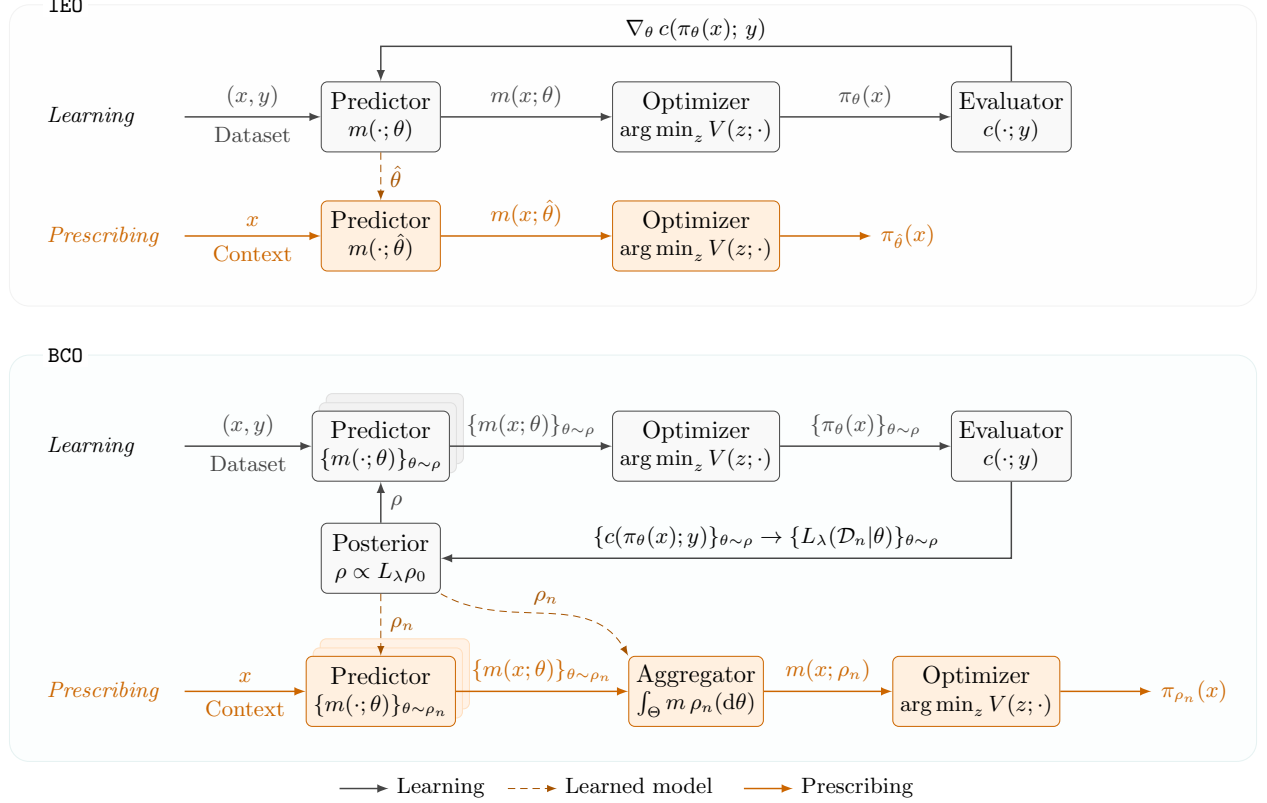
\begin{figure}[ht]
\centering
\caption{Learning and prescribing pipelines of gradient-based IEO (upper panel) and BCO (lower panel).}
\resizebox{\textwidth}{!}{%
\begin{tikzpicture}[
    >=latex,
    sbox/.style={draw, rounded corners=2.5pt, minimum height=10mm, minimum width=17mm,
                inner sep=3pt, font=\small, align=center},
    pbox/.style={sbox, fill=orange!12, draw=orange!80!black},
    lbox/.style={sbox, , fill=gray!5, draw=gray!55!black},
    bcobox/.style={sbox, fill=teal!10},
    stackshadow/.style={sbox, fill=gray!10, draw=gray!25, text opacity=0},
    stackshadowinfer/.style={sbox, fill=orange!10, draw=orange!25, text opacity=0},
    parr/.style={->, semithick, orange!80!black},
    larr/.style={->, semithick, gray!55!black},
    fbarr/.style={->, semithick, gray!55!black},
    trnarr/.style={->, thin, densely dashed, orange!65!black},
    plbl/.style={font=\footnotesize, text=orange!80!black},
    llbl/.style={font=\footnotesize},
    fblbl/.style={font=\footnotesize},
    trnlbl/.style={font=\footnotesize},
    seclbl/.style={font=\footnotesize\bfseries},
    rowlbl/.style={font=\footnotesize\itshape},
]

\fill[gray!1, rounded corners=8pt] (-0.3, 1.8) rectangle (17.5, -2.5);
\draw[gray!10, rounded corners=8pt] (-0.3, 1.8) rectangle (17.5, -2.5);
\node[seclbl, fill=white, inner sep=2pt] at (0.5, 1.8) {\texttt{IEO}};

\node[rowlbl, anchor=west, text=orange!80!black] at (0.1, -1.5) {Prescribing};
\node[pbox] (pp1) at (5, -1.5) {Predictor\\[-1pt]{\footnotesize $m(\cdot;\hat\theta)$}};
\node[pbox] (op1) at (9.5, -1.5) {Optimizer\\[-1pt]{\footnotesize $\argmin_z V(z;\cdot)$}};
\draw[parr] (2.2, -1.5) -- (pp1) node[midway, above, plbl] {\footnotesize $x$} node[midway, below, plbl] {Context};
\draw[parr] (pp1) -- (op1) node[midway, above, plbl] {$m(x;\hat\theta)$};
\draw[parr] (op1.east) -- ++(1.3,0) node[at end, right, plbl] {$\pi_{\hat\theta}(x)$};

\node[rowlbl, anchor=west] at (0.1, 0.2) {Learning};
\node[lbox] (pl1) at (5, 0.2) {Predictor\\[-1pt]{\footnotesize $m(\cdot;\theta)$}};
\node[lbox] (ol1) at (9.5, 0.2) {Optimizer\\[-1pt]{\footnotesize $\argmin_z V(z;\cdot)$}};
\node[lbox] (el1) at (14, 0.2) {Evaluator\\[-1pt]{\footnotesize $c(\cdot;y)$}};
\draw[larr] (2.2, 0.2) -- (pl1) node[midway, above, llbl] {\footnotesize $(x,y)$} node[midway, below, llbl] {Dataset};
\draw[larr] (pl1) -- (ol1) node[midway, above, llbl] {$m(x;\theta)$};
\draw[larr] (ol1) -- (el1) node[midway, above, llbl] {$\pi_\theta(x)$};

\draw[trnarr] (pl1.south) -- (pp1.north) node[midway, right, trnlbl] {$\hat\theta$};

\draw[fbarr] (el1.north) -- ++(0,0.5) -| (pl1.north);
\node[fblbl] at (9.5, 1.45) {$\nabla_\theta\,c(\pi_\theta(x);\,y)$};

\def\by{-5.5}
\fill[teal!1, rounded corners=8pt] (-0.3, \by+2.3) rectangle (17.5, \by-3.5);
\draw[teal!10, rounded corners=8pt] (-0.3, \by+2.3) rectangle (17.5, \by-3.5);
\node[seclbl, fill=white, inner sep=2pt] at (0.5, \by+2.3) {\texttt{BCO}};

\node[rowlbl, anchor=west, text=orange!80!black] at (0.1, \by-2.5) {Prescribing};
\node[stackshadowinfer] at (5.2, \by-2.25) {Predictor\\[-1pt]{\footnotesize $\{m(\cdot;\theta)\}_{\theta\sim\Gibbs}$}};
\node[stackshadowinfer] at (5.1, \by-2.38) {Predictor\\[-1pt]{\footnotesize $\{m(\cdot;\theta)\}_{\theta\sim\Gibbs}$}};
\node[pbox] (pp2) at (5, \by-2.5) {Predictor\\[-1pt]{\footnotesize $\{m(\cdot;\theta)\}_{\theta\sim\Gibbs}$}};
\node[pbox] (ag2) at (9.5, \by-2.5) {Aggregator\\[-1pt]{\footnotesize $\int_\Theta m\,\Gibbs(\mathrm{d}\theta)$}};
\node[pbox] (op2) at (13.5, \by-2.5) {Optimizer\\[-1pt]{\footnotesize $\argmin_z V(z;\cdot)$}};
\draw[parr] (2.2, \by-2.5) -- (pp2) node[midway, above, plbl] {\footnotesize $x$} node[midway, below, plbl] {Context};
\draw[parr] (pp2) -- (ag2) node[midway, above, plbl] {$\{m(x;\theta)\}_{\theta\sim\Gibbs}$};
\draw[parr] (ag2) -- (op2) node[midway, above, plbl] {$m(x;\Gibbs)$};
\draw[parr] (op2.east) -- ++(1.3,0) node[at end, right, plbl] {$\pi_{\Gibbs}(x)$};

\node[rowlbl, anchor=west] at (0.1, \by+1.0) {Learning};
\node[stackshadow] at (5.2, \by+1.25) {Predictor\\[-1pt]{\footnotesize $\{m(\cdot;\theta)\}_{\theta\sim\rho}$}};
\node[stackshadow] at (5.1, \by+1.12) {Predictor\\[-1pt]{\footnotesize $\{m(\cdot;\theta)\}_{\theta\sim\rho}$}};
\node[lbox] (pl2) at (5, \by+1.0) {Predictor\\[-1pt]{\footnotesize $\{m(\cdot;\theta)\}_{\theta\sim\rho}$}};
\node[lbox] (ol2) at (9.5, \by+1.0) {Optimizer\\[-1pt]{\footnotesize $\argmin_z V(z;\cdot)$}};
\node[lbox] (el2) at (14, \by+1.0) {Evaluator\\[-1pt]{\footnotesize $c(\cdot;y)$}};
\draw[larr] (2.2, \by+1.0) -- (pl2) node[midway, above, llbl] {\footnotesize $(x,y)$} node[midway, below, llbl] {Dataset};
\draw[larr] (pl2) -- (ol2) node[midway, above, llbl] {$\{m(x;\theta)\}_{\theta\sim\rho}$};
\draw[larr] (ol2) -- (el2) node[midway, above, llbl] {$\{\pi_\theta(x)\}_{\theta\sim\rho}$};

\node[lbox] (po2) at (5, \by-0.6) {Posterior\\[-1pt]{\footnotesize $\rho \propto L_\lambda \rho_0$}};

\draw[trnarr] (po2.south) -- (pp2.north) node[midway, right, trnlbl] {$ \Gibbs$};

\draw[trnarr] (po2.south east) to [out=-30, in=120] node[midway, sloped, above, trnlbl] {$ \Gibbs$} (ag2.north west);

\draw[fbarr] (el2.south) |- (po2.east);
\node[fblbl] at (10.5, \by-0.3) {$\{c(\pi_\theta(x);y)\}_{\theta\sim\rho} \rightarrow \{L_\lambda(\mathcal{D}_n|\theta)\}_{\theta\sim\rho}$};
\draw[fbarr] (po2.north) -- (pl2.south) node[midway,right, fblbl] {$\rho$};

\node[anchor=north, font=\footnotesize] at (8.5, \by-3.6) {%
  \tikz[baseline=-0.5ex]{\draw[->, semithick, gray!55!black] (0,0) -- (0.7,0);}~Learning\quad
  \tikz[baseline=-0.5ex]{\draw[->, thin, densely dashed, orange!65!black] (0,0) -- (0.7,0);}~Learned model \quad
  \tikz[baseline=-0.5ex]{\draw[->, semithick, orange!80!black] (0,0) -- (0.7,0);}~Prescribing
};

\end{tikzpicture}%
}

\begin{minipage}{\textwidth}
\vskip 0.2cm
\small
\textit{Notes}. \textbf{IEO}: during learning, labeled data $(x,y)$ traverse the full pipeline through the Evaluator, and the gradient $\nabla_\theta c(\pi_\theta(x);y)$ feeds back to update~$\theta$; the resulting $\hat\theta$ is then transferred (dashed orange) to the Predictor for prescription. At prescribing time, the trained model $m(\cdot;\hat\theta)$ maps context~$x$ to a prediction, and the Optimizer yields decision $\pi_{\hat\theta}(x)$. \textbf{BCO}: during learning, an ensemble of predictors (stacked Predictor blocks, shaded) is maintained; each individual prediction is optimized directly, and the resulting costs inform the operational likelihood $L_\lambda(\mathcal{D}_n\mid\theta)$, which updates the posterior~$\rho$ used to generate the ensemble. The learned posterior~$\Gibbs$ is then deployed at prescribing time, where ensemble predictions are aggregated into the posterior predictive distribution $m(x;\Gibbs)=\int_\Theta m(x;\theta)\,\Gibbs(\mathrm{d}\theta)$ before optimization, yielding a single decision~$\pi_{\Gibbs}(x)$.
\end{minipage}
\label{fig:pipeline-v3}
\end{figure}

\subsection{Connections with Bayesian Statistics and PAC-Bayes Learning}\label{sec:33}
We now situate our decision-focused belief updating mechanism within the broader landscape of Bayesian statistics and statistical learning, motivating its theoretical foundations through two complementary perspectives: generalized Bayesian updating~\citep{Bissiri-2016-JRSSB} and probably approximately correct (PAC)-Bayes learning~\citep{Alquier-2024-BOOK}. 

Under standard regularity, the Gibbs posterior $\rho_n$ in Equation~\eqref{da-bayes} is the unique minimizer over $\mathcal{P}(\Theta)$ of the KL divergence-regularized risk functional~\citep{Bissiri-2016-JRSSB}:
\begin{equation}
    \inf_{\rho \in \mathcal{P}(\Theta) }\mathcal{J}(\rho) = \E_{\rho}\lrb{R_n(\theta)} + \frac{1}{\lambda n }\KL{\rho}{\rho_0}, \label{bissiri_eq}
\end{equation}
and is also the unique minimizer of a type of PAC-Bayes upper bound on the population risk $\E_{\rho}\lrb{R(\theta)}$; see Lemma~\ref{lem:1} and~\ref{lem:2} for formal statements.

\cite{Bissiri-2016-JRSSB} arrived at the Gibbs measure from a decision-theoretic axiomatization of coherent belief updating in the presence of a loss function; the loss is treated as a primitive of the inference problem. We arrive at the same posterior by replacing the parametric likelihood in classical Bayesian inference with the operational likelihood induced by the downstream decision cost, that is, as a direct Bayesian generalization of frequentist IEO. The two derivations yield the same expression but support different modeling intuitions: the former justifies the Gibbs posterior as a coherent update; we motivate it as the natural Bayesian counterpart of a decision-focused estimator. More importantly, both the generalized Bayesian updating and PAC-Bayes results bound the risk of a randomized predictor $\theta \sim \rho$, i.e., the Thompson policy $\pi_\rho^\mathrm{T}$. 
Neither addresses the risk of the aggregate policy $\pi_\rho$ that we propose.
Closing this gap, and quantifying when $\pi_\rho$ strictly outperforms its IEO counterpart $\pi_{\theta}$, requires a separate analysis which we provide in Section~\ref{sec:42}.

\section{Theoretical Analysis}\label{sec:4}

\subsection{Posterior Concentration} \label{sec:41}
Posterior concentration is a crucial property for Bayesian frameworks to guarantee the asymptotic consistency of the belief updating as data accumulates. If $\Theta^\star$ denotes the set of best-in-class parameters under a certain criterion, posterior concentration can be expressed as:
\begin{equation}
    \forall \varepsilon > 0, \quad \rho_n\lra{ \lrc{\theta: \varphi(\theta, \Theta^\star) \geq \varepsilon} } \stackrel{n \rightarrow \infty}{\longrightarrow} 0, \label{concentration}
\end{equation}
where $\varphi(\theta,\Theta^\star)$ denotes the divergence between an arbitrary $\theta$ and the set $\Theta^\star$. 
Different choices of $\Theta^\star$ and $\varphi$ lead to substantially different targets and notions of consistency for the Bayesian update.
Our framework contrasts with existing work by considering the best-in-class decision risk minimizer as the learning target to preserve the estimation-optimization consistency. To establish the concentration result, we first define the decision risk for the frequentist policy as follows:
\begin{equation}
    R(\theta) \coloneqq \E_\P[c(\pi_\theta(X);Y)], \quad R^\star \coloneqq \inf_{\theta \in \Theta} R(\theta), 
\end{equation}
whose empirical counterpart $R_n$ was previously defined in Equation~\eqref{emprisk}.
Accordingly, the optimal parameter vectors and the distance function $\varphi$ are defined as:
\begin{equation*}
    \theta^\star\in\Theta^\star \coloneqq \{\theta \mid R(\theta) = R^\star \}, 
    \quad \varphi(\theta,\Theta^\star) \coloneqq R(\theta) - R^\star.
\end{equation*}

With these quantities defined, the following technical assumptions are leveraged to establish the concentration of the Gibbs posterior defined in Equation~\eqref{da-bayes}.
\begin{assumption}[I.I.D. Bounded Loss]\label{A1}
The dataset $\mathcal{D}_n$ consists of i.i.d. observations drawn from a distribution $\P$. In addition, the cost function $0 \leq c(z;y) \leq 1$ is bounded uniformly over $\mathcal{Z} \times \mathcal{Y}$.
\end{assumption}

\begin{assumption}[Regularity]\label{A2}
(i) The parameter space $\Theta \subset \R^d$ is compact with finite diameter $D \coloneqq \max_{\theta,\nu \in \Theta}\norm{\theta - \nu} < \infty$, and $\Theta^\star \cap \mathrm{int}(\Theta) \neq \emptyset$. Fix a point $\theta^\star \in \Theta^\star \cap \mathrm{int}(\Theta)$ and let $r_0 \coloneqq \mathrm{dist}(\theta^\star, \partial \Theta) > 0$. (ii) The prior $\rho_0$ admits a Lebesgue density on $\Theta$ bounded below by $l_0 >0$.
\end{assumption}

\begin{assumption}[Model Complexity]\label{A3}
Let $\mathcal{F}_\Theta \coloneqq \lrc{ (x,y) \mapsto c(\pi_\theta(x);y):\theta \in \Theta }$. There exist constants $d_\mathcal{F} \geq 1$ and $A_\mathcal{F} \geq e$ such that for every probability measure $\mathbb{Q}$ on $\mathcal{X} \times \mathcal{Y}$ and every $r \in (0,1]$,
\begin{equation*}
    \log \mathcal{N}\lra{ r, \mathcal{F}_\Theta, L_2(\mathbb{Q}) } \leq d_\mathcal{F} \log \frac{A_\mathcal{F}}{r},
\end{equation*}
where $\mathcal{N}\lra{ r, \mathcal{F}_\Theta, L_2(\mathbb{Q}) }$ is the $r$-covering number of $\mathcal{F}_\Theta$ in $L_2(\Q)$.
\end{assumption}

\begin{assumption}[Local Risk Stability]\label{A4}
    Let $\theta^\star$ be the point fixed in Assumption~\ref{A2}(i). There exist $\bar{r} > 0$ and a non-decreasing function $\psi: [0, \bar{r}] \rightarrow \R_+$ with $\psi(r) \downarrow 0$ as $r \downarrow 0$ such that
    \begin{equation*}
        R(\theta) - R^\star \leq \psi \lra{\norm{\theta - \theta^\star}}, \quad \forall \, \theta \in \mathbb{B}_{\bar{r}}(\theta^\star) \cap \Theta,
    \end{equation*}
    where $\mathbb{B}_{\bar{r}}(\theta^\star) = \lrc{\theta: \norm{\theta - \theta^\star} \leq \bar{r}}$.
\end{assumption}
Assumptions~\ref{A1}--\ref{A4} separate the statistical and geometric ingredients needed for posterior concentration. Assumption~\ref{A1} ensures that the empirical decision risk \(R_n(\theta)\) concentrates uniformly around its population counterpart \(R(\theta)\); Assumption~\ref{A2} guarantees that the prior \(\rho_0\) assigns non-negligible mass to neighborhoods of interior best-in-class parameters, so the Gibbs posterior has sufficient denominator mass near the optimum; Assumption~\ref{A3} controls the complexity of the induced loss class $\mathcal F_\Theta$ through a uniform \(L_2\)-metric entropy bound; Assumption~\ref{A4} imposes only local stability of the population decision risk near an optimal parameter; it is not a global smoothness requirement. We state Assumptions~\ref{A3}--\ref{A4} in generic forms for the sake of Theorem~\ref{thm:1} and verify them for canonical OR/MS problems in Corollaries~\ref{cor:1} and~\ref{cor:2}.

\begin{theorem}[Posterior Concentration]\label{thm:1}
    Suppose Assumptions~\ref{A1}--\ref{A4} hold. Let $\Gibbs$ be the Gibbs posterior defined in Equation~\eqref{da-bayes} with fixed $\lambda > 0$ and $\mathcal{D}_n$. For any $\varepsilon > 0$, define
    \begin{equation*}
        B_\varepsilon \coloneqq \{ \theta \in \Theta: R(\theta) - R^\star \geq \varepsilon \}.
    \end{equation*}
    By Assumption~\ref{A4}, the set $\{r\in (0, \bar{r}]: \psi(r) \leq \varepsilon / 6\}$ is nonempty; fix any $s_\varepsilon$ in this set and let  $r_\varepsilon\coloneqq \min\{r_0/2, s_\varepsilon\}$. 
    
    There exist universal constants $C_0,  C_1 >0$ such that, for any $\eta \in (0,1)$ and any sample size satisfying
    \begin{equation*}
        n \geq \frac{C_0}{\varepsilon^2}\lra{d_\mathcal{F} \log (C_1 A_\mathcal{F}) + \log\frac{1}{\eta}},
    \end{equation*}
    with probability at least $1-\eta$ over the draw of $\mathcal{D}_n$,
    \begin{equation*}
        \Gibbs(B_\varepsilon) \leq \frac{\exp\lra{-\lambda n \varepsilon / 2}}{l_0 \mathbb{V}_d r_\varepsilon^d} ,
    \end{equation*}
    where $\mathbb{V}_d$ is the Lebesgue volume of the Euclidean unit ball in $\R^d$.
\end{theorem}
\proof{Proof of Theorem~\ref{thm:1}} The  proof of Theorem~\ref{thm:1} is provided in the e-companion to this paper. \halmos\endproof

Theorem~\ref{thm:1} establishes a nonasymptotic concentration result for the Gibbs posterior, demonstrating that it contracts around the target set at an exponential rate in the sample size $n$, under the generic conditions specified in Assumptions~\ref{A3}--\ref{A4}. These conditions constrain the complexity of the policy-induced loss function class and the local geometric properties of the associated risk function. 

Assumptions~\ref{A3}--\ref{A4} can be verified for canonical OR/MS problems with specific structures. In this work, we instantiate Theorem~\ref{thm:1} for two important problem classes, summarized in Condition~\ref{cond:1} and Condition~\ref{cond:2}, distinguished by the smoothness of the estimate-then-optimize pipeline. To keep this verification tractable and align with practice, we adopt two further practical assumptions, on the predictive model and on the cost function, respectively. Assumption~\ref{A5} restricts attention to regular point predictors, a parametric class broad enough to cover bounded-weight neural networks. Assumption~\ref{A6} imposes minimal regularity on the cost function.

\begin{assumption}[Regular Point Predictor]\label{A5}
(i) The parametric model $m(\cdot;\theta)$ is a point predictor, i.e., $m(x;\theta) = \delta(f_\theta(x))$
for a measurable function $f_\theta: \mathcal{X} \to \mathcal{Y}$ that is $L_f$-Lipschitz in $\theta$ uniformly over $\mathcal{X}$.
(ii) For every fixed affine functional $y \mapsto a^\top y + b$, the scalar class $\mathcal{S} = \lrc{ x \mapsto a^\top f_\theta(x) + b : \theta \in \Theta}$ has pseudo-dimension at most $d_\mathcal{S} < \infty$. 
\end{assumption}

\begin{assumption}[Lipschitz Cost in Uncertainty]\label{A6}
There exists $L_y < \infty$ such that 
\begin{equation*}
    \abs{ c(z;y) - c(z;y')} \leq L_y \norm{y - y'}, \quad \forall \, z \in \mathcal{Z}, \; \forall \, y, y' \in \mathcal{Y}.
\end{equation*}
\end{assumption}

We first study the nonsmooth problems described in Condition~\ref{cond:1} where the optimal decision may jump among finitely many feasible decisions. This class covers many instances of linear programs (LPs), mixed-integer linear programs (MILPs), and two-stage linear programs (TSLPs).
\begin{condition}[Nonsmooth Finite Problems]\label{cond:1}
The downstream problem has the following properties:
\begin{enumerate}[label=\textup{(\roman*)}]
    \item \textup{(Finite Effective Decisions)} There is a finite set $\mathcal{V} = \{z_1, \dots, z_K\} \subseteq \mathcal{Z}$ with $\abs{\mathcal{V}} \geq 2$ such that for every $\theta \in \Theta$ and $x \in \mathcal{X}$, the policy output satisfies $\pi_\theta(x) = z_k, \; k = \zeta_\theta(x)$ for a multiclass map $\zeta_\theta:\mathcal{X} \rightarrow \{1,\dots,K\}$ obtained by minimizing $c(z_k;f_\theta(x))$ over $k$ with a fixed selection rule. Also, $\max_{z \in \mathcal{V}} \sup_{y \in \mathcal{Y}} c(z;y) \leq 1$.
    \item \textup{(Piecewise-Affine Pairwise Comparisons)} For every pair $j \neq k$, the region
    \begin{equation*}
        \lrc{
        y \in \mathcal{Y} : c(z_k;y) \leq c(z_j;y)
        }    
    \end{equation*}
    can be represented by a Boolean formula involving at most $M$ affine inequalities in $y$ of the form $a^\top y + b \leq 0$.
    \item \textup{(Nondegeneracy)} For the interior $\theta^\star$ in Assumption~\ref{A2}, let $z^\star(x) = \pi_{\theta^\star}(x)$, and
    \begin{equation*}
        \Delta^\star(x) = \min_{z \in \mathcal{V} \setminus \{z^\star(x)\} } \lrb{ c(z;f_{\theta^\star}(x)) - c(z^\star(x);f_{\theta^\star}(x)) }.
    \end{equation*}
    There exist constants $C_\Delta < \infty$ and $\beta > 0$ such that
    \begin{equation*}
        \P_X\lra{ 0 < \Delta^\star(X) \leq t } \leq C_\Delta t^\beta, \quad \forall t > 0,
    \end{equation*}
    and $\P_X(\Delta^\star(X) = 0) = 0$.
\end{enumerate}
\end{condition}
\begin{remark}
    Condition~\ref{cond:1} is met by many instances of LPs, MILPs, and TSLPs. Specifically, Condition~\ref{cond:1}(i)-(ii) covers linear objective coefficients when $M=1$, bounded LP or MILP deterministic problems with finite extreme-point representations, and two-stage linear recourse functions on regions where the relevant finite basis or dual-extreme-point representation applies. Condition~\ref{cond:1}(iii) parallels the noise condition in~\citep[Assumption 2]{Hu-2022-MS} and rules out pathological cases.
\end{remark}

The following Corollary~\ref{cor:1} verifies generic Assumptions~\ref{A3}--\ref{A4} for problems satisfying Condition~\ref{cond:1} using practical Assumptions~\ref{A5}--\ref{A6}.
\begin{corollary}[Posterior Concentration of Nonsmooth Pipelines]\label{cor:1}
Suppose Assumptions~\ref{A1}--\ref{A2},~\ref{A5}--\ref{A6} hold, and Condition~\ref{cond:1} holds. Then Assumptions~\ref{A3}--\ref{A4} hold with 
\begin{equation*}
    d_\mathcal{F} \leq K^2 M d_\mathcal{S}, \quad A_\mathcal{F} = C_{c1} K\sqrt{M}, \quad \psi(r) = C_\Delta(2 L_y L_f r)^\beta, \quad \bar{r} = D,
\end{equation*}
for a universal constant $C_{c1}>0$, and Theorem~\ref{thm:1} applies accordingly.
\end{corollary}
\proof{Proof of Corollary~\ref{cor:1}} The proof of Corollary~\ref{cor:1} is provided in the e-companion to this paper. \halmos\endproof
\begin{remark}
    The pseudo-dimension $d_\mathcal{S}$ in Corollary~\ref{cor:1} is finite for many standard predictor classes; we record explicit $L_f$ and $d_\mathcal{S}$ for three of them. In each case $\varphi$ is a fixed feature map with $\sup_x \norm{\varphi(x)} \leq B_\varphi < \infty$. (i) \textbf{Affine or fixed-basis models}: for $f_\theta(x) = W\varphi(x) + b$ with $\theta = (\mathrm{vec}(W),b) \in \R^d$, the map $\theta \mapsto f_\theta(x)$ is affine, so $L_f \leq \sqrt{B^2_\varphi + 1}$. The scalar class $\mathcal{S}$ lies in a $(d+1)$-dimensional vector space of functions, hence $d_\mathcal{S} \leq d + 1$~\citep[Theorem 11.4]{anthony-1999-book}.
    (ii) \textbf{Generalized linear models}: for $f_\theta(x) = g(W\varphi(x) + b)$ with a fixed $L_g$-Lipschitz link $g$, the same argument gives $L_f\leq L_g \sqrt{B_\varphi^2 + 1}$. When $g$ is monotone, the thresholds defining $\mathcal{S}$ reduce to thresholds of the affine predictor, so $d_\mathcal{S} = O(d)$ as in (i); for general semi-algebraic or piecewise-polynomial links, $d_\mathcal{S}$ remains polynomial in $d$~\citep{Goldberg-1995-ML}.
    (iii) \textbf{Bounded-weight ReLU networks}: let $f_\theta$ be an $L$-layer feedforward ReLU network with $d$ trainable parameters, input bound $\sup_x \norm{x} \leq B_x$, and layer-wise operator norm at most $B_\theta$. COmposing the Lipschitz layers gives $L_f \leq L B_\theta^{L-1}\max\lrc{B_x,1}$, while the piecewise-linear pseudo-dimension bound yields $d_\mathcal{S} \leq C_{nn} L d \log d $ for a universal constant $C_{nn}$~\citep{Bartlett-2019-JMLR}.
\end{remark}

Next, we focus on another class of problems where the estimate-then-optimize pipeline is sufficiently smooth. In OR/MS, newsvendor and mean-variance portfolio problems are canonical examples of this case.

\begin{condition}[Smooth Problems]\label{cond:2}
    The downstream problem has the following properties:
    \begin{enumerate}[label=\textup{(\roman*)}]
    \item \textup{(Lipschitz Cost in Decision)} The objective function is $L_{z}$-Lipschitz continuous in $z$ uniformly over $\mathcal{Y}$.
    \item \textup{(Lipschitz Policy in Parameter)} The policy $\pi_\theta$ is $L_\theta$-Lipschitz continuous in $\theta$ uniformly over $\mathcal{X}$.
    \end{enumerate}
\end{condition}
\begin{remark}
    Given the Lipschitz constant $L_f$ of a model, for a newsvendor problem defined in Example~\ref{examp:newsvendor}, Condition~\ref{cond:2} is met with $L_z = \max\{c_u, c_o\}$ and $L_\theta = L_f$. For a mean-variance portfolio: $c(z;y) = -y^\top z + \frac{\kappa}{2} z^\top \Sigma z$ with $\mathcal{Z} = \{z \geq 0: \mathbf{1}^\top z\leq 1 \},\,\Sigma \succ 0$, and $\kappa > 0$, if $\mathcal{Y}$ is bounded with radius $B_y$, then $L_z=B_y + \kappa \norm{\Sigma}$, and $L_\theta = L_f / \lra{\kappa \lambda_{\mathrm{min}}}$ where $\lambda_{\mathrm{min}}$ is the smallest eigenvalue of $\Sigma$.
\end{remark}

Condition~\ref{cond:2}(ii) is what most distinguishes smooth from nonsmooth problems in Condition~\ref{cond:1} by the Lipschitz continuity of policy $\pi_\theta$ in $\theta$. The following Corollary~\ref{cor:2} verifies generic Assumptions~\ref{A3}--\ref{A4} directly for problems satisfying Condition~\ref{cond:2}.

\begin{corollary}[Posterior Concentration of Smooth Pipelines]\label{cor:2}
Suppose Assumptions~\ref{A1}--\ref{A2} hold, and Condition~\ref{cond:2} holds. Then Assumptions~\ref{A3}--\ref{A4} hold with 
\begin{equation*}
    d_\mathcal{F} = d, \quad A_\mathcal{F} = \max\lrc{3 D L_z L_\theta, e}, \quad \psi(r) = L_z L_\theta r, \quad \bar{r} = D,
\end{equation*}
and Theorem~\ref{thm:1} applies accordingly.
\end{corollary}
\proof{Proof of Corollary~\ref{cor:2}} The proof of Corollary~\ref{cor:2} is provided in the e-companion to this paper. \halmos\endproof
\begin{remark}
    The constants $L_y, \, L_f$, and $d_\mathcal{S}$ from Assumptions~\ref{A5} and~\ref{A6} do not appear in Corollary~\ref{cor:2}. The smooth case sidesteps the prediction layer entirely: Condition~\ref{cor:2} provides a direct Lipschitz chain with constants $L_\theta$ and $L_z$, so the entropy bound and the local risk modulus follow from a $d$-dimensional Lipschitz parameterization of the loss class. This contrasts with Corollary~\ref{cor:1}, where the loss class is piecewise constant in $\theta$ and complexity must be controlled through the multiclass partition structure of the predictor, which is where $d_\mathcal{S}$ enters.
\end{remark}
Corollary~\ref{cor:2} does not require Assumption~\ref{A6} because the smooth case derives both metric entropy and local risk stability from $L_z$ and $L_\theta$ in Condition~\ref{cor:2} alone. Assumption~\ref{A6} is retained as a global hypothesis because the out-of-sample guarantees in Section~\ref{sec:42} invoke it.
\subsection{Out-of-Sample Guarantees} \label{sec:42}
To characterize the out-of-sample performance of BCP, we first define its decision risk given any posterior $\rho \in \mathcal{P}(\Theta)$ as:
\begin{equation}
    \mathcal{R}(\rho) = \E_{\P}\lrb{ c (\pi_\rho(X); Y) }, \quad \mathcal{R}^\star = \inf_{\rho \in \mathcal{P}(\Theta)} \mathcal{R}(\rho). \label{risk-functional}
\end{equation}
We use the calligraphic $\mathcal{R}$ for a measure $\rho$ to distinguish it from $R$ for finite-dimensional $\theta$. 
To bound the out-of-sample risk $\mathcal{R}(\rho)$ induced by a posterior $\rho$, two natural comparators are $R^\star$ (parameter-level) and $\mathcal{R}^\star$ (measure-level).
We first derive the gap $\mathcal{R}(\Gibbs) - R^\star$ for the Gibbs posterior in Proposition~\ref{prop:1} to highlight the impact of model misspecification, model uncertainty, and posterior concentration on decision-making.

\begin{proposition}[Bayesian-Frequentist Gap]\label{prop:1}
    Suppose Assumptions~\ref{A5}(i) and~\ref{A6} hold, and suppose $\Theta$ has finite diameter $D$. Let $\rho \in \mathcal{P}(\Theta)$ be any posterior measure. For any measurable set $B\subseteq \Theta$, let $G = \Theta \setminus B$, assume $\rho(G) > 0$, and define the conditional posterior $\rho_{G}(\cdot) = \rho(\cdot \mid G)$.
    Then for any $\theta^\star \in \Theta^\star$,
    \begin{equation*}
        - \mathcal{E}(\rho_G) - \Delta(\rho_G) \leq \mathcal{R}(\rho) - R^\star \leq \mathcal{E}(\rho_G) + 2 L_y L_f D \rho(B), 
    \end{equation*}
    where $\mathcal{E}(\rho_G) = 2 L_y \E_{\P_X}\lrb{ \mathcal{W}_1\lra{ \P_{Y \mid X}, m(X;\rho_G) }  }$ and $\Delta(\rho_G) = 2 L_y L_f \E_{\rho_G}\lrb{\norm{\theta-\theta^\star}}$ serve as proxies for model misspecification and within-class model uncertainty, respectively.

    In addition, if $\rho=\rho_n$ is the Gibbs posterior, $B=B_\varepsilon, \, G=G_\varepsilon$ for any $\varepsilon > 0$, and the concentration $\rho_n(B_\varepsilon) \leq \delta_{n,\varepsilon}$ holds with probability at least $1- \eta$, then with probability at least $1 - \eta$,
    \begin{equation*}
        - \mathcal{E}(\rho_G) - \Delta(\rho_G) \leq \mathcal{R}(\rho) - R^\star \leq \mathcal{E}(\rho_G) + 2 L_y L_f D \delta_{n,\varepsilon}.
    \end{equation*}
    Under the assumptions of Theorem~\ref{thm:1}, 
    \begin{equation*}
        \delta_{n,\varepsilon} = \frac{\exp\lra{-\lambda n \varepsilon / 2}}{l_0 \mathbb{V}_d r_\varepsilon^d}.
    \end{equation*}
    
\end{proposition}
\proof{Proof of Proposition~\ref{prop:1}} The proof is provided in the e-companion to this paper. \halmos\endproof
Proposition~\ref{prop:1} first dissects the gap between the out-of-sample performance of BCP given any posterior $\rho$ and the frequentist best-in-class policy. This gap decomposes into three factors. First, we use $\mathcal{E}(\rho_G)$ to indicate the impact of model misspecification, and it appears symmetrically on both sides of the gap; the second factor $\Delta(\rho_G)$ measures how the conditional posterior on the risk-sublevel set $G$ disperses around $\theta^\star$ and enters only in the lower bound, so a large $\Delta(\rho_G)$ allows the gap to be strictly negative. The last factor $\rho(B)$ measures the residual probability that a Bayesian framework puts on the risk-superlevel set $B$. This factor reveals the suboptimality of conventional likelihood-based Bayesian frameworks, in which the posterior concentrates around the \textit{pseudo-true parameter} $\theta_{\mathrm{KL}} \in \argmin_{\theta \in \Theta} \E_{\P_X}\KL{\P_{Y \mid X}}{f_{Y \mid X,\theta}}$. Unless the KL divergence is proved to be a consistent surrogate for the downstream optimization objective~\citep{Ho-Nguyen-2022-MS}, $\theta_{\mathrm{KL}}$ does not necessarily lie in $\Theta^\star$. In this case, $\rho(B)$ may fail to vanish regardless of sample size. 

For the Gibbs posterior, whose concentration has been established in Theorem~\ref{thm:1}, the third factor $\rho(B)$ vanishes exponentially fast in $n$. In a well-specified case, $\mathcal{E}(\rho_G)$ vanishes on both sides of the gap. When $\Theta^\star$ is a singleton and the risk-sublevel set $G_\varepsilon$ is convex, hence the optimizer $\theta^\star$ is identifiable, $\rho_n$ will asymptotically concentrate around $\theta^\star$, and the policy $\pi_{\rho_n}$ will recover the behavior of $\pi_{\theta^\star}$. When $\Theta^\star$ is not a singleton and $G_\varepsilon$ retains a union of disjoint neighborhoods around $\theta^\star \in \Theta^\star$, $\Delta(\rho_G)$ stays bounded away from zero for every $\varepsilon$; the lower bound is then strictly negative regardless of model misspecification, leaving potential for BCP to outperform the best-in-class frequentist policy through aggregation that the latter cannot achieve.
Our numerical experiments especially highlight this advantage of aggregation when $n$ is small, so that frequentist approaches can hardly distinguish the optimal estimator given sampling variability and noise.

Proposition~\ref{prop:1} shows that the bound does not preclude strict improvement by BCP; the variational perspective below identifies structural settings in which such improvement can occur.
To this end, consider the following variational optimization problem:
\begin{align}
     \mathcal{R}^\star = \inf_{\rho \in \mathcal{P}(\Theta)} \mathcal{R}(\rho) \leq \inf_{\rho \in \mathcal{P}_\delta(\Theta)} \mathcal{R}(\rho) = R^\star
     , \label{FP}
\end{align}
where $\mathcal{P}_\delta(\Theta)$ denotes the set of all Dirac measures in $\mathcal{P}(\Theta)$. The second equality in Problem~\eqref{FP} holds because each Dirac measure $\delta(\theta)$ corresponds to a choice of parameter $\theta$. Frequentist approaches that solve Problem~\eqref{FP} can be interpreted as a finite-dimensional restriction of $\inf_{\rho \in \mathcal{P}(\Theta)} \mathcal{R}(\rho)$ by limiting the feasible measure set to $\mathcal{P}_\delta(\Theta)$ for computational tractability in the finite-dimensional parameter space $\Theta$. Yet, restricting the variational problem to a single Dirac measure is often suboptimal~\citep{Rockafellar-2009-book}. This explains why $\pi_\rho$ derived from a nondegenerate posterior may outperform the frequentist optimizer if the posterior does not degenerate to a Dirac measure. To see this, consider a simple example: The true law $Y \mid X =x \sim \frac{1}{2}\mathcal{N}(-x, 1) + \frac{1}{2}\mathcal{N}(x, 1)$ is a Gaussian mixture, and if the predictive model is $m(x;\theta) = \mathcal{N}(\theta x , 1)$, the optimal posterior for any downstream problem is a mixture of Dirac measures: $\frac{1}{2}\delta(\theta =-1) + \frac{1}{2}\delta(\theta = 1)$, not a single Dirac measure $\delta(\theta_0)$ for some $\theta_0 \in \Theta$.

Our mechanism in Equation~\eqref{da-bayes} is a tractable approximation to Problem~\eqref{FP} that does not restrict $\mathcal{P}(\Theta)$ to $\mathcal{P}_\delta(\Theta)$, and that requires only pointwise evaluations of $R_n(\theta)$. To quantify the approximation error in decision-making, we construct an oracle-type out-of-sample guarantee for our policy in Theorem~\ref{thm:2} by bounding the excess risk against the measure-level oracle $\mathcal{R}^\star$ rather than $R^\star$ used in Proposition~\ref{prop:1}. In Section~\ref{sec:5_1}, we show that this variational problem has a tractable finite-dimensional form when the feasible set is limited to a variational family $\mathcal{Q} \subset \mathcal{P}(\Theta)$, not $\mathcal{P}_\delta(\Theta)$.
\begin{theorem}[Variational Excess Risk]\label{thm:2}
    Suppose Assumptions~\ref{A1} and~\ref{A6} hold and that the oracle measure $\rho^\star \in \argmin_{\rho \in \mathcal{P}(\Theta)} \mathcal{R}(\rho)$ exists and satisfies $\KL{\rho^\star}{\rho_0} < \infty$. 
    Let $\rho_n$ be the Gibbs posterior defined in Equation~\eqref{da-bayes} with a fixed $\lambda$ and $\mathcal{D}_n$.
    Then for any $\lambda >0$, any $\eta \in (0,1)$, the following inequality holds with probability at least $ 1- 2\eta$ over the draw of $\mathcal{D}_n$:
    \begin{equation*}
        \mathcal{R}(\Gibbs) - \mathcal{R}^\star \leq  \mathcal{E}(\rho_n)  + \Gamma^\star  +
        \frac{ \KL{\rho^\star}{\rho_0} + \log(1 / \eta)  }{\lambda n}
         + \frac{\lambda}{8} + \sqrt{\frac{\log (1/\eta)}{2n}},
    \end{equation*}
    where $\mathcal{E}(\rho_n) = 2L_y\E_{\P_X}[\W_1(\P_{Y\mid X}, m(X;\Gibbs ))]$, and $\Gamma^\star \equiv \mathbb{E}_{\rho^\star}\lrb{R(\theta)} - \mathcal{R}^\star$ is data-independent. In particular, taking $\lambda = O(n^{-1/2})$ leads to
    \begin{equation*}
        \mathcal{R}(\Gibbs) - \mathcal{R}^\star \leq 
        \mathcal{E}(\rho_n) + \Gamma^\star + O(n^{-1/2}).
    \end{equation*}
\end{theorem}
\proof{Proof of Theorem~\ref{thm:2}} The proof is provided in the e-companion to this paper. \halmos\endproof
Theorem~\ref{thm:2} establishes an $O(n^{-1/2})$ excess risk rate for the Gibbs posterior $\rho_n$ against the oracle measure $\rho^\star$ up to two fixed residuals. The bound contains two non-vanishing residuals beyond the parametric rate. The first factor $\mathcal E(\rho_n)$ reflects model misspecification; under well-specification and asymptotic concentration of $\rho_n$ on $\Theta^\star$, $\mathcal E(\rho_n)$ vanishes as $n \to \infty$. The second term $\Gamma^\star$ is the oracle aggregate gap, a deterministic, data-independent quantity that captures the irreducible bias. This residual is the price of computational tractability: our framework requires only pointwise evaluation of $R_n(\theta)$ in Equation~\eqref{da-bayes} and never evaluates $\mathcal{R}_n(\rho) = \tfrac{1}{n}\sum_{i=1}^n c(\pi_\rho(x_i); y_i)$ or the first variation of $\mathcal{R}$ with respect to $\rho$, sidestepping the analytical and computational obstacles that would otherwise arise.

Theorem~\ref{thm:2} also reveals that the inverse temperature $\lambda$ governs a bias--variance tradeoff in the parametric component of the bound. A small $\lambda$ enforces strong KL regularization toward the prior $\rho_0$, reducing the variance term $\lambda/8$ but inflating the complexity term $\KL{\rho^\star}{\rho_0}/(\lambda n)$, which manifests as underfitting bias. A large $\lambda$ sharpens the posterior around empirical risk minimizers and reverses this tradeoff. Balancing the two yields the optimal scaling $\lambda = C_\lambda n^{-1/2}$ for some $C_\lambda > 0$, so that the weight placed on data relative to the prior grows as $O(n^{1/2})$ with sample size. The residuals $\mathcal{E}(\rho_n)$ and $\Gamma^\star$ lie outside this tradeoff: the former is determined by the predictive model class and the latter is data-independent, so neither is materially affected by the choice of $\lambda$ once $\rho_n$ has sufficiently concentrated on $\Theta^\star$. For moderate $n$, we therefore recommend selecting $\lambda$ in a decision-focused manner to improve the practical performance of $\pi_\rho$, analogous to tuning regularization in frequentist methods. 
\begin{remark}[BCO Oracle Advantage]\label{rem:oracle}
Since $\mathcal{P}_\delta(\Theta) \subset \mathcal{P}(\Theta)$, the variational inequality in Problem~\eqref{FP} gives $\mathcal{R}^\star \leq R^\star$ directly: the measure-level comparator is no worse than the Dirac-restricted comparator. When the cost is affine in $y$, i.e., $c(z;y) = a(z) + b(z)^\top y$, and the model is a point predictor, this gap has a clean structural interpretation. By linearity, $\E_\rho[c(z;f_\theta(x))] = c(z;\bar{f}_\rho(x))$ where $\bar{f}_\rho(x) \coloneqq \E_\rho[f_\theta(x)]$, so the BCP under any posterior $\rho$ optimizes against the posterior predictive mean. As $\rho$ ranges over $\mathcal{P}(\Theta)$, the mean $\bar{f}_\rho$ sweeps out $\mathrm{conv}(\mathcal{F}_\Theta)$, which is strictly richer than $\mathcal{F}_\Theta$ when the model class is nonlinear. The strict inequality $\mathcal{R}^\star < R^\star$ then holds whenever the true conditional mean $\mu(X) \in \mathrm{conv}(\mathcal{F}_\Theta)$ $\P_X$-a.e.\ and no single $\theta \in \Theta$ achieves the oracle decision $\pi^\star(X)$ $\P_X$-a.e. This is the structural mechanism through which BCO can strictly improve over any frequentist policy under affine cost.
\end{remark}

\section{Computational Methods}\label{sec:5}
The operational likelihood introduces a computational challenge: unlike standard Bayesian frameworks in which conjugate likelihood--prior pairs yield closed-form posteriors, the operational likelihood generally leads to a Gibbs posterior without an analytical density.
We therefore turn to approximate Bayesian inference.
The differentiability of the downstream optimization problem provides a natural criterion for selecting algorithms.  
When the cost gradient $\nabla_\theta c(\pi_\theta(x);y)$ exists and is cheap to evaluate, we tailor a variational inference scheme that recasts posterior approximation as a tractable optimization problem (Section~\ref{sec:5_1}).  
When the gradient is undefined or prohibitive to evaluate, as is common in combinatorial and nonconvex settings, we propose an MCMC sampler that requires only pointwise evaluations of the operational likelihood (Section~\ref{sec:5_2}).

\subsection{Variational Inference for Differentiable Problems}\label{sec:5_1}
Variational inference refers to a family of methods that approximate an intractable posterior by optimizing over a parametric family of distributions, thereby recasting Bayesian inference as an optimization problem~\citep{Blei-2017-JASA}. Note that in our algorithm the parametric family governs the posterior approximation, not the likelihood itself: the latter retains its decision-focused form throughout.

Let $\mathcal{Q} \coloneqq \{q_\omega : \omega \in \Omega\} \subset \mathcal{P}(\Theta)$ denote a parametric density family.  We approximate the Gibbs posterior $\Gibbs$ by the best-in-class candidate under the reverse KL divergence:
\begin{equation}\label{KLVI}
  q_{\omega^\star}
  \coloneqq \argmin_{q_\omega \in \mathcal{Q}}\;
  \KL{q_\omega}{\Gibbs}.
\end{equation}
Because Problem~\eqref{KLVI} minimizes the reverse KL divergence $\KL{q_\omega}{\Gibbs}$, the optimizer $q_{\omega^\star}$ exhibits the well-known mode-seeking behavior and may underestimate posterior uncertainty, especially when a unimodal family $\mathcal{Q}$ is used to approximate a multimodal posterior.  Alternative divergence measures can mitigate this limitation but generally incur additional computational cost; we therefore adopt Problem~\eqref{KLVI} for its tractability.

Indeed, Problem~\eqref{KLVI} admits the following reformulation:
\begin{align}
  \argmin_{q_\omega \in \mathcal{Q}}\;\KL{q_\omega}{\Gibbs}
  &= \argmin_{q_\omega \in \mathcal{Q}}\;
     \E_{q_\omega}\!\bigl[\log q_\omega(\theta)
       - \log\Gibbs(\theta)\bigr]
     \notag\\
  &= \argmin_{q_\omega \in \mathcal{Q}}\;
     \E_{q_\omega}\!\bigl[\log q_\omega(\theta)
       + \lambda n\,R_n(\theta) - \log\rho_0(\theta)\bigr]
     + \log I_{n,\lambda}
     \notag\\
  &= \argmin_{q_\omega \in \mathcal{Q}}\;\Bigl\{
     \E_{q_\omega}[R_n(\theta)]
     + \tfrac{1}{\lambda n}\,\KL{q_\omega}{\rho_0}
     \Bigr\},
  \label{VI}
\end{align}
where $I_{n,\lambda} = \int_{\Theta}\exp\lra{-\lambda n\,R_n(\theta)}\,\rho_0(\theta)\,\mathrm{d}\theta$ is the normalizing constant independent of $q_\omega$ and hence drops from the optimization. Problem~\eqref{VI} reveals that the variational problem is again a finite-dimensional restriction: the feasible set reduces from $\mathcal{P}(\Theta)$ to $\mathcal{Q}$, not to $\mathcal{P}_\delta(\Theta)$, while the objective retains the same structure as Problem~\eqref{bissiri_eq}.

When the estimate-then-optimize pipeline is differentiable, that is, when $\nabla_\theta c(\pi_\theta(x);y)$ exists, and hence so does $\nabla_\theta R_n(\theta)$, Problem~\eqref{VI} can be solved by first-order methods whose per-iteration cost matches that of the frequentist IEO alternative. In our numerical experiments, we implement a mean-field Gaussian family (Section~\ref{sec:6_1}) and a mean-field Laplace family (Section~\ref{sec:6_2}) for their computational tractability. 
When the prior and variational family are both mean-field Gaussian (or both mean-field Laplace), $\KL{q_\omega}{\rho_0}$ is available in closed form, eliminating the need for stochastic approximation of the regularizer.
We summarize this gradient-descent procedure in Algorithm~\ref{algo:1} of~\ref{EC_algorithm}.

Once $q_{\omega^\star}$ is obtained, the PPD $m(x;q_{\omega^\star})$ is fully determined for every $x \in \mathcal{X}$.  The remaining step is to evaluate the BCP $\pi_{q_{\omega^\star}}(x)$ by solving, for each $x$, a stochastic program whose expectation is taken over $m(x;q_{\omega^\star})$.  We approximate this program via SAA:
\begin{equation*}
  \pi_{q_{\omega^\star}}(x)
  \approx \argmin_{z \in \mathcal{Z}}\;
  \frac{1}{K}\sum_{k=1}^{K} c(z;y_k),
\end{equation*}
where $\theta_k \sim q_{\omega^\star}$ and
$y_k = f_{\theta_k}(x)$ is the corresponding prediction, so that each $y_k$ constitutes a draw from $m(x;q_{\omega^\star})$.  The number of scenarios $K$ is a user-specified quantity that trades off prescriptive accuracy against computational cost.  Since the PPD is fully determined after inference, scenario reduction techniques~\citep{Henrion-2022-MP, Bertsimas-2023-OR} can be applied to refine the scenario set; we leave this extension for future work.

\subsection{MCMC for Nondifferentiable Problems}\label{sec:5_2}
Our Gibbs posterior assigns probability density over $\Theta$ based on the operational likelihood $L_\lambda(\mathcal{D}_n\mid\theta)$, requiring only pointwise evaluation of $R_n(\theta)$.
This gradient-independence enables MCMC sampling, a workhorse of Bayesian statistics.
In this work, we tailor a Metropolis-Hastings (MH) algorithm for BCO, summarized in Algorithm~\ref{algo:2} of~\ref{EC_algorithm}, because of its simplicity and generality.

Derivative-free MH algorithms mainly use a random-walk proposal and an accept-reject mechanism to sample from complex distributions. It is well established, however, that na\"ive random-walk MH suffers from a curse of dimensionality: in high-dimensional spaces the acceptance rate collapses and the chain fails to produce effective samples~\citep{Roberts-2001-SS}. 
To improve the efficiency of MH algorithms in our framework for expressive models, we consider the following practical techniques. 

First, we adopt the adaptive-scaling scheme~\citep{Garthwaite-2016-CSTM} and apply it to diagonal Gaussian proposals, avoiding the computation of a high-dimensional covariance matrix. Second, following the standard practice of MCMC methods, we construct multiple chains using different initial points to explore the multimodality of the posterior and discard the first $T_b$ samples (burn-in) from each chain of length $T$ to reduce initialization bias and better approximate stationarity. Third, to determine the inverse temperature, i.e., the constant $C_\lambda$ characterizing $\lambda n = C_\lambda n^{1/2}$ following Theorem~\ref{thm:2}, we gradually increase the value of $C_\lambda$ until the traces of empirical decision loss stabilize. This algorithm is applied to the numerical experiment in Section~\ref{sec:6_3}.

After executing the MH algorithm, we use the sample set, denoted by $\Theta_{\mathrm{MH}}$, to construct the SAA model for BCP. It is straightforward to use all the post-burn-in samples, setting the number of scenarios $K=|\Theta_{\mathrm{MH}}|=N_c(T-T_b)$, where $N_c$ is the number of chains, and $w_k=1/K$. However, a smaller scenario set may be preferred for computational reasons. MCMC thinning~\citep{South-2022-ARSA} offers one principled approach. For simplicity, we extract $K$ samples at regular intervals. A more theoretically sound alternative is the \emph{support points} approach~\citep{Mak-2018-AOS}; we regard this as a promising direction for future work.
\section{Numerical Experiments}\label{sec:6}

\subsection{Experiment 1: Two-Stage Shipment}\label{sec:6_1}
Our first numerical experiment is based on the synthetic problem of~\cite{Bertsimas-2020-MS}, which corresponds to a two-stage linear program with recourse that can be formulated as follows:
\begin{equation*}
    \min_{z \geq 0}\; c(z;y) = \min_{z \geq 0}\; c_q^\top z + Q(z,y),
\end{equation*}
where $z \in \R^4_+$ denotes the first-stage ordering amounts, $y\in \R^{12}$ the uncertain demand, $c_q \in \R^4$ the unit-order cost, and $Q(\cdot,y): \mathcal{Z} \rightarrow \R$ the second-stage recourse function. We adopt the original setting for the cost coefficients and data-generating process.

We design our experiments as follows. First, we consider $n\in\lrc{50,100,200, 400}$ to focus on the small-to-moderate sample regimes. 
Based on the original setting, we append noisy dimensions to $X$ that have no prescriptive power to vary the signal-to-noise ratio in order to illustrate the robustness of different approaches against model uncertainty. 
The original ARMA(2,2) process generates a 3-dimensional $X$; we append additional independent standard Gaussian features so that the number of noisy dimensions ranges over $\{0, 3, 6, 9\}$ (0 corresponds to the original setting).

For each run, we generate an independent validation set with 50 samples to facilitate hyperparameter tuning and early stopping. Once the training finishes, we use an independently generated test set $\mathcal{T}$, with 5,000 samples, to evaluate the out-of-sample performance of different approaches. Given a data-driven policy $\pi$ that prescribes $z$, we compute 
$
    \hat{R} = \frac{1}{\abs{\mathcal{T}}} \sum_{(x_i,y_i) \in \mathcal{T}} c(\pi(x_i); y_i)
$
as the out-of-sample metric. Note that this metric evaluates the decision quality of $\hat\pi(x_i)$ by solving the recourse problem $Q$ under the realization $y_i$. We repeat 50 independent runs for each sample size.

Five policies, drawn from three frameworks, are considered in this experiment. The first is the exponential Nadaraya-Watson kernel estimator baseline, denoted by ENW, which assigns weights to training samples conditional on a new context and solves the corresponding weighted SAA. The second one is the frequentist IEO baseline, which solves the empirical risk minimization in Equation~\eqref{erm-g} by gradient descent. The third is BCO, which yields three policies: TPS, MAP, and BCP, derived from the same Gibbs posterior with a mean-field Gaussian variational family. For both IEO and BCO, we use a simple piecewise affine model $f_\theta(X) = \mathrm{ReLU}(WX+b)$ where $\mathrm{ReLU}$ refers to the rectified linear unit. Details on each framework’s configuration, including kernel bandwidth, learning rate, batch size, and early-stopping criterion, are provided in~\ref{ECEXP1}.



\begin{figure}[htbp]
\centering
\caption{Average Out-of-Sample Cost of Five Policies in the Two-Stage Shipment Problem}\label{fig:case1}
\includegraphics[width=0.95\textwidth]{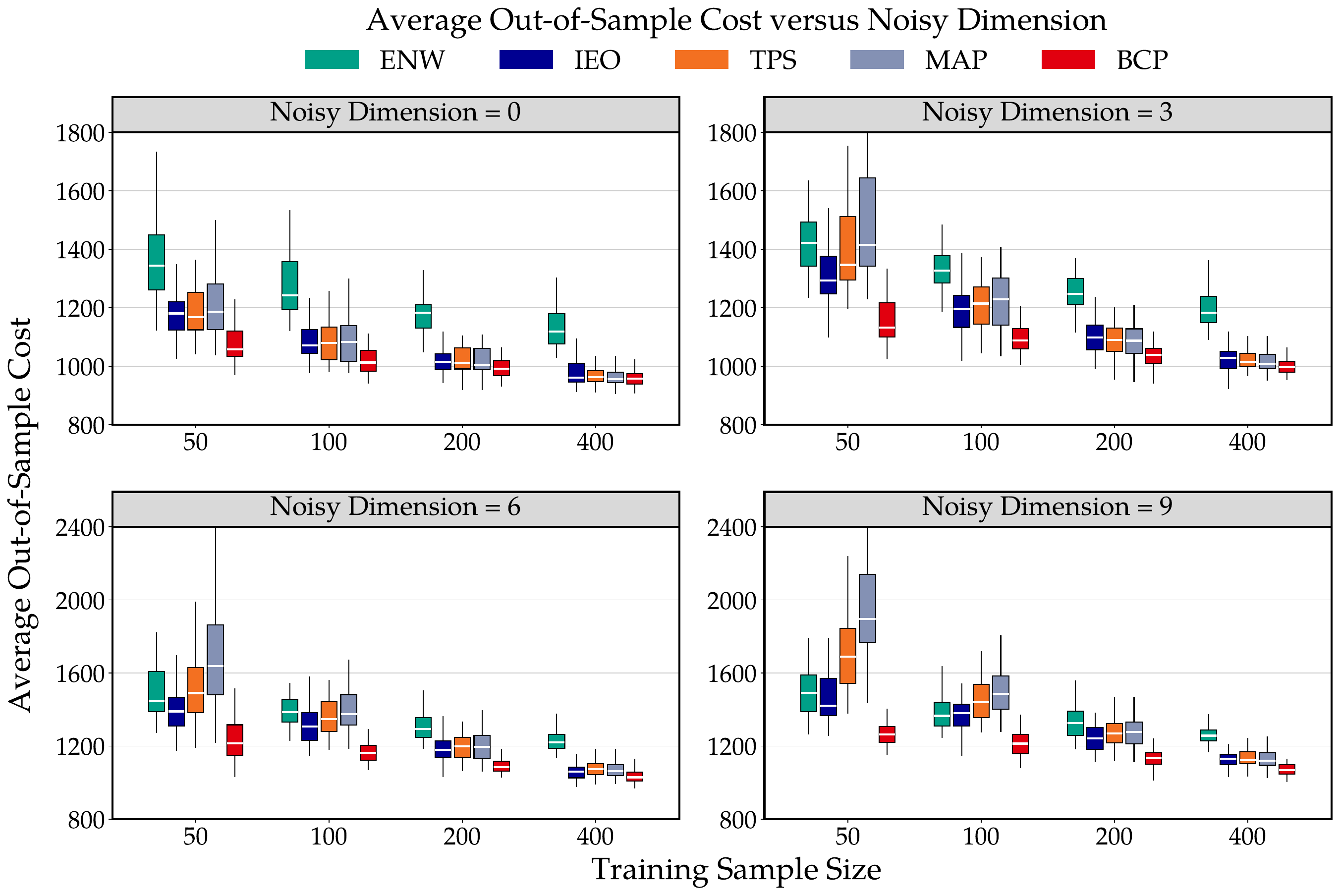}
\begin{minipage}{\textwidth}
\vskip 0.1cm
\small
\textit{Notes}. ENW: exponential Nadaraya-Watson kernel regressor; IEO: integrated estimate-optimize; TPS: Thompson policy; MAP: maximum a posteriori; BCP: Bayesian contextual policy. All results over 50 independent runs. TPS, MAP, and BCP share the same posterior. Box: 25th-75th percentiles; horizontal box line: median.
\end{minipage}
\end{figure}

Figure~\ref{fig:case1} reports the average out-of-sample cost across four sample sizes and four noisy dimensions. All policies improve with more training data and deteriorate with more noisy dimensions, reflecting lower signal-to-noise ratios and elevated model uncertainty. Overall, BCP outperforms IEO, the frequentist alternative, and ENW, which also leverages distributional estimates.

Figure~\ref{fig:case1} clearly reveals the posterior concentration established by Theorem~\ref{thm:1} in the empirical behavior of TPS and MAP compared to BCP.
Regardless of the noisy dimension, the medians of TPS and MAP costs converge to that of BCP. This pattern is consistent with the exponential concentration of $\Gibbs$ around the best-in-class parameter: as $\lambda n$ grows in $O(n^{1/2})$, the posterior becomes increasingly sharp, so a single posterior draw (TPS) or the posterior mode (MAP) becomes an increasingly accurate proxy of the full PPD used by BCP. Conversely, in the small-sample regime where the posterior is sufficiently wide, TPS draws parameter vectors far from the high-density region of the posterior and MAP commits to a mode that is not representative of the bulk of the posterior mass, resulting in inferior and volatile performance. In this sense, the gap between TPS (or MAP) and BCP can be interpreted as an empirical signature of model uncertainty that is discarded when the posterior is collapsed to a single parameter vector.
\begin{figure}[htbp]
\centering
\caption{Normalized Out-of-Sample Performance of Five Policies using BCP Median as Baseline in the Two-Stage Shipment Problem}\label{fig:case1b}
\includegraphics[width=0.7\textwidth]{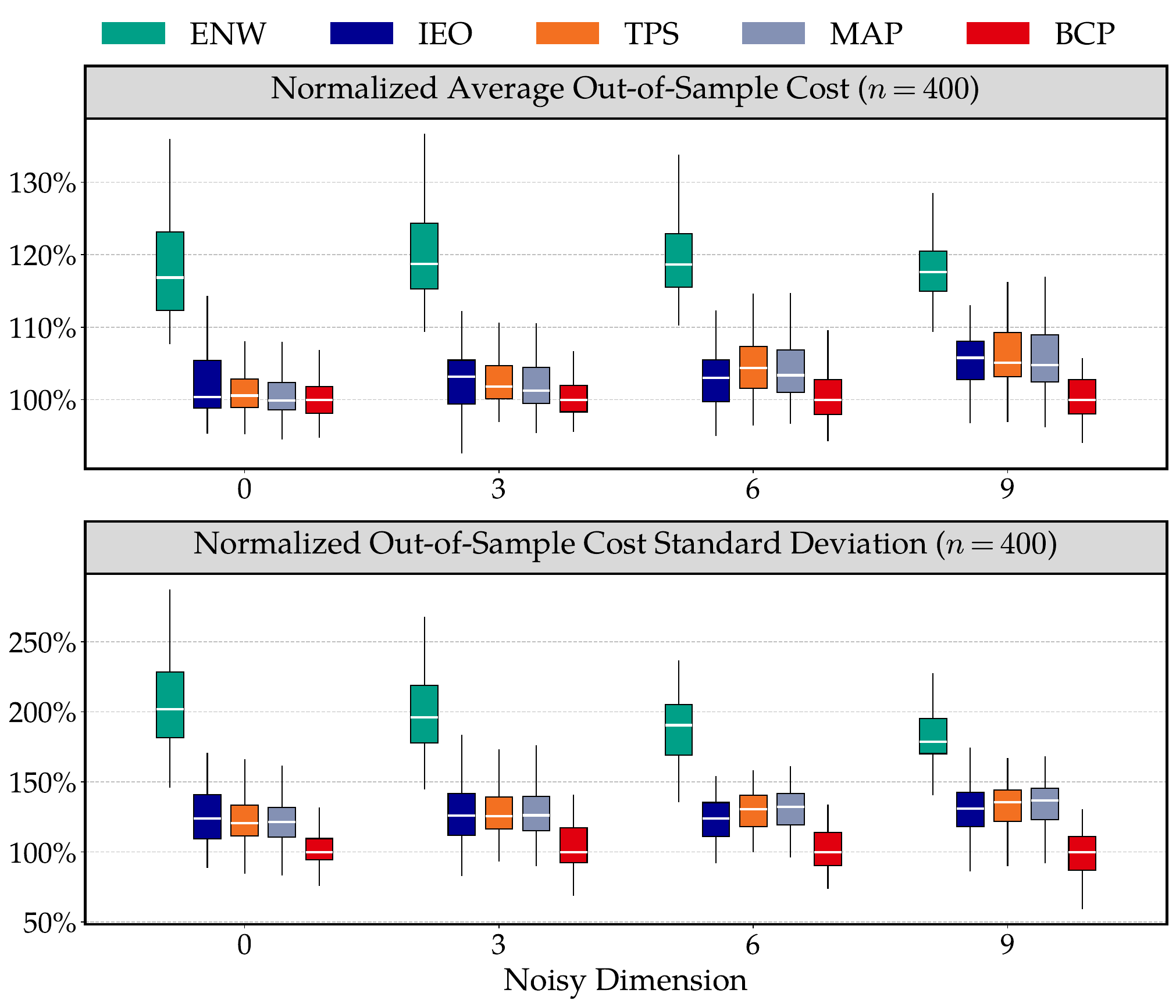}
\begin{minipage}{\textwidth}
\vskip 0.1cm
\small
\textit{Notes}. Out-of-sample performance normalized by BCP's median at $n=400$, per each noisy dimension. Top: normalized average out-of-sample cost. Bottom: normalized standard deviation. Both computed over 50 independent runs. TPS, MAP, and BCP share the same Gibbs posterior.
\end{minipage}
\end{figure}

BCP consistently outperforms IEO across all settings, with the largest gains at $n=50$ under severe noise, where model uncertainty is highest and a single empirical risk minimizer is least reliable. This highlights the benefit of the decision-focused belief updating mechanism combined with Bayesian model aggregation: by propagating the posterior mass through the PPD rather than committing to a single estimator, BCP hedges against model uncertainty that frequentist approaches cannot capture, echoing the Bayesian-frequentist gap characterized in Proposition~\ref{prop:1}. Yet BCP's advantage extends beyond average cost reduction. 

Figure~\ref{fig:case1b} normalizes out-of-sample performance at $n = 400$ using BCP's median as baseline, reporting average cost and standard deviation for all policies across four noisy dimensions. The average-cost advantage of BCP is modest under no noise, consistent with posterior concentration, but grows to at least 5\% as the noisy dimension increases to 9, confirming that Bayesian aggregation remains beneficial even at moderate sample sizes.

More strikingly, the bottom panel of Figure~\ref{fig:case1b} shows that the standard deviation of the out-of-sample cost for ENW, IEO, TPS, and MAP is systematically $\sim$125\% (175\% for ENW) regardless of noise level. This variance reduction is a known consequence of Bayesian aggregation~\citep{Satop-2022-OR}, naturally inherited by BCP. Even when the Gibbs posterior is sufficiently concentrated around the frequentist minimizer, the residual posterior dispersion still yields a predictive distribution strictly richer than any single-point predictor. The resulting policy is therefore less sensitive to training sample realizations. For decision-makers concerned with tail behavior, this stability may be as valuable as the reduction in expected cost.

\subsection{Experiment 2: Newsvendor}\label{sec:6_2}
We now consider an unconstrained contextual newsvendor problem with cost function:
\begin{equation*}
    c(z;y) \coloneqq c_u\lra{y-z}^+ + c_o\lra{z-y}^+,
\end{equation*}
where $c_u$ and $c_o$ are the unit underage and overage costs, respectively. We denote the critical fractile $\tau = c_u / (c_u + c_o)$ and set $c_u + c_o = 1$. We follow~\cite{Shen-2024-JMLR} to construct three data-generating processes reflecting distinct feature-demand characteristics. Their conditional $\tau$-quantile functions $f(\tau\mid X)$ are defined as:
\begin{align}
f(\tau\mid X) &= 2 A^\top X + F^{-1}_{t_3}(\tau), \tag{DGP-1}\\
f(\tau\mid X) &= \exp\lra{0.1\cdot A^\top X} + \abs{\sin\lra{\pi B^\top X}} \Phi^{-1}(\tau), \tag{DGP-2}\\
f(\tau\mid X) &= 3 X_1 + 4\lra{X_2-0.5}^2 + 2 \sin\lra{\pi X_3} - 5 \abs{X_4 - 0.5} \notag \\
              &+ \exp\lra{0.1\lra{B^\top X - 0.5}} \Phi^{-1}(\tau) \tag{DGP-3}
\end{align}
where $\pi$ exceptionally denotes the mathematical constant, $X \sim \mathrm{Uniform}(0,1)^8$, $F^{-1}_{t_3}(\tau)$ is the $\tau$-quantile of a Student’s \(t_3\) distribution, $\Phi^{-1}(\tau)$ the standard normal quantile, and $A, B \in \R^8$ are factor vectors. DGP-1 tests heavy-tail handling under linear location structure; DGP-2 tests separation of location and scale features; DGP-3 tests recovery of sparse nonlinear and periodic structure under heteroskedasticity.

We consider $n\in\{256, 512, 1024, 2048, 4096\}$ and $\tau\in\{0.50, 0.75, 0.95\}$ spanning small- to large-sample regimes and symmetric through highly asymmetric loss structures.
Since a separate model is trained per $\tau$, we use a feedforward network mapping $8$-dimensional inputs to scalar demands via two hidden layers of width $128$ with ReLU activations (17,793 parameters). For the unconstrained newsvendor, $f_\theta(x) = \pi_\theta(x)$, so frequentist IEO reduces to the decision rule of~\cite{Ban-2019-OR}. An independent validation set of size $n/4$ is used for early stopping and regularization selection. A mean-field Laplace family is used for both the prior and posterior.

We compared only IEO and BCP, omitting nonparametric baselines given the established advantage of neural networks in newsvendor problems~\cite{Han-2025-MS}. Out-of-sample performance is evaluated on a $2^{15}$ test set generated via a scrambled Sobol sequence, reporting the optimality gap relative to the oracle policy (since the quantile function is known). Results are averaged over 50 independent runs.~\ref{ECEXP2} lists all the omitted details.



\begin{table}[htbp]
\centering
\caption{Optimality Gap ($\mu\% \pm \sigma \%$) in the Newsvendor Problem}
\label{tab:newsvendor}
\begin{tabular}{lcccccc}
\toprule
\multirow{2}{*}{$\tau$} & \multirow{2}{*}{Policy}
& \multicolumn{5}{c}{Sample size $n$} \\
\cmidrule(lr){3-7}
& & $256$ & $512$ & $1024$ & $2048$ & $4096$ \\
\midrule
\multicolumn{7}{c}{DGP-1} \\
\multirow{2}{*}{$0.50$}
 & IEO & $26.3\pm29.9$ & $10.5\pm4.09$ & $6.47\pm4.31$ & $2.42\pm1.64$ & $1.04\pm0.74$ \\
 & \bcohl{BCP} & \bcohl{$5.18\pm2.33$} & \bcohl{$2.10\pm1.04$} & \bcohl{$1.22\pm0.33$} & \bcohl{$0.94\pm0.35$} & \bcohl{$0.76\pm0.27$} \\
\multirow{2}{*}{$0.75$}
 & IEO & $22.5\pm22.7$ & $13.9\pm6.11$ & $7.79\pm4.57$ & $2.52\pm2.26$ & $1.28\pm0.96$ \\
 & \bcohl{BCP} & \bcohl{$5.11\pm1.86$} & \bcohl{$2.90\pm1.30$} & \bcohl{$1.39\pm0.61$} & \bcohl{$0.97\pm0.35$} & \bcohl{$0.79\pm0.24$} \\
\multirow{2}{*}{$0.95$}
 & IEO & $50.4\pm40.6$ & $31.1\pm15.9$ & $13.3\pm7.61$ & $5.16\pm3.13$ & $3.41\pm2.23$ \\
 & \bcohl{BCP} & \bcohl{$7.70\pm5.36$} & \bcohl{$5.35\pm2.68$} & \bcohl{$3.47\pm2.22$} & \bcohl{$2.38\pm1.43$} & \bcohl{$1.82\pm0.75$} \\
\midrule
\multicolumn{7}{c}{DGP-2} \\
\multirow{2}{*}{$0.50$}
 & IEO & $31.0\pm22.9$ & $14.7\pm9.13$ & $8.26\pm5.59$ & $3.04\pm3.05$ & $0.98\pm0.81$ \\
 & \bcohl{BCP} & \bcohl{$2.49\pm1.38$} & \bcohl{$1.33\pm0.75$} & \bcohl{$0.68\pm0.30$} & \bcohl{$0.42\pm0.14$} & \bcohl{$0.30\pm0.14$} \\
\multirow{2}{*}{$0.75$}
 & IEO & $37.7\pm23.3$ & $24.6\pm8.19$ & $15.4\pm5.08$ & $10.2\pm1.75$ & $8.97\pm0.86$ \\
 & \bcohl{BCP} & \bcohl{$10.2\pm2.77$} & \bcohl{$9.22\pm1.07$} & \bcohl{$8.53\pm0.45$} & \bcohl{$8.37\pm0.39$} & \bcohl{$8.14\pm0.25$} \\
\multirow{2}{*}{$0.95$}
 & IEO & $97.9\pm52.6$ & $63.0\pm16.7$ & $41.6\pm11.8$ & $30.7\pm4.93$ & $25.9\pm2.44$ \\
 & \bcohl{BCP} & \bcohl{$31.2\pm4.94$} & \bcohl{$28.7\pm3.54$} & \bcohl{$27.1\pm2.16$} & \bcohl{$25.5\pm1.10$} & \bcohl{$24.9\pm0.78$} \\
\midrule
\multicolumn{7}{c}{DGP-3} \\
\multirow{2}{*}{$0.50$}
 & IEO & $33.9\pm13.2$ & $25.6\pm6.38$ & $17.4\pm4.88$ & $12.2\pm3.59$ & $6.71\pm3.46$ \\
 & \bcohl{BCP} & \bcohl{$21.3\pm2.19$} & \bcohl{$15.4\pm3.96$} & \bcohl{$7.91\pm2.61$} & \bcohl{$4.31\pm1.13$} & \bcohl{$3.01\pm0.49$} \\
\multirow{2}{*}{$0.75$}
 & IEO & $32.6\pm9.77$ & $25.8\pm4.16$ & $18.9\pm4.19$ & $14.1\pm3.50$ & $7.12\pm3.14$ \\
 & \bcohl{BCP} & \bcohl{$22.0\pm2.82$} & \bcohl{$18.2\pm3.16$} & \bcohl{$10.6\pm4.35$} & \bcohl{$4.96\pm1.35$} & \bcohl{$2.85\pm0.65$} \\
\multirow{2}{*}{$0.95$}
 & IEO & $51.9\pm17.2$ & $41.3\pm12.4$ & $25.6\pm6.14$ & $17.8\pm3.75$ & $9.31\pm3.91$ \\
 & \bcohl{BCP} & \bcohl{$23.7\pm3.21$} & \bcohl{$21.6\pm2.27$} & \bcohl{$19.2\pm1.92$} & \bcohl{$12.4\pm3.98$} & \bcohl{$6.46\pm2.99$} \\
\bottomrule
\end{tabular}
\begin{minipage}{\textwidth}
\vskip 0.3cm
\small
\textit{Notes}. Each entry is computed from 50 independent trials and reports the mean and standard deviation $(\mu\pm\sigma)$ of the optimality gap.
\end{minipage}
\end{table}

Table~\ref{tab:newsvendor} shows that BCP outperforms IEO in all settings, with the most pronounced advantage at $n \leq 2048$. 
Note that $\ell_1$-regularized regressor (IEO here) is a Bayesian MAP estimate under a Laplace prior. 
When $n$ is small and many parameter vectors are empirically indistinguishable, BCP's decision-focused aggregation over the posterior reduces mean gap by at least 5\% for $n \leq 1024$.

BCP's robustness against sampling variability is further highlighted by the standard deviation of the optimality gap in Table~\ref{tab:newsvendor} (and Figure~\ref{fig:intro}). Even with a homogeneous $\mathrm{Laplace}(0,1)$ prior, BCP hedges effectively against sampling variability, whereas IEO remains unstable despite $\ell_1$-regularization and validation-based model selection.
This robustness persists even at $n=4096$ across all settings. Although the two policies share a similar convergence order under standard assumptions, their practical performance differs substantially. As a byproduct, these results demonstrate that DFL is applicable not only large-sample regimes but also under small- and moderate-sample conditions, provided that model uncertainty and sampling variability are properly addressed.

Across critical fractiles, DGP-1 and DGP-3 exhibits similar performance at $\tau = 0.50$ and $0.75$, while the gap of each policy grows under DGP-2. 
This is because the oscillatory scale component $\abs{\sin\lra{\pi B^\top X}}\Phi^{-1}(\tau)$ in DGP-2 vanishes at $\tau = 0.50$ (where $\Phi^{-1}(0.5)=0$) but activates at $\tau = 0.75$.
At $\tau=0.95$, residual gaps reflect model misspecification from the heteroscedastic architecture's inability to capture the oscillatory scale; BCP nonetheless converges faster to the best-in-class performance than IEO.

\subsection{Experiment 3: Return-Constrained Portfolio}\label{sec:6_3}
To validate our framework on a nondifferentiable problem, we adapt the contextual portfolio problem of~\cite{SPO}: the decision-maker minimizes the portfolio variance, subject to a hard return constraint, namely:
\begin{align}
    \argmin_{\pi}         \quad & \E_{\P}\,[\pi^\top(X) \, \Sigma \,  \pi(X)] \\
    \text{s.t.} \quad  & \P\lra{ Y^\top \pi(X) \geq U_0  } \geq \alpha, \; \mathbf{1}^\top \pi(X) \leq 1, \notag
\end{align}
where $X \in \mathcal{X} \subset \R^5$ is the context, $Y \in \mathcal{Y} \subset \R^{50}$ the uncertain return, $\pi: \mathcal{X} \rightarrow \mathcal{Z} \subset \R^{50}$ the policy, $\Sigma \in \R^{50 \times 50}$ a known covariance matrix, $\alpha \in [0,1]$ the target feasibility level, and $U_0$ the target return. A dataset $\mathcal{D}_n$ drawn from $\P$ is available.~\ref{ECEXP3} provides the details of the data-generating process. 

We consider $n\in\{100, 200, 400\}$, as larger samples do not materially affect feasibility for any policy. Each configuration is repeated 50 times with independent training sets; Evaluation uses a fixed test set $\mathcal T$ of 10,000 samples drawn from the same distribution.
Two metrics are reported. The first is marginal feasibility: $F_\pi = \frac{1}{\abs{\mathcal T}} \sum_{(x_i,y_i) \in \mathcal T}\mathbb{I}\lrc{ y_i^\top \pi(x_i) \geq U_0 } $. Second is the relative regret of feasible samples:
\begin{equation*}
G_{\pi} = \frac{ \sum_{(x_i, y_i) \in \mathcal{T}} \mathbb{I}\lrc{ y_i^\top \pi(x_i) \geq U_0 } \pi^\top(x_i) \Sigma\, \pi(x_i) }
{ \sum_{(x_i, y_i) \in \mathcal{T}} \mathbb{I}\lrc{ y_i^\top \pi(x_i) \geq U_0 } c_i^\star } ,
\end{equation*}
where $c_i^\star$ is the oracle objective for sample $y_i \in \mathcal{T}$. We set $\alpha = 0.95$ and $U_0 = 0.09$, chosen so that the constraint is non-trivial yet the oracle remains feasible on nearly all test samples.

Three policies are compared. The first baseline pairs the ENW estimator with a CVaR model for feasibility, following~\cite{Rahimian-2023-SIAMOPT} and~\cite{Lin-2022-POMS}. The second baseline is frequentist IEO from~\cite{Mandi-2025-NIPS}, which trains point predictors and produces risk-averse point predictions rather than combining with a risk-averse model. BCP is paired with the same CVaR model, demonstrating its compatibility with risk-averse optimization. Both IEO and BCP use a one-hidden-layer network of width 200 (11,250 parameters). See~\ref{ECEXP3} for policy construction and training details. 

\begin{figure}[htbp]
\centering
\caption{Out-of-Sample Feasibility and Relative Regret of ENW, IEO and BCP in the Return-Constrained Portfolio Problem}\label{fig:case3}
\includegraphics[width=1\textwidth]{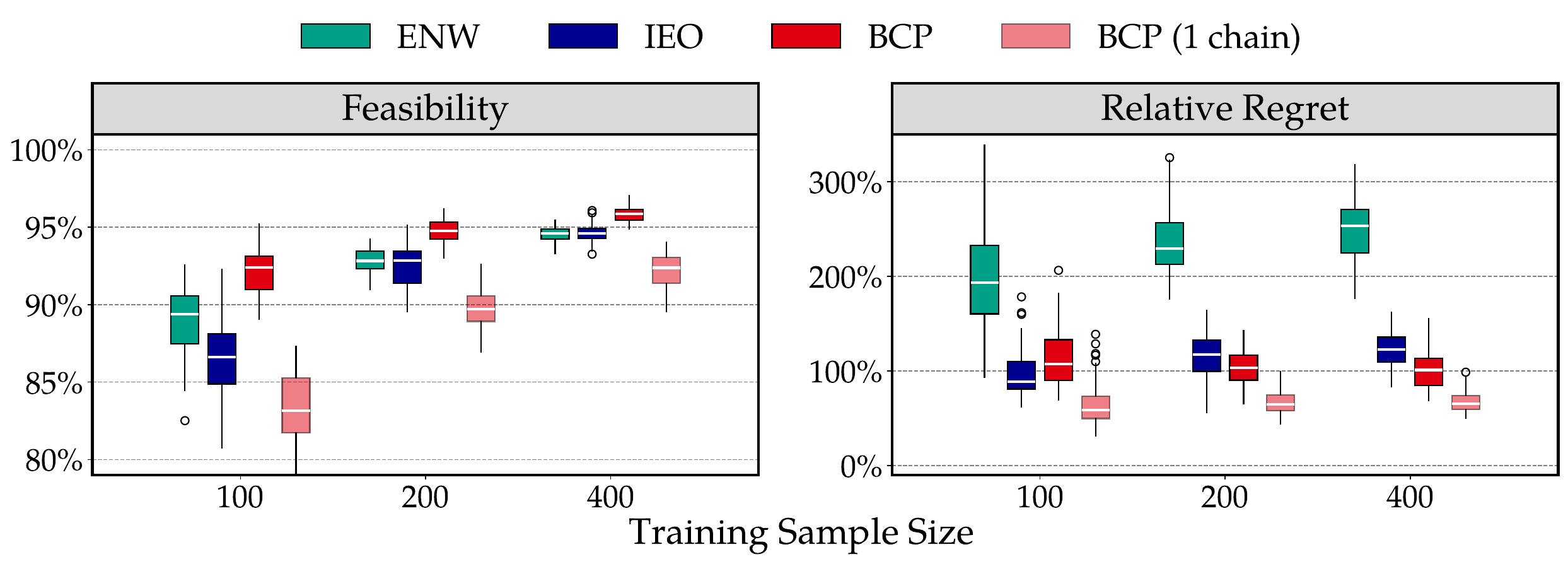}
\begin{minipage}{\textwidth}
\vskip 0.2cm
\small
\textit{Notes}. Both ENW and BCP prescribe decisions via a CVaR model given distributional estimates. IEO uses a deterministic optimization model given point predictions. BCP aggregates posterior samples from four independent MCMC chains, whereas BCP (1 chain) uses samples from a single MCMC chain. Both metrics are computed from 50 independent trials per sample size.
\end{minipage}
\end{figure}

The left panel of Figure~\ref{fig:case3} reports out-of-sample feasibility. BCP reaches the target $\alpha = 0.95$ at $n=400$ and dominates both baselines across all sample sizes, while ENW and IEO remain statistically indistinguishable. This is consistent with their shared failure to propagate model uncertainty: ENW reweights training samples and IEO commits to a point predictor, whereas BCP evaluates constraint against the PPD $m(x;\Gibbs)$, which absorbs posterior dispersion that frequentist alternatives discard. The substantially lower feasibility of BCP (1 chain) isolates a separate mechanism. 
Our MH algorithm accepts a proposal only if its empirical feasibility exceeds $\alpha$ on $\mathcal{D}_n$; a single chain concentrates on a in-sample feasible region unrepresentative of the full posterior, yielding an overly optimistic out-of-sample feasibility estimate. Four independent chains attenuates this bias by covering posterior modes that single-chain exploration misses.

The right panel reports relative regret. BCP and IEO achieve comparable regret while ENW is roughly twice as costly, confirming that BCP retains the benefit of DFL.
At $n=100$, BCP shows marginally higher median regret than IEO but substantially higher feasibility; this trade-off resolves at $n \geq 200$, where BCP attains both lower regret and higher feasibility. 
We attribute this to IEO's reliance on a surrogate loss to approximate decision quality, whereas BCP optimizes the actual downstream decision loss during the random walk.

This experiment demonstrates two extensions of BCO beyond the differentiable settings of Sections~\ref{sec:6_1}--\ref{sec:6_2}. First, paired with the adaptive MH sampler, the framework handles nondifferentiable problems where gradient-based methods are inapplicable. Second, the aggregation mechanism carries over to problems with uncertainty in the constraints, where propagating posterior mass enables compatibility with risk-averse formulations such as CVaR constraints. We note that the out-of-sample guarantees of Section~\ref{sec:42} cover uncertainty in the objective; the chance-constrained setting is treated empirically here, and a formal analysis remains an open direction.

\section{Concluding Remarks}\label{sec:7}
We proposed BCO, a framework unifying Bayesian statistics, machine learning, and OR/MS for data-driven CSO. 
By replacing a statistical likelihood with an operational one induced by downstream decision loss, BCO yields a decision-focused Gibbs posterior that accounts for model uncertainty and is robust to misspecification. The resulting Bayesian contextual policy aggregates candidate models through the posterior predictive distribution, propagating model uncertainty into the decision objective. We established exponential posterior concentration and nonasymptotic out-of-sample guarantees that characterize the roles of misspecification, model uncertainty, and posterior concentration. We further developed practical inference via variational inference and adaptive Metropolis-Hastings. Experiments on two-stage shipment, newsvendor, and return-constrained portfolio problems confirm that BCO improves out-of-sample performance and stability against model uncertainty and sampling variability, especially in small- and moderate-sample regimes.

\clearpage
\putbib[ref]
\end{bibunit}

\clearpage
\appendix
\setcounter{section}{0}
\renewcommand{\thesection}{EC.\arabic{section}}
\renewcommand{\thesubsection}{\thesection.\arabic{subsection}}
\renewcommand{\thesubsubsection}{\thesubsection.\arabic{subsubsection}}
\numberwithin{equation}{section}
\setcounter{lemma}{0}
\renewcommand{\thelemma}{EC.\arabic{lemma}}
\setcounter{figure}{0}
\renewcommand{\thefigure}{EC.\arabic{figure}}
\setcounter{table}{0}
\renewcommand{\thetable}{EC.\arabic{table}}
\setcounter{algorithm}{0}
\renewcommand{\thealgorithm}{EC.\arabic{algorithm}}

\begin{bibunit}[plainnat]
\ECHead{\centering E-Companion to \\ From Frequentist to Bayesian Contextual Optimization}
\section{Omitted Proofs}

\subsection{Supplementary Results}

\subsubsection{PAC-Bayes and Generalized Bayesian Updating}
\begin{lemma}[Gibbs Posterior Optimality,~\textnormal{\citet{EC-Bissiri-2016-JRSSB}}]\label{lem:1} 
Let $(\Theta, \mathcal{B})$ be a measurable space, let prior $\rho_0$ be a probability measure on $(\Theta, \mathcal{B})$, and let $L:\Theta \rightarrow \mathbb{R}\cup\{ + \infty\}$ be $\mathcal{B}$-measurable. $\forall \lambda >0$, assume the normalizing constant
    $
        I_0 \coloneqq \int_\Theta e^{-\lambda L(\theta)}\rho_0(\mathrm{d}\theta)
    $
    satisfies  $I_0 < \infty$ and assume 
    $
        \int_\Theta \abs{L(\theta)}\rho(\mathrm{d}\theta) < \infty
    $
    for all probability measures $\rho$ on $(\Theta, \mathcal{B})$. Define the Gibbs measure $\rho^*(\mathrm{d}\theta)$ given any prior $\rho_0$ by:
    \begin{equation*}
        \rho^*(\mathrm{d}\theta) \coloneqq \frac{ e^{-\lambda L(\theta)} }{ I_0 } \rho_0(\mathrm{d}\theta).
    \end{equation*}
    Consider a functional $\mathcal{J}$ of probability measures $\rho$ on $(\Theta, \mathcal{B})$:
    \begin{equation*}
    \mathcal{J}(\rho) \coloneqq \begin{cases}
        \E_{\rho}[L(\theta)] + \frac{1}{\lambda }\KL{\rho}{\rho_0}, & \text{if } \rho \ll \rho_0,\\ 
        +\infty, & \text{otherwise},
    \end{cases}
    \end{equation*}
    where $\rho \ll \rho_0$ means $\rho$ is absolutely continuous with respect to $\rho_0$. 
    Then the Gibbs posterior $\rho^*$ is the unique minimizer of the functional $\mathcal{J}$, i.e. $\rho^* = \argmin_{\rho \in \mathcal{P}(\Theta)} \mathcal{J}(\rho)$.
\end{lemma}
\begin{proof}{Proof of Lemma~\ref{lem:1}}
    The proof is given in~\citet[page 8]{EC-Bissiri-2016-JRSSB}. \hfill\Halmos
\end{proof}
\begin{lemma}[PAC-Bayes Bound,~\textnormal{\citet{EC-Alquier-2024-BOOK}}]\label{lem:2}
    Denote $S = (X,Y), s = (x,y)$. Given a loss function $\ell(\theta, S): \theta \mapsto \R_+$, denote its risk and empirical counterpart in $\theta$ as:
    \begin{equation*}
        L(\theta) = \mathbb{E}_{\P}[\ell(\theta,S)],\quad  L_n(\theta) = \frac{1}{n}\sum_{i=1}^n \ell(\theta, s_i).
    \end{equation*}
    Assume $S_1,\dots, S_n \stackrel{i.i.d.}{\sim} \P$ and assume $\ell$ is a bounded random variable such that $\ell \in [0, 1]$.
    Let $\rho_0 \in\mathcal{P}(\Theta)$ be a prior independent of $\mathcal{D}_n$. Then for any $\lambda > 0$, and any $\eta \in (0,1)$, the following inequality holds with probability at least $1 - \eta$ over the draw $\mathcal{D}_n$ simultaneously for all $\rho \ll \rho_0$:
    \begin{equation*}
        \mathbb{E}_\rho[L(\theta)] \leq \E_\rho[L_n(\theta)] + \frac{ \KL{\rho}{\rho_0} + \log ( 1 / \eta)  }{\lambda n }  + \frac{\lambda }{8}. \tag{PAC-Bayes Bound}
    \end{equation*}
    Fix $\lambda, \eta,$ and $\mathcal{D}_n$. Minimizing the right-hand side of the inequality in $\rho$ is equivalent to minimizing 
    $
        \E_\rho[L_n(\theta)] + \frac{1}{\lambda n} \KL{\rho}{\rho_0}.
    $
    Therefore, the Gibbs posterior $\rho^*(\mathrm{d}\theta) = \exp\lra{ - \lambda n L_n(\theta)}\rho_0(\mathrm{d}\theta)$ minimizes the PAC-Bayes bound as shown in Lemma~\ref{lem:1}.
\end{lemma}
\begin{proof}{Proof of Lemma~\ref{lem:2}}
    The proof is given in~\citet[Theorem 2.1]{EC-Alquier-2024-BOOK}. \hfill\Halmos
\end{proof}
\subsubsection{Supplementary Lemmas}
\begin{lemma}\label{lem:3}
Suppose Assumption~\ref{A6} holds. For any $P, Q \in \mathcal{P}(\mathcal{Y})$,
\begin{enumerate}[label=(\alph*)]
    \item for every $z \in \mathcal{Z}$,
    \begin{equation*}
        \abs{V(z,P) - V(z,Q)} \leq L_y \mathcal{W}_1(P,Q);
    \end{equation*}
    \item for every $z \in \mathcal{Z}$, if $z_Q \in \argmin_{z \in \mathcal{Z}} V(z,Q)$, then
    \begin{equation*}
        V(z_Q,P) \leq V(z, P) + 2 L_y \mathcal{W}_1(P,Q);
    \end{equation*}
    \item if also $z_P \in \argmin_{z \in \mathcal{Z}} V(z,P)$, then
    \begin{equation*}
        0 \leq V(z_Q,P) - V(z_P, P) \leq 2 L_y \mathcal{W}_1(P,Q).
    \end{equation*}
\end{enumerate}
\end{lemma}
\proof{Proof of Lemma~\ref{lem:3}}
By Assumption~\ref{A6}, for each fixed $z$, the map $y \mapsto c(z;y)$ is $L_y$-Lipschitz. The Kantorovich-Rubinstein duality for $\mathcal{W}_1$ gives Lemma~\ref{lem:3}(a); see~\citet[Theorem 5.10]{OptimalTransport}.

For Lemma~\ref{lem:3}(b), use Lemma~\ref{lem:3}(a) twice and the optimality of $z_Q$ under $Q$:
\begin{equation*}
    \begin{aligned}
        V(z_Q, P) &\leq V(z_Q, Q) + L_y \mathcal{W}_1(P,Q) \\
        &\leq V(z, Q) + L_y \mathcal{W}_1(P,Q) \\
        &\leq V(z, P) + 2 L_y \mathcal{W}_1(P,Q).
    \end{aligned}
\end{equation*}
Lemma~\ref{lem:3}(c) follows from Lemma~\ref{lem:3}(b) with $z=z_P$, together with the optimality inequality $V(z_Q,P) \geq V(z_P, P)$.
\hfill\Halmos
\endproof
\subsection{Omitted Proofs in Section~\ref{sec:41}}

\subsubsection{Proof of Theorem~\ref{thm:1}}
\proof{Proof of Theorem~\ref{thm:1}}
    First, denote $ Z_n \coloneqq \sup_{\theta \in \Theta} \abs{R_n(\theta) - R(\theta)}$. Let $\P_n$ be the probability law generating $\mathcal{D}_n$, we first claim there exist universal constants $C_0, C_1 > 0$ such that if the sample size condition holds, then
    \begin{equation}
        \P_n\lra{Z_n \leq \varepsilon / 6} \geq 1 - \eta. \label{Thm1_1}
    \end{equation}
    
    By Assumption~\ref{A1} the loss class $\mathcal{F}_\Theta$ is uniformly bounded by $1$ and admits the constant envelope $F\equiv 1$. By Assumption~\ref{A3}, for every probability measure $\Q$ on $\mathcal X \times \mathcal Y$ and every $r \in (0, 1]$, 
    \begin{equation*}
        \log \mathcal{N}\lra{r, \mathcal{F}_\Theta, L_2(\mathbb{Q})} \leq d_\mathcal{F} \log (A_\mathcal{F} / r).
    \end{equation*}
    Hence the uniform entropy integral satisfies
    \begin{align}
        J(1,\mathcal{F}_\Theta) 
        &\coloneqq \sup_{\Q} \int_0^1 
        \sqrt{
        1 + \log \mathcal{N}\lra{u, \mathcal{F}_\Theta, L_2(\mathbb{Q})}
        }\,\mathrm{d}u \notag\\
        &\leq \int_0^1 \sqrt{1 + d_\mathcal{F} \log(A_\mathcal{F} / u)} \, \mathrm{d}u 
        \leq C_a \sqrt{d_\mathcal{F} \log (C_b A_\mathcal{F})} \label{Thm1_2},
    \end{align}
    for universal constants $C_a, C_b >0$ by using $d_\mathcal{F} \geq 1, \, A_\mathcal{F} \geq e$, and the change of variables $t=\sqrt{\log (A_\mathcal{F}/ u)}$. Applying the maximal inequality for VC-type classes with envelope 1~\citep[Theorem 2.14.1]{vandervaart-1996-book} to the centered process $\theta \mapsto R_n(\theta) - R(\theta)$ yields a universal constant $C_m$ with 
    \begin{equation}
        \E_{\P_n}[Z_n] \leq \frac{C_m J(1,\mathcal{F}_\Theta)}{\sqrt{n}} \leq 
        \frac{C_m C_a \sqrt{d_\mathcal{F} \log(C_b A_\mathcal{F})}}{\sqrt{n}}. \label{Thm1_3}
    \end{equation}

    Because the loss is bounded in $[0,1]$, replacing a single observation $(x_i,y_i)$ alters $Z_n$ by at most $1/ n$. McDiarmid’s bounded-differences inequality~\citep[Theorem 6.2]{Boucheron-2013-book} therefore yields
    \begin{equation}
        Z_n \leq \frac{C_m C_a \sqrt{d_\mathcal{F} \log(C_b A_\mathcal{F})}}{\sqrt{n}} + 
        \sqrt{\frac{\log(1/\eta)}{2n}}.
    \end{equation}
    A sufficient condition for the right-hand side to be at most $\varepsilon/6$ is:
    \begin{equation*}
        n \geq \frac{C_0}{\varepsilon^2}\lra{
        d_\mathcal{F} \log(C_1 A_\mathcal{F}) + \log (1 / \eta)
        },
    \end{equation*}
    for suitable universal constants $C_0 = 36\lra{C_a C_m + 1 / \sqrt{2}}^2$ and $C_1 = C_b$. Denote the event $E_n = \{Z_n \leq \varepsilon / 6\}$. The remainder of the argument is deterministic on $E_n$. Next, we bound the numerator and denominator of the Gibbs posterior.

    On the event $E_n$, every $\theta \in B_\varepsilon$ satisfies
    \begin{equation*}
        R_n(\theta)\geq R(\theta) - \frac{\varepsilon}{6} \geq R^\star + \varepsilon - \frac{\varepsilon}{6} = R^\star + \frac{5\varepsilon}{6}.
    \end{equation*}
    Since $\rho_0$ is a probability measure with $\rho_0(B_\varepsilon) \leq 1$,
    \begin{equation}
        \int_{B_\varepsilon} e^{-\lambda n R_n(\theta)} \rho_0(\mathrm{d}\theta) \leq 
        \exp\lra{-\lambda n (R^\star + 5\varepsilon /6)}.\label{Thm1_4}
    \end{equation}

    Consider the ball $\mathbb{B}_{r_\varepsilon}(\theta^\star)$. Because $r_\varepsilon \leq r_0 /2 < r_0 = \mathrm{dist}(\theta^\star, \partial \Theta)$ by Assumption~\ref{A2}(i), the ball is contained in the interior of $\Theta: \mathbb{B}_{r_\varepsilon}(\theta^\star) \subset \mathrm{int}(\Theta)$. Because $r_\varepsilon \leq s_\varepsilon \leq \bar{r}$, Assumption~\ref{A4} applies on this ball, yielding for every $\theta \in \mathbb{B}_{r_\varepsilon}(\theta^\star)$:
    \begin{equation*}
        R(\theta) - R^\star \leq \psi\lra{\norm{\theta - \theta^\star}} \leq \psi(r_\varepsilon) \leq \psi(s_\varepsilon) \leq \varepsilon / 6.
    \end{equation*}
    Combined with $Z_n \leq \varepsilon /6$ on $E_n$,
    \begin{equation*}
        R_n(\theta) \leq R(\theta) + \varepsilon / 6 \leq R^\star + \varepsilon / 3, \quad \forall \theta \in \mathbb{B}_{r_\varepsilon}(\theta^\star).
    \end{equation*}
    Using the prior lower bound $\rho_0(\theta) \geq l_0$ by Assumption~\ref{A2}(ii) and the standard Lebesgue volume $\mathrm{Vol}(\mathbb{B}_{r_\varepsilon}(\theta^\star)) = \mathbb{V}_d r_\varepsilon^d$,
    \begin{equation}
        \int_\Theta e^{-\lambda n R_n(\theta)} \rho_0(\mathrm{d}\theta) \geq 
        \int_{\mathbb{B}_{r_\varepsilon}(\theta^\star)} e^{-\lambda n R_n(\theta)} \rho_0(\mathrm{d}\theta) \geq 
        \exp\lra{
        -\lambda n (R^\star + \varepsilon/ 3) 
        }l_0\mathbb{V}_dr_\varepsilon^d \label{Thm1_5}.
    \end{equation}
    Combining~\eqref{Thm1_4} and~\eqref{Thm1_5} on event $E_n$ completes the proof. \hfill\Halmos
\endproof
\subsubsection{Proof of Corollary~\ref{cor:1}}
\proof{Proof of Corollary~\ref{cor:1}}
The proof consists of two parts by verifying Assumptions~\ref{A3} and~\ref{A4} separately.

We start with Assumption~\ref{A4}. Fix $\theta^\star \in \Theta^\star \cap \mathrm{int}(\Theta)$. For each $x \in \mathcal{X}$, write $z^\star(x) = \pi_{\theta^\star}(x)$ and recall the margin from Condition~\ref{cond:1}(iii):
\begin{equation*}
    \Delta^\star(x) = \min_{ z \in \mathcal{V} \setminus \{z^\star(x)\}} \lrb{ 
    c(z;f_{\theta^\star}(x)) - c(z^\star(x);f_{\theta^\star}(x))
    } \geq 0.
\end{equation*}
By Assumption~\ref{A5}(i), $f_\theta$ is $L_f$-Lipschitz in $\theta$ uniformly in $x$, so $\norm{f_\theta(x) - f_{\theta^\star}(x)} \leq L_f \norm{\theta - \theta^\star}$. By Assumption~\ref{A6}, $c(z;\cdot)$ is $L_y$-Lipschitz uniformly in $z \in \mathcal{V}$. Hence for every $z \in \mathcal{V}$,
\begin{equation}
    \abs{
    c(z;f_\theta(x)) - c(z;f_{\theta^\star}(x))
    } \leq L_y L_f \norm{\theta - \theta^\star}. \label{cor1:1}
\end{equation}
Applying~\eqref{cor1:1} to $z = z^\star(x)$ and to any $z \in \mathcal{V} \setminus \lrc{z^\star(x)}$:
\begin{align*}
    &c(z;f_\theta(x)) - c(z^\star(x);f_{\theta}(x)) \\
    &\quad \geq \lrb{c(z;f_{\theta^\star}(x)) - c(z^\star(x);f_{\theta^\star}(x))}
    - 2 L_y L_f \norm{\theta - \theta^\star} \\
    &\quad \geq \Delta^\star(x) - 2L_yL_f\norm{\theta - \theta^\star}.
\end{align*}
Therefore, if $\Delta^\star(x) > 2 L_y L_f \norm{\theta -\theta^\star}$, then $z^\star(x)$ strictly minimizes $c(\cdot;f_\theta(x))$ over $\mathcal{V}$, and Condition~\ref{cond:1}(i) gives $\pi_\theta(x) = z^\star(x) = \pi_{\theta^\star}(x)$. Contrapositively,
\begin{equation}
    \lrc{\pi_\theta(X) \neq \pi_{\theta^\star}(X)} \subseteq \lrc{\Delta^\star(X) \leq 2 L_y L_f \norm{\theta - \theta^\star}}. \label{cor1:2}
\end{equation}
By Assumption~\ref{A1}, $c\in[0,1]$, so $\abs{ c(\pi_\theta(X);Y) - c(\pi_{\theta^\star}(X);Y) } \leq 1$ and the difference vanishes when $\pi_\theta(X) = \pi_{\theta^\star}(X)$. Hence
\begin{equation}
    R(\theta) - R^\star \leq \P_X\lra{ \pi_\theta(X) \neq \pi_{\theta^\star}(X) }.
\end{equation}
Combining~\eqref{cor1:2} with Condition~\ref{cond:1}(iii), using $\P_{X}\lra{\Delta^\star(X) = 0}=0$, we obtain
\begin{equation*}
    R(\theta) - R^\star \leq C_{\Delta}\lra{2 L_y L_f \norm{\theta - \theta^\star}}^\beta.
\end{equation*}
This verifies Assumption~\ref{A4} with $\psi(r) = C_{\Delta}\lra{2 L_y L_f r}^\beta$ and $\bar{r} = D$. The function $\psi$ is non-decreasing and $\psi(r) \downarrow 0$ as $r \downarrow 0 $.

Next, we verify Assumption~\ref{A3}. By Condition~\ref{cond:1}(ii), for every pair $j \neq k$, the region
\begin{equation*}
    \lrc{y \in \mathcal{Y}: c(z_k;y) \leq c(z_j;y)}
\end{equation*}
is a Boolean formula in at most $M$ affine inequalities in $y$. Let $\lrc{\lra{a_t,b_t}}_{t=1}^{T}$ enumerate the distinct affine functionals appearing across all $\binom{K}{2}$ unordered pairs, so $T \leq \binom{K}{2}M\leq K^2M /2$. Define
\begin{equation*}
    h_t(x;\theta) \coloneqq \mathbb{I}\lrc{a_t^\top f_\theta(x) + b_t \leq 0}, \quad 
    \sigma_\theta(x) \coloneqq \lra{ h_1(x;\theta), \dots, h_T(x;\theta) } \in \lrc{0,1}^T.
\end{equation*}
By Condition~\ref{cond:1}(i)--(ii), $\zeta_\theta(x)$ is a fixed deterministic function of $\sigma_\theta(x)$. In particular, $c(\pi_\theta(x);y)$ depends on $(\theta,x)$ only through $\sigma_\theta(x)$.

Fix $t\in\lrc{1,\dots, T}$. By Assumption~\ref{A5}(ii), the scalar class $\mathcal{S}$ has pseudo-dimension at most $d_\mathcal{S}$. Since the pseudo-dimension of a real-valued class upper-bounds the Vapnik–Chervonenkis (VC) dimension of any of its threshold subclasses~\citep[Theorem 11.3]{EC-anthony-1999-book}, the class
\begin{equation*}
    \mathcal{H}_t \coloneqq \lrc{x \mapsto h_t(x;\theta): \theta \in \Theta}
\end{equation*}
of $\lrc{0,1}$-valued functions satisfies $\mathrm{VCdim}(\mathcal{H}_t) \leq d_\mathcal{S}$. Because $h_t \in \lrc{0,1}$, for any probability measure $\Q_X$ on $\mathcal{X}$, 
\begin{equation*}
    \Q_X\lra{h_t(\cdot;\theta_1) \neq h_t(\cdot;\theta_2)} = \norm{ h_t(\cdot;\theta_1) - h_t(\cdot;\theta_2) }_{L_2(\Q_X)}^{2}.
\end{equation*}
Since $\sigma_{\theta_1}(x) = \sigma_{\theta_2}(x) \Rightarrow \zeta_{\theta_1}(x) = \zeta_{\theta_2}(x)$, so a union bound over the $T$ coordinates yields, 
\begin{equation}
    \Q_X(\zeta_{\theta_1} \neq \zeta_{\theta_2}) \leq \sum_{t=1}^{T} \norm{ h_t(\cdot;\theta_1) - h_t(\cdot;\theta_2) }_{L_2(\Q_X)}^{2}. \label{cor1:3}
\end{equation}
For any probability measure $\Q$ on $\mathcal{X} \times \mathcal{Y}$ with marginal $\Q_X$, the cost values at $\theta_1, \theta_2$ coincide on the event $\lrc{\zeta_{\theta_1}(X) = \zeta_{\theta_2}(X)}$, and differ by at most 1 off this event by boundedness in Assumption~\ref{A1}. Hence
\begin{equation}
    \norm{
    c\lra{\pi_{\theta_1}(\cdot);\cdot} - c\lra{\pi_{\theta_2}(\cdot);\cdot}
    }_{L_2(\Q)}^2  \leq \Q_X\lra{\zeta_{\theta_1} \neq \zeta_{\theta_2}}
    \leq \sum_{t=1}^{T} \norm{ h_t(\cdot;\theta_1) - h_t(\cdot;\theta_2) }_{L_2(\Q_X)}^{2}.\label{cor1:4}
\end{equation}
For each $t$, fix an $r/\sqrt{T}$-net $\mathcal{N}_t \subset L_2(\Q_X)$ of $\mathcal{H}_t$ of minimum size $N_t \coloneqq \mathcal{N}(r/\sqrt{T}, \mathcal{H}_t, L_2(\Q_X))$. For every $\theta \in \Theta$, pick $g_t(\theta) \in \mathcal{N}_t$ with $\norm{h_t(\cdot;\theta) - g_t(\theta)}_{L_2(\Q_X)}\leq r / \sqrt{T}$. Because the elements of $\mathcal{N}_t \subset \mathcal{H}_t$ are $\lrc{0,1}$-valued, the tuple $\lra{g_1(\theta),\dots, g_T(\theta)} \in \lrc{0,1}^T$ is a valid input to the deterministic rule of choosing decisions; let $\hat{\zeta}_\theta(x) \in \lrc{1,\dots,K}$ be the rule applied to the tuple, and define
\begin{equation*}
    \hat{f}_\theta(x,y) \coloneqq c(z_{\hat{\zeta}_{\theta}(x)}; y).
\end{equation*}
The bound in~\eqref{cor1:4} applied with $\lra{h_t(\cdot;\theta), g_t(\theta)}$ in place of $\lra{h_t(\cdot;\theta_1), h_t(\cdot;\theta_2)}$ gives
\begin{equation*}
    \norm{c\lra{\pi_\theta(\cdot);\cdot} - \hat{f}_\theta}_{L_2(\Q)}^{2} \leq \sum_{t=1}^{T}\norm{h_t(\cdot;\theta) - g_t(\theta)}_{L_2(\Q_X)}^{2} \leq T \cdot (r/\sqrt{T})^2 = r^2.
\end{equation*}
The set $\lrc{\hat f_\theta: \theta \in \Theta}$ has cardinality at most $\prod_{t=1}^{T} N_t$, so
\begin{equation*}
    \log \mathcal{N}(r,\mathcal{F}_\Theta, L_2(\Q)) \leq \sum_{t=1}^{T} \log \mathcal{N}(r/\sqrt{T},\mathcal{H}_t, L_2(\Q_X)). \label{cor1:5}
\end{equation*}

By Haussler’s covering bound for $\lrc{0,1}$-valued VC classes~\citep[Theorem 2.6.7]{vandervaart-1996-book}, for any probability measure $\Q_X$, any VC class $\mathcal{H}$ with $\mathrm{VCdim}(\mathcal{H}) \leq C_V$, and any $\epsilon \in (0,1]$,
\begin{equation*}
    \log \mathcal{N}(\epsilon,\mathcal{H}, L_2(\Q_X)) \leq 2 C_V \log(2e /\epsilon).
\end{equation*}
Applying this with $C_V = d_\mathcal{S}$ and $\epsilon = r / \sqrt{T}$, then summing over $t$ via~\eqref{cor1:5},
\begin{equation*}
    \log \mathcal{N}(r,\mathcal{F}_\Theta, L_2(\Q)) \leq 2 T d_\mathcal{S} \log \frac{2e\sqrt{T}}{r}, \quad r \in (0,1].
\end{equation*}
Using $T \leq K^2 M / 2$ derives
\begin{equation*}
    2T d_\mathcal{S} \leq K^2Md_\mathcal{S}, \quad 2 e \sqrt{T} \leq \sqrt{2}eK \sqrt{M}.
\end{equation*}
Hence Assumption~\ref{A3} holds with
\begin{equation*}
    d_\mathcal{F} \leq K^2 M d_\mathcal{S}, \quad A_\mathcal{F} = \max\lrc{\sqrt{2}eK \sqrt{M}, e}.
\end{equation*}
Rename $\sqrt{2}e = C_{c1}$ completes the proof.\hfill\Halmos
\endproof
\subsubsection{Proof of Corollary~\ref{cor:2}}
\proof{Proof of Corollary~\ref{cor:2}}
The proof consists of two parts by verifying Assumptions~\ref{A3} and~\ref{A4} separately.

We start with Assumption~\ref{A4}. By the two Lipschitz conditions in Condition~\ref{cond:2},
\begin{equation*}
    \abs{ c\lra{\pi_{\theta_1}(x);y} - c\lra{\pi_{\theta_2}(x);y} } \leq L_z \norm{\pi_{\theta_1}(x) - \pi_{\theta_2}(x)} \leq L_z L_\theta \norm{\theta_1 - \theta_2}. \label{cor2:1}
\end{equation*}
Taking $\theta_2 = \theta^\star$ and integrating against $\P$,
\begin{equation*}
    R(\theta) - R^\star = \E_{\P}\lrb{
    c\lra{\pi_{\theta}(X);Y} - c\lra{\pi_{\theta^\star}(X);Y}
    } \leq L_z L_\theta \norm{\theta - \theta^\star}.
\end{equation*}
This verifies Assumption~\ref{A4} with $\psi(r) = L_z L_\theta r$ and $\bar{r} =D$, i.e., $\psi$ is globally Lipschitz.

Next, we verify Assumption~\ref{A3}.~\eqref{cor2:1} implies that for any probability measure $\Q$ on $\mathcal{X} \times \mathcal{Y}$, 
\begin{equation*}
    \norm{
    c\lra{ \pi_{\theta_1}(\cdot);\cdot } - c\lra{ \pi_{\theta_2}(\cdot);\cdot }
    }_{L_2(\Q)} \leq L_z L_\theta \norm{\theta_1 - \theta_2}.
\end{equation*}
Hence any $\epsilon$-net of $\lra{\Theta,  \norm{\cdot}}$ induces an $L_zL_\theta \epsilon$-net of $\mathcal{F}_\Theta$ in $L_2(\Q)$:
\begin{equation}
    \mathcal{N}(r, \mathcal{F}_\Theta, L_2(\Q)) \leq \mathcal{N} \lra{
    \frac{r}{L_z L_\theta}, \Theta, \norm{\cdot}
    }.\label{cor2:2}
\end{equation}
By Assumption~\ref{A2}(i), $\Theta$ is compact with diameter $D$. The standard volumetric covering bound~\citep[Lemma 2.5]{vandervaart-1996-book} gives, for every $\epsilon \in (0,D]$,
\begin{equation*}
    \mathcal{N}(\epsilon, \Theta, \norm{\cdot}) \leq \lra{\frac{3D}{\epsilon}}^d.
\end{equation*}
Substituting $\epsilon = r / (L_z L_\theta)$ into~\eqref{cor2:2}:
\begin{equation*}
    \log \mathcal{N}(r ,\mathcal{F}_\Theta, L_2(\Q)) \leq d \log \frac{3DL_zL_\theta}{r}, \quad r \in (0, L_z L_\theta D].
\end{equation*}
Set $A_\mathcal{F}= \max\lrc{3 D L_z L_\theta, e}$ to ensure $A_\mathcal{F} \geq e $ as required by Assumption~\ref{A3}. If $3 D L_z L_\theta \geq e $, then $A_\mathcal{F} = 3 D L_z L_\theta$ and the bound holds verbatim for all $ r \in (0, 1]$. If $3 D L_z L_\theta < e$, the $A_\mathcal{F} = e$ and $\log (A_\mathcal{F}/ r) = \log (e / r) \geq \log (3 D L_z L_\theta / r)$ for every $r > 0$, so the bound still holds. Hence Assumption~\ref{A3} holds with $d_\mathcal{F} = d$ and $A_\mathcal{F} = \max\{3 D L_z L_\theta ,e \}$, and Theorem~\ref{thm:1} applies. \hfill\Halmos
\endproof
\subsection{Omitted Proofs in Section~\ref{sec:42}}

\subsubsection{Proof of Proposition~\ref{prop:1}}
\proof{Proof of Proposition~\ref{prop:1}}
Fix $\theta^\star\in\Theta^\star$. Let
\begin{equation*}
    z^\circ_x \in \argmin_{z \in \mathcal{Z}} V(z, \P_{Y\mid X =x})
\end{equation*}
be an oracle decision under the true conditional distribution $P_x = \P_{Y \mid X = x}$, and set 
\begin{equation*}
    R^\circ \coloneqq \E_{\P_X}\lrb{V(z^\circ_X, \P_{Y \mid X})}.
\end{equation*}
Since $R^\star = R(\theta^\star) = \E_{\P_{X}}\lrb{ V(\pi_{\theta^\star}(X) , \P_{Y \mid X}) }$, we have $R^\circ \leq R^\star$.

By Lemma~\ref{lem:3}(c), applied pointwise with $P=P_x, \, Q = m(x;\rho)$, and $z_Q = \pi_\rho(x)$, 
\begin{equation*}
    V(\pi_\rho(x), P_x) - V(z^\circ_x, P_x) \leq 2 L_y \mathcal{W}_1(P_x, m(x;\rho)).
\end{equation*}
It remains to replace $m(x;\rho)$ by $m(x;\rho_G)$. If $\rho(B) > 0$, let $\rho_B = \rho(\cdot \mid B)$; otherwise the following is interpreted with the $B$-term equal to zero. Since
\begin{equation*}
    \rho = \rho(G)\rho_G + \rho(B) \rho_B,
\end{equation*}
Definition~\ref{def:1} gives
\begin{equation*}
    m(x;\rho) = \rho(G) m(x;\rho_G) + \rho(B) m(x;\rho_B).
\end{equation*}
Under Assumption~\ref{A5}(i), $m(x;\theta) = \delta(f_\theta(x))$ and $f_\theta$ is $L_f$-Lipschitz in $\theta$. Couple $m(x;\rho)$ and $m(x;\rho_G)$ as follows: with probability $\rho(G)$, use the same $\theta \in \rho_G$ in both distributions; with probability $\rho(B)$, use $\theta_B \sim \rho_B$ for $m(x;\rho)$ and an independent $\theta_G \sim \rho_G$ for $m(x;\rho_G)$. Since $\Theta$ has diameter $D$,
\begin{equation*}
    \mathcal{W}_1( m(x;\rho), m(x;\rho_G) ) \leq \rho(B) \E_{\rho_B,\rho_G}\lrb{\norm{f_{\theta_B}(x) - f_{\theta_G}(x)}} \leq \rho(B) L_f D.
\end{equation*}
Therefore,
\begin{equation*}
    \mathcal{R}(\rho) - R^\star \leq 2 L_y \E_{\P_X}\lrb{
    \mathcal{W}_1( P_x, m(x;\rho_G) ) + \mathcal{W}_1(m(x;\rho_G), m(x;\rho))
    }
    \leq \mathcal{E}(\rho_G) + 2 L_y L_f D \rho(B).
\end{equation*}
For every $x$, optimality of $z^\circ_x$ under $P_x$ gives $V(\pi_\rho(x), P_x) \geq V(z_x^\circ, P_x)$. Hence
\begin{equation*}
    V \lra{\pi_\rho(x), P_x} - V \lra{\pi_{\theta^\star}(x), P_x} \geq 
    V \lra{z_x^\circ, P_x} - V \lra{\pi_{\theta^\star}(x), P_x} \geq - 2 L_y \mathcal{W}_1\lra{P_x, m(x;\theta^\star)},
\end{equation*}
where the last inequality is Lemma~\ref{lem:3}(c) applied with $Q = m(x;\theta^\star)$, because $\pi_{\theta^\star}(x)$ minimizes $V(\cdot, m(x;\theta^\star))$.

By the triangle
inequality,
\begin{equation*}
    \mathcal{W}_1 (P_x, m(x;\theta^\star)) \leq 
    \mathcal{W}_1 (P_x, m(x;\rho_G)) +
    \mathcal{W}_1 (m(x;\rho_G), m(x;\theta^\star)) .
\end{equation*}
Using Assumption~\ref{A5}(i) again and coupling $\theta \sim \rho_G$ with the fixed point $\theta^\star$,
\begin{equation*}
    \mathcal{W}_1 (m(x;\rho_G) ,m(x;\theta^\star)) \leq \E_{\rho_G}\lrb{\norm{  f_\theta(x) - f_{\theta^\star}(x)  }} \leq L_f \E_{\rho_G}\lrb{\norm{\theta - \theta^\star}}.
\end{equation*}
Integrating over $\P_X$ yields
\begin{equation*}
    R(\rho) - R^\star \geq - \mathcal{E}(\rho_G) - \Delta(\rho_G).
\end{equation*}
If $\rho = \rho_n, \, B = B_\varepsilon, \, G = G_\varepsilon$, and $\rho_n(B_\varepsilon) \leq \delta_{n,\varepsilon}$, the preceding deterministic bound gives the claimed inequality. Using the upper bound in Theorem~\ref{thm:1} for $\rho_n(B_\varepsilon)$ completes the proof. \hfill\Halmos
\endproof
\subsubsection{Proof of Theorem~\ref{thm:2}}
\proof{Proof of Theorem~\ref{thm:2}}
For each $x$, denote $P_x = \P_{Y \mid X =x }$. Definition~\ref{def:2} gives
\begin{equation*}
    \pi_{\rho_n}(x) \in \argmin_{z \in \mathcal{Z}} V(z, m(x;\rho_n)).
\end{equation*}
Applying Lemma~\ref{lem:3}(b) with $P = P_x,\, Q = m(x;\rho_n),\, z_Q = \pi_{\rho_n}(x)$, and $z = \pi_\theta(x)$ yields, for every $\theta \in \Theta$,
\begin{equation*}
    V(\pi_{\rho_n}(x), P_x) \leq V(\pi_\theta(x), P_x) + 2 L_y \mathcal{W}_1(P_x, m(x;\rho_n)).
\end{equation*}
Integrating first with respect to $\rho_n(\mathrm{d}\theta)$ and then with respect to $\P_X$ gives 
\begin{equation}
    \mathcal{R}(\rho_n) \leq \E_{\rho_n}\lrb{R(\theta)} + \mathcal{E}(\rho_n). \label{thm2:1}
\end{equation}
Next, set 
\begin{equation*}
    \ell(\theta,(x,y)) \coloneqq c(\pi_\theta(x);y).
\end{equation*}
By Assumption~\ref{A1}, $\ell \in \lrb{0,1}$, its population risk is $R(\theta)$, and its empirical risk is $R_n(\theta)$ by definition. Applying Lemma~\ref{lem:2} with this loss, and then taking $\rho = \rho_n$, gives an event $\mathcal{A}$ with probability at least $1 -\eta$ on which
\begin{equation}
    \E_{\rho_n}\lrb{R(\theta)} \leq \E_{\rho_n}\lrb{R_n(\theta)} + 
    \frac
    {\KL{\rho_n}{\rho_0} + \log(1/\eta)}
    {\lambda n} + \frac{\lambda}{8}. \label{thm2:2}
\end{equation}
By Lemma~\ref{lem:1} applied to $L = R_n$ and inverse temperature $\lambda n$, the Gibbs posterior $\rho_n$ minimizes
\begin{equation*}
    \rho \mapsto \E_{\rho}\lrb{R_n(\theta)} + \frac{1}{\lambda n}\KL{\rho}{\rho_0}.
\end{equation*}
Since $\KL{\rho^\star}{\rho_0} < \infty$, on the same event $\mathcal{A}$,
\begin{equation}
    \E_{\rho_n}\lrb{R(\theta)} \leq \E_{\rho^\star}\lrb{R_n(\theta)} + 
    \frac
    {\KL{\rho^\star}{\rho_0} + \log(1/\eta)}
    {\lambda n} + \frac{\lambda}{8}. \label{thm2:3}
\end{equation}
It remains to replace $\E_{\rho^\star}\lrb{R_n(\theta)}$ by its population counterpart. Define
\begin{equation*}
    G(S_i) \coloneqq \E_{\rho^\star}\lrb{\ell(\theta, S_i)}, \quad S_i = (X_i, Y_i).
\end{equation*}
Again by Assumption~\ref{A1}, $0 \leq G(S_i) \leq 1$. Moreover,
\begin{equation*}
    \frac{1}{n}\sum_{i=1}^{n} G(S_i) = \E_{\rho^\star}\lrb{R_n(\theta)}, \quad 
    \E_{\P}\lrb{G(S_i)} = \E_{\rho^\star}\lrb{R(\theta)}.
\end{equation*}
Hoeffding’s inequality gives an event $\mathcal{B}$ with probability at least $1 - \eta$ on which
\begin{equation}
    \E_{\rho^\star}\lrb{R_n(\theta)} \leq \E_{\rho^\star}\lrb{R(\theta)} + \sqrt{\frac{\log(1/\eta)}{2n}}.\label{thm2:4}
\end{equation}
On $\mathcal{A} \cap \mathcal{B}$,  whose probability is at least $ 1 - 2 \eta$, combining~\eqref{thm2:1},~\eqref{thm2:3}, and~\eqref{thm2:4} yields
\begin{equation*}
\begin{aligned}
    \mathcal{R}(\rho_n) - \mathcal{R}^\star 
    &\leq \mathcal{E}(\rho_n) + \E_{\rho^\star}\lrb{R(\theta)} - \mathcal{R}^\star 
    +  \frac {\KL{\rho^\star}{\rho_0} + \log(1/\eta)}{\lambda n} + \frac{\lambda}{8} + \sqrt{\frac{\log(1/\eta)}{2n}} \\
    &= \mathcal{E}(\rho_n) + \Gamma^\star +  \frac {\KL{\rho^\star}{\rho_0} + \log(1/\eta)}{\lambda n} + \frac{\lambda}{8} + \sqrt{\frac{\log(1/\eta)}{2n}}.
\end{aligned}
\end{equation*}
This proves the stated high-probability bound.

Finally, minimizing the right-hand side in $\lambda$ gives $\lambda = O(n^{-1/2})$ for fixed $\eta$ and $\KL{\rho^\star}{\rho_0}$. All other $\lambda$-related terms are in $O(n^{-1/2})$ accordingly. \hfill\Halmos
\endproof

\newpage
\section{Experimental Details}
\subsection{Omitted Algorithms}\label{EC_algorithm}
We present the two algorithms mentioned in Section~\ref{sec:5} here.
\begin{algorithm}
\caption{Mean-Field Variational Inference Gradient Descent (Gaussian / Laplace)}\label{algo:1}
\begin{algorithmic}[1]
\SingleSpacedXI
\small
\Procedure{MeanFieldVI}{$\mathcal{D}_n, \lambda, \alpha, \rho_0, T, \mathcal{Q}$}
    \State $\{(\mu_j, \sigma_j)\}_{j=1}^{d} \gets$ Initialization
    \For{$t \gets 1$ \textbf{to} $T$}
        \\
        \State \( 
        \xi_t \gets
        \begin{cases}
            u \sim \mathcal{N}(0, \mathbf{I}_d), & \text{if } \mathcal{Q} = \textsc{Gaussian} \\
            -\mathrm{sign}(u)\log(1 - \abs{u}),\, u \sim \mathrm{Uniform}(-1,1)^d, & \text{if } \mathcal{Q} = \textsc{Laplace}
        \end{cases}
        \)
        \\
        \State $\theta_t \gets \mu + \xi_t \circ \sigma$
        \State \(
        J(\theta_t) \gets \frac{1}{n}\sum_{i=1}^{n} c(\pi_{\theta_t}(x_i); y_i) + \frac{1}{\lambda n}\lrb{\log q_{\mu,\sigma}(\theta_t) - \log \rho_0(\theta_t)}
        \)
        \State $\mu_j \gets \mu_j - \alpha \cdot \nabla_{\theta_t} J(\theta_t)$
        \State $\sigma_j \gets \sigma_j - \alpha \,\xi_{j,t} \circ \nabla_{\theta_t} J(\theta_t)$
    \EndFor
    \State \Return $\{(\mu_j, \sigma_j)\}_{j=1}^{d}$
\EndProcedure
\end{algorithmic}
\end{algorithm}

\begin{algorithm}[htbp]
\caption{Adaptive Metropolis-Hastings for Nondifferentiable Problems}\label{algo:2}
\begin{algorithmic}[1]
\SingleSpacedXI
\small
\Procedure{AdaptiveMH}{$\mathcal{D}_n, \lambda, \rho_0, T_b, T, A^\star, \theta_0, v_0, s_0$}
\State $\Theta_{\mathrm{MH}} \gets \{\}, \; \hat{\mu} \gets \theta_{0}, \; \hat{v} \gets v_0$
\For{$t = 1, 2, \ldots, T$}                            
\State $\alpha_t = ( 1 + t)^{-0.6}$              
\State $\theta^p_{t} \gets \theta \sim \mathcal{N}(\theta_{t-1}, s_{t-1}^2 v_{t-1})$
\State 
\(
M_t \gets  L_\lambda(\mathcal{D}_n \mid \theta^p_t) /L_\lambda(\mathcal{D}_n \mid \theta_{t-1}) \cdot  \rho_0(\theta^p_t) /  \rho_0(\theta_{t-1}),  \; \text{Draw}\; a \sim \mathrm{Uniform}(0,1)
\)
\If{$a \leq M_t$} 
\State $\theta_t = \theta_t^p, \quad A_t = 1.$
\Else
\State $\theta_t = \theta_{t-1},\quad A_t = 0.$
\EndIf
\If{$t \leq T_b$}
\State  $\hat{\mu}_{t} \gets \hat{\mu}_{t-1} + \alpha_t(\theta_{t} - \hat{\mu}_{t-1})$
\State  $\hat{v}_{t} \gets ( 1-\alpha_t)\hat{v}_{t-1} + \alpha_t(\theta_{t} - \hat{\mu}_{t-1})^{\circ 2}$
\State  $ s_t \gets s_{t-1}\exp\lra{\alpha_t(A_t - 0.234)}$.
\Else 
\State  $\Theta_{\mathrm{MH}} \gets \Theta_{\mathrm{MH}} \cup \{\theta_{t}\}$.
\EndIf
\EndFor
\State \Return $\Theta_{\mathrm{MH}}$
\EndProcedure
\end{algorithmic}
\end{algorithm}

\subsection{Experiment 1}\label{ECEXP1}
For completeness, we present the details of the synthetic problem proposed by~\cite{Bertsimas-2020-MS-EC}. In this problem, the complete cost function is
\begin{align*}
    c(z;y) &= c_q^\top z + Q(z,y), \\
    Q(z,y) &= \min_{z_l, z_t \geq 0}\lrc{c_l^\top z_l + \sum_{i=1}^{I}\sum_{j=1}^{J} c_{t,ij}z_{t,ij},
    \sum_{i=1}^{I} z_{t,ij} \geq y_{j}\; \forall j, \sum_{j=1}^{J} z_{t,ij} \leq z_{i} + z_{l,i} \; \forall i
    },
\end{align*}
where $c_q = (5,5,5,5)^\top \in \R^4, c_l = (100, 100, \dots, 100)^\top \in \R^{12}$. The distance matrix is 
\begin{equation*}
    c_t = \begin{pmatrix}
0.15 & 1.3124 & 1.85 & 1.3124 \\
0.50026 & 0.93408 & 1.7874 & 1.6039 \\
0.93408 & 0.50026 & 1.6039 & 1.7874 \\
1.3124 & 0.15 & 1.3124 & 1.85 \\
1.6039 & 0.50026 & 0.93408 & 1.7874 \\
1.7874 & 0.93408 & 0.50026 & 1.6039 \\
1.85 & 1.3124 & 0.15 & 1.3124 \\
1.7874 & 1.6039 & 0.50026 & 0.93408 \\
1.6039 & 1.7874 & 0.93408 & 0.50026 \\
1.3124 & 1.85 & 1.3124 & 0.15 \\
0.93408 & 1.7874 & 1.6039 & 0.50026 \\
0.50026 & 1.6039 & 1.7874 & 0.93408
\end{pmatrix}.
\end{equation*}
Furthermore, the 3-dimensional ARMA(2,2) process is used to generate the feature vector $X$:
\begin{equation*}
    X(t) = \Phi_1 X(t-1) + \Phi_2 X(t-2) + \Theta_1 U(t-1) + \Theta_2 U(t-2) + U(t),
\end{equation*}
where $U(t) \sim \mathcal{N}(0, \Sigma_U)$ are i.i.d. innovations, and 
\begin{align*}
&\Phi_1 =
\begin{pmatrix}
0.5 & -0.9 & 0 \\
1.1 & -0.7 & 0 \\
0 & 0 & 0.5
\end{pmatrix},
\quad
\Phi_2 =
\begin{pmatrix}
0. & -0.5 & 0 \\
0.5 & 0 & 0 \\
0 & 0 & 0
\end{pmatrix},\\
&\Theta_1 =
\begin{pmatrix}
0.4 & 0.8 & 0 \\
-1.1 & -0.3 & 0 \\
0 & 0 & 0
\end{pmatrix},
\quad
\Theta_2 =
\begin{pmatrix}
0 & -0.8 & 0 \\
-1.1 & 0 & 0 \\
0 & 0 & 0
\end{pmatrix}, 
\Sigma_U =
\begin{pmatrix}
1. & 0.5 & 0 \\
0.5 & 1.2 & 0.5 \\
0 & 0.5 & 0.8
\end{pmatrix}.
\end{align*}
When generating all the training, validation, and testing sets, we randomly draw $X(0)$ and $X(1)$ from $\mathcal{N}(0, \mathbf{I}_3)$, and continue to generate future states $X(t), t > 2$.

The following factor model is employed to generate $Y$ for each $X$:
\begin{equation*}
    Y = \max\lrc{ 0, A^\top (X+\delta /4) + (B^\top X)\epsilon},
\end{equation*}
where $\delta$ and $\epsilon$ are independent standard Gaussian samples, and
\begin{equation*}
    A = 2.5 \times
\begin{pmatrix}
0.8 & 0.1 & 0.1 \\
0.1 & 0.8 & 0.1 \\
0.1 & 0.1 & 0.8 \\
0.8 & 0.1 & 0.1 \\
0.1 & 0.8 & 0.1 \\
0.1 & 0.1 & 0.8 \\
0.8 & 0.1 & 0.1 \\
0.1 & 0.8 & 0.1 \\
0.1 & 0.1 & 0.8 \\
0.8 & 0.1 & 0.1 \\
0.1 & 0.8 & 0.1 \\
0.1 & 0.1 & 0.8
\end{pmatrix},
\quad
B = 7.5 \times
\begin{pmatrix}
0 & -1 & -1 \\
-1 & 0 & -1 \\
-1 & -1 & 0 \\
0 & -1 & 1 \\
-1 & 0 & 1 \\
-1 & 1 & 0 \\
0 & 1 & -1 \\
1 & 0 & -1 \\
1 & -1 & 0 \\
0 & 1 & 1 \\
1 & 0 & 1 \\
1 & 1 & 0
\end{pmatrix}.
\end{equation*}
The exponential Nadaraya-Watson kernel estimator prescribes the decision for a new context $x_\mathrm{new}$ as follows:
\begin{align*}
    \argmin_{z \geq 0}& \quad c_q^\top z + \sum_{i=1}^{n} w^{\mathrm{NW}}_i Q(z,y_i), \\
    \text{s.t.} &\quad  w^{\mathrm{NW}}_i = \frac{\exp\lra{-\norm{x_\mathrm{new} - x_i}_2/h}}{
    \sum_{j=1}^{n} \exp\lra{-\norm{x_\mathrm{new} - x_j}_2/h}
    }.
\end{align*}
We allow the bandwidth $h$ to take values in $\{0.10, 0.15, 0.20,\dots, 0.50\}$, and we choose the bandwidth that achieves the lowest validation loss for out-of-sample evaluation.

For both IEO and BCO, the cost gradient with respect to a prediction $\hat{y}_i$ is derived as follows:
\begin{align*}
    \hat{z}_i &= \argmin_{z\geq 0}\; c_q^\top z + Q(z,\hat{y}_i), \\
    \nabla_{\hat{y}_i}c(\hat{z}_i; y_i) &= \nabla_{\hat{y}_i} \lra{ c_q^\top \hat{z}_i + Q(\hat{z}_i, y_i) }.
\end{align*}
To train IEO, we use the \texttt{SGD} optimizer with learning rate 0.1 and batch size 50, and employ early stopping using the validation set with patience of 20 epochs. To train BCO, we also use the \texttt{SGD} optimizer with learning rate 0.1 and batch size 50, and set $\lambda n = C_\lambda (n/50)^{0.5}$ where $C_\lambda = 0.01$. This value is chosen by sweeping several values of $C_\lambda$. Similarly, we early stop the training of BCO until no further improvement on the validation set is observed for 20 epochs. We set the number of scenarios $K = 100$ for BCP in this experiment. The prior is set independently as $\mathcal{N}(0,1)$ for each parameter in the piecewise affine model.
\subsection{Experiment 2}\label{ECEXP2}
The factor vectors in this experiment are:
\begin{equation*}
\begin{aligned}
    A &= \lra{0.409, 0.908, 0, 0, -2.061, 0.254, 3.024, 1.280}^\top, \\
    B &= \lra{ 1.386, -0.902, 5.437, 0, 0, -0.482, 4.611, 0}^\top.
\end{aligned}
\end{equation*}
    In this experiment, IEO is trained using the \texttt{Adam} optimizer with learning rate 0.01 and batch size 128. We implement early stopping with patience of 20 epochs. Moreover, we sweep regularization strength $\gamma \in \{10^{-6}, 10^{-5}, \dots, 10^{0}\}$ for IEO, and report the results for the model trained with the value of $\gamma$ that achieves the best performance over the validation set. Similarly, we use the same optimizer, learning rate, and batch size for BCO. The prior is set independently as $\mathrm{Laplace}(0,1)$ for each parameter in the neural network. We set $\lambda n = C_\lambda (n/256)^{0.5}$ and choose $C_\lambda = 10^{-6}$ for all DGPs, determined by sweeping several values to balance bias-variance on the validation set.

\subsection{Experiment 3}\label{ECEXP3}
We use a synthetic data-generating process based on the contextual portfolio optimization setting in~\cite{EC-SPO}. The model generates paired samples $(x_i,y_i)$ where $x_i \in \R^5$ is the contextual feature vector and $y_i \in \R^{50}$ is the realized return vector of 50 assets. To generate this dataset, we first sample a feature-loading matrix $B \in \R^{50 \times 5}$, where each entry is independently drawn from a Bernoulli distribution with probability 0.5. This matrix determines which contextual features affect the expected return of each asset. Next, a factor-loading matrix $E \in \R^{50 \times 4}$, whose entries are independently drawn from the uniform distribution $\mathrm{Uniform}(-0.0025 \tau, 0.0025\tau)$, is generated to produce the covariance matrix of the asset returns $\Sigma = E E^\top + (0.01\tau)^2\mathbf{I}_{50}$. In this work, we take $\tau =2$. Additionally, we fix $B$ and $E$ across all experiments once they are generated to ensure a consistent analysis across sample sizes.

For each observation $i=1,\dots,n$, the uncertain return vector $y_i$ is generated by:
\begin{equation*}
    y_i = \lra{ \frac{0.05}{\sqrt{5}} B x_i + 0.1^{1/p}}^{p} + F_i E^\top + \epsilon_i,
\end{equation*}
where $p$ is the polynomial degree and is set to 1 in this work, feature vector $x_i$ is drawn from $\mathcal{N}(0,\mathbf{I}_5)$, the latent factor shock $F_i$ is drawn from $\mathcal{N}(0,\mathbf{I}_4)$, and the idiosyncratic noise $\epsilon_i$ is drawn from $\mathcal{N}(0,(0.01\tau)^2\mathbf{I}_{50})$. In each independent experiment for a given sample size, we change the random seed to ensure variation across runs.

In this experiment, we demonstrate the benefit of distributional estimators by incorporating them into risk-averse optimization models, adopting a conditional value-at-risk (CVaR) model. The ENW-CVaR policy in this experiment can be formulated as follows for a new context $x_{\mathrm{new}}$:
\begin{equation*}
    \begin{aligned}
        \argmin_{z \geq 0}\min_{\eta, v}  \quad & z^\top \Sigma \, z  \\
        \text{s.t.}\quad & w^{\mathrm{NW}}_i = \frac{\exp\lra{-\norm{x_\mathrm{new} - x_i}_2/h}}{
            \sum_{j=1}^{n} \exp\lra{-\norm{x_\mathrm{new} - x_j}_2/h}}, \\
        & \eta + \frac{1}{1-\alpha} \sum_{i=1}^{n} v_iw^{\mathrm{NW}}_i \leq 0, \\
        & v_i \geq U_0  - y_i^\top z  - \eta, && i = 1,\dots, n,\\
        & v \geq 0, \; \bm{1}^\top z \leq 1.
    \end{aligned}
\end{equation*}
We sweep $h \in \{0.20, 0.25, 0.30, 0.35, 0.40\}$ for ENW. In Figure~\ref{fig:case3}, we report the bandwidth that achieves the best feasibility. Accordingly, $h=0.25$ is chosen when $n$ is 100 and 200; $h=0.20$ is chosen when $n = 400$.

Similarly, our BCP policy becomes:
\begin{equation*}
    \begin{aligned}
        \argmin_{z \geq 0}\min_{\eta, v}  \quad & z^\top \Sigma \, z  \\
        \text{s.t.}\quad & \eta + \frac{1}{1-\alpha} \sum_{k=1}^{K} v_kw_k \leq 0, \\
        & v_k \geq U_0  - y_k^\top z  - \eta, && k = 1,\dots, K,\\
        & v \geq 0, \; \bm{1}^\top z \leq 1,
    \end{aligned}
\end{equation*}
where $y_k$ are samples from the PPD and $w_k = 1 / K$ are their weights.

\begin{algorithm}[htbp]
\caption{Adaptive Metropolis-Hastings for Hard Constraint Problems}
\label{algo:3}
\begin{algorithmic}[1]
\OneAndAHalfSpacedXI
\small
\Procedure{FeasibleMH}{$\mathcal{D}_n, \lambda, \rho_0, T_b, T, \theta_0, v_0, s_0, F_0$}
\State $\Theta_{\mathrm{MH}} \gets \{\}, \; \hat{\mu} \gets \theta_{0}, \; \hat{v} \gets v_0$
\For{$t = 1, 2, \ldots, T$}                            
\State $\alpha_t = ( 1 + t)^{-0.6}$              
\State $\theta^p_{t} \gets \theta \sim \mathcal{N}(\theta_{t-1}, s_{t-1}^2 v_{t-1})$
\State 
\(
M_t \gets  L_\lambda(\mathcal{D}_n \mid \theta^p_t) /L_\lambda(\mathcal{D}_n \mid \theta_{t-1}) \cdot  \rho_0(\theta^p_t) /  \rho_0(\theta_{t-1}),  \; \text{Draw}\; a \sim \mathrm{Uniform}(0,1)
\)
\If{$a \leq M_t \mathbb{I}\{ F_n(\theta_t^p) \geq F_0 \}$} 
\State $\theta_t = \theta_t^p, \quad A_t = 1.$
\Else
\State $\theta_t = \theta_{t-1},\quad A_t = 0.$
\EndIf
\If{$t \leq T_b$}
\State  $\hat{\mu}_{t} \gets \hat{\mu}_{t-1} + \alpha_t(\theta_{t} - \hat{\mu}_{t-1})$
\State  $\hat{v}_{t} \gets ( 1-\alpha_t)\hat{v}_{t-1} + \alpha_t(\theta_{t} - \hat{\mu}_{t-1})^{\circ 2}$
\State  $ s_t \gets s_{t-1}\exp\lra{\alpha_t(A_t - 0.234)}$.
\Else 
\State  $\Theta_{\mathrm{MH}} \gets \Theta_{\mathrm{MH}} \cup \{\theta_{t}\}$.
\EndIf
\EndFor
\State \Return $\Theta_{\mathrm{MH}}$. 
\EndProcedure
\end{algorithmic}
\end{algorithm}

To accommodate the requirement of feasibility due to the hard constraint in this problem, we adjust the MH algorithm to take feasibility into consideration. The modified algorithm is summarized in Algorithm~\ref{algo:3}, in which we introduce a feasibility measure to accept or reject a sample. Its definition is given by:
\begin{equation*}
    F_n(\theta) \coloneqq \frac{1}{n} \sum_{i=1}^{n} \mathbb{I}\lrc{ y_i^\top \pi_\theta(x_i) \geq U_0 },
\end{equation*}
which measures the empirical marginal feasibility of a frequentist policy $\pi_\theta$ over the dataset $\mathcal{D}_n$. We only accept parameters that ensure $F_n(\theta) \geq F_0$, where $F_0$ is the preferred feasibility level. We set $F_0 = \alpha = 0.95$ in this experiment. To prioritize feasibility in this problem, we do not consider a gradual sequence for $\lambda n$ but fix $\lambda n = 1000$ for all sample sizes in this experiment. The burn-in length for each MCMC chain is 5,000, and the full-chain length is 15,000. We run four independent MCMC chains for BCP.

To the best of our knowledge, there is no decision-focused frequentist IEO framework that can address nonlinear problems with uncertainty in the constraints with strong theoretical guarantees. Therefore, we turn to the heuristic proposed by~\cite{EC-Mandi-2025-NIPS}, which we denote IEO in this experiment. This approach proposes two surrogate loss functions for decision quality (called OPL) and feasibility (called IPL), and constructs a convex combination of them to train a point predictor. When applying this approach to our problem, we found that it struggled to obtain an acceptable parameter vector. Therefore, we use MSE loss to warm-start this approach. We train the model until the empirical feasibility exceeds the target $\alpha=0.95$. Note that this approach supports only deterministic optimization and is not applicable to the CVaR model. We repeat this training procedure five times for each sample size and report the one with the highest feasibility on the test set.

\clearpage
\renewcommand{\refname}{References for the Appendix}
\putbib[ref_ec]
\end{bibunit}


\begin{thebibliography}{55}
\providecommand{\natexlab}[1]{#1}
\providecommand{\url}[1]{\texttt{#1}}
\expandafter\ifx\csname urlstyle\endcsname\relax
  \providecommand{\doi}[1]{doi: #1}\else
  \providecommand{\doi}{doi: \begingroup \urlstyle{rm}\Url}\fi

\bibitem[Agrawal et~al.(2026)Agrawal, Avadhanula, Goyal, and
  Zeevi]{Agrawal-2025-MOR}
Shipra Agrawal, Vashist Avadhanula, Vineet Goyal, and Assaf Zeevi.
\newblock Thompson sampling for the multinomial logit bandit.
\newblock \emph{Mathematics of Operations Research}, 51\penalty0 (1):\penalty0
  568--590, 2026.

\bibitem[Alquier(2024)]{Alquier-2024-BOOK}
Pierre Alquier.
\newblock User-friendly introduction to {PAC-Bayes} bounds.
\newblock \emph{Foundations and Trends in Machine Learning}, 17\penalty0
  (2):\penalty0 174--303, 2024.

\bibitem[Anthony and Bartlett(1999)]{anthony-1999-book}
Martin Anthony and Peter~L. Bartlett.
\newblock \emph{Neural Network Learning: Theoretical Foundations}.
\newblock Cambridge University Press, Cambridge, UK, 1999.

\bibitem[Azoury and Miyaoka(2009)]{Azoury-2009-MS}
Katy~S Azoury and Julia Miyaoka.
\newblock Optimal policies and approximations for a {Bayesian} linear
  regression inventory model.
\newblock \emph{Management Science}, 55\penalty0 (5):\penalty0 813--826, 2009.

\bibitem[Ban and Rudin(2019)]{Ban-2019-OR}
Gah-Yi Ban and Cynthia Rudin.
\newblock The big data newsvendor: Practical insights from machine learning.
\newblock \emph{Operations Research}, 67\penalty0 (1):\penalty0 90--108, 2019.

\bibitem[Bartlett et~al.(2019)Bartlett, Harvey, Liaw, and
  Mehrabian]{Bartlett-2019-JMLR}
Peter~L. Bartlett, Nick Harvey, Christopher Liaw, and Abbas Mehrabian.
\newblock Nearly-tight {VC}-dimension and pseudo-dimension bounds for piecewise
  linear neural networks.
\newblock \emph{Journal of Machine Learning Research}, 20\penalty0
  (63):\penalty0 1--17, 2019.

\bibitem[Bertsimas and Kallus(2020)]{Bertsimas-2020-MS}
Dimitris Bertsimas and Nathan Kallus.
\newblock From predictive to prescriptive analytics.
\newblock \emph{Management Science}, 66\penalty0 (3):\penalty0 1025--1044,
  2020.

\bibitem[Bertsimas and Koduri(2022)]{Bertsimas-2022-OR}
Dimitris Bertsimas and Nihal Koduri.
\newblock Data-driven optimization: A reproducing kernel {Hilbert} space
  approach.
\newblock \emph{Operations Research}, 70\penalty0 (1):\penalty0 454--471, 2022.

\bibitem[Bertsimas and Mundru(2023)]{Bertsimas-2023-OR}
Dimitris Bertsimas and Nishanth Mundru.
\newblock Optimization-based scenario reduction for data-driven two-stage
  stochastic optimization.
\newblock \emph{Operations Research}, 71\penalty0 (4):\penalty0 1343--1361,
  2023.

\bibitem[Bertsimas and Van~Parys(2022)]{Bertsimas-2022-MP}
Dimitris Bertsimas and Bart Van~Parys.
\newblock Bootstrap robust prescriptive analytics.
\newblock \emph{Mathematical Programming}, 195\penalty0 (1):\penalty0 39--78,
  2022.

\bibitem[Bissiri et~al.(2016)Bissiri, Holmes, and Walker]{Bissiri-2016-JRSSB}
P.~G. Bissiri, C.~C. Holmes, and S.~G. Walker.
\newblock A general framework for updating belief distributions.
\newblock \emph{Journal of the Royal Statistical Society Series B: Statistical
  Methodology}, 78\penalty0 (5):\penalty0 1103--1130, 2016.

\bibitem[Blei et~al.(2017)Blei, Kucukelbir, and McAuliffe]{Blei-2017-JASA}
David~M Blei, Alp Kucukelbir, and Jon~D McAuliffe.
\newblock Variational inference: A review for statisticians.
\newblock \emph{Journal of the American Statistical Association}, 112\penalty0
  (518):\penalty0 859--877, 2017.

\bibitem[Chen and Ryzhov(2020)]{Chen-2020-OR}
Ye~Chen and Ilya~O Ryzhov.
\newblock Consistency analysis of sequential learning under approximate
  {Bayesian} inference.
\newblock \emph{Operations Research}, 68\penalty0 (1):\penalty0 295--307, 2020.

\bibitem[Dias~Garcia et~al.(2025)Dias~Garcia, Street, Homem-de Mello, and
  Mu{\~n}oz]{Dias-2025-OR}
Joaquim Dias~Garcia, Alexandre Street, Tito Homem-de Mello, and Francisco~D
  Mu{\~n}oz.
\newblock Application-driven learning: A closed-loop prediction and
  optimization approach applied to dynamic reserves and demand forecasting.
\newblock \emph{Operations Research}, 73\penalty0 (1):\penalty0 22--39, 2025.

\bibitem[Elmachtoub and Grigas(2022)]{SPO}
Adam~N Elmachtoub and Paul Grigas.
\newblock Smart “predict, then optimize”.
\newblock \emph{Management Science}, 68\penalty0 (1):\penalty0 9--26, 2022.

\bibitem[Elmachtoub et~al.(2023)Elmachtoub, Lam, Zhang, and
  Zhao]{Elmachtoub-2023-ARXIV}
Adam~N Elmachtoub, Henry Lam, Haofeng Zhang, and Yunfan Zhao.
\newblock Estimate-then-optimize versus integrated-estimation-optimization
  versus sample average approximation: a stochastic dominance perspective.
\newblock \emph{arXiv preprint arXiv:2304.06833}, 2023.

\bibitem[Elmachtoub et~al.(2025)Elmachtoub, Lam, Lan, and
  Zhang]{Elmachtoub-2025-AISTATS}
Adam~N. Elmachtoub, Henry Lam, Haixiang Lan, and Haofeng Zhang.
\newblock Dissecting the impact of model misspecification in data-driven
  optimization.
\newblock In \emph{Internat. Conf. Artificial Intelligence Statist.}, volume
  258, pages 1594--1602. PMLR, 2025.

\bibitem[Garthwaite et~al.(2016)Garthwaite, Fan, and
  Sisson]{Garthwaite-2016-CSTM}
Paul~H Garthwaite, Yanan Fan, and Scott~A Sisson.
\newblock Adaptive optimal scaling of {Metropolis--Hastings} algorithms using
  the {Robbins--Monro} process.
\newblock \emph{Communications in Statistics-Theory and Methods}, 45\penalty0
  (17):\penalty0 5098--5111, 2016.

\bibitem[Goldberg and Jerrum(1995)]{Goldberg-1995-ML}
Paul~W. Goldberg and Mark~R. Jerrum.
\newblock Bounding the vapnik-chervonenkis dimension of concept classes
  parameterized by real numbers.
\newblock \emph{Machine Learning}, 18\penalty0 (2-3):\penalty0 131--–148,
  1995.

\bibitem[Han et~al.(2025)Han, Hu, and Shen]{Han-2025-MS}
Jinhui Han, Ming Hu, and Guohao Shen.
\newblock Deep neural newsvendor.
\newblock \emph{Management Science}, Articles in Advance, 2025.

\bibitem[Harrison et~al.(2012)Harrison, Keskin, and Zeevi]{Harrison-2012-MS}
J~Michael Harrison, N~Bora Keskin, and Assaf Zeevi.
\newblock {Bayesian} dynamic pricing policies: Learning and earning under a
  binary prior distribution.
\newblock \emph{Management Science}, 58\penalty0 (3):\penalty0 570--586, 2012.

\bibitem[Henrion and R{\"o}misch(2022)]{Henrion-2022-MP}
R{\'e}ne Henrion and Werner R{\"o}misch.
\newblock Problem-based optimal scenario generation and reduction in stochastic
  programming.
\newblock \emph{Mathematical Programming}, 191\penalty0 (1):\penalty0 183--205,
  2022.

\bibitem[Ho-Nguyen and K{\i}l{\i}n{\c{c}}-Karzan(2022)]{Ho-Nguyen-2022-MS}
Nam Ho-Nguyen and Fatma K{\i}l{\i}n{\c{c}}-Karzan.
\newblock Risk guarantees for end-to-end prediction and optimization processes.
\newblock \emph{Management Science}, 68\penalty0 (12):\penalty0 8680--8698,
  2022.

\bibitem[Hu et~al.(2022)Hu, Kallus, and Mao]{Hu-2022-MS}
Yichun Hu, Nathan Kallus, and Xiaojie Mao.
\newblock Fast rates for contextual linear optimization.
\newblock \emph{Management Science}, 68\penalty0 (6):\penalty0 4236--4245,
  2022.

\bibitem[Jiang and Tanner(2008)]{Wenxin-2008-AS}
Wenxin Jiang and Martin~A. Tanner.
\newblock {Gibbs posterior for variable selection in high-dimensional
  classification and data mining}.
\newblock \emph{The Annals of Statistics}, 36\penalty0 (5):\penalty0
  2207--2231, 2008.

\bibitem[Kallus and Mao(2023)]{Kallus-2023-MS}
Nathan Kallus and Xiaojie Mao.
\newblock Stochastic optimization forests.
\newblock \emph{Management Science}, 69\penalty0 (4):\penalty0 1975--1994,
  2023.

\bibitem[Kannan et~al.(2025)Kannan, Bayraksan, and Luedtke]{Kannan-2025-OR}
Rohit Kannan, G{\"u}zin Bayraksan, and James~R Luedtke.
\newblock Data-driven sample average approximation with covariate information.
\newblock \emph{Operations Research}, 73\penalty0 (6):\penalty0 3245--3259,
  2025.

\bibitem[Lin et~al.(2022)Lin, Chen, Li, and Shen]{Lin-2022-POMS}
Shaochong Lin, Youhua Chen, Yanzhi Li, and Zuo-Jun~Max Shen.
\newblock Data-driven newsvendor problems regularized by a profit risk
  constraint.
\newblock \emph{Production and Operations Management}, 31\penalty0
  (4):\penalty0 1630--1644, 2022.

\bibitem[Luo et~al.(2023)Luo, Guo, and Wang]{Luo-2023-MSOM}
Zhenwei Luo, Pengfei Guo, and Yulan Wang.
\newblock Manage inventories with learning on demands and buy-up substitution
  probability.
\newblock \emph{Manufacturing \& Service Operations Management}, 25\penalty0
  (2):\penalty0 563--580, 2023.

\bibitem[Mak and Joseph(2018)]{Mak-2018-AOS}
Simon Mak and V~Roshan Joseph.
\newblock Support points.
\newblock \emph{The Annals of Statistics}, 46\penalty0 (6A):\penalty0
  2562--2592, 2018.

\bibitem[Mandi et~al.(2024)Mandi, Kotary, Berden, Mulamba, Bucarey, Guns, and
  Fioretto]{DFLreview}
Jayanta Mandi, James Kotary, Senne Berden, Maxime Mulamba, Victor Bucarey, Tias
  Guns, and Ferdinando Fioretto.
\newblock Decision-focused learning: Foundations, state of the art, benchmark
  and future opportunities.
\newblock \emph{Journal of Artificial Intelligence Research}, 80:\penalty0
  1623--1701, 2024.

\bibitem[Mandi et~al.(2025)Mandi, Defresne, Berden, and Guns]{Mandi-2025-NIPS}
Jayanta Mandi, Marianne Defresne, Senne Berden, and Tias Guns.
\newblock Feasibility-aware decision-focused learning for predicting parameters
  in the constraints.
\newblock In \emph{{Adv. Neural Inform. Processing Systems}}, volume~38, pages
  163285--163309, Red Hook, NY, 2025. Curran Associates, Inc.

\bibitem[Murphy(2012)]{Murphy-2012-book}
Kevin~P. Murphy.
\newblock \emph{Machine Learning: A Probabilistic Perspective}.
\newblock MIT Press, Cambridge, MA, 2012.

\bibitem[Nott et~al.(2023)Nott, Drovandi, and Frazier]{Nott-2023-ARSA}
David~J Nott, Christopher Drovandi, and David~T Frazier.
\newblock {Bayesian} inference for misspecified generative models.
\newblock \emph{Annual Review of Statistics and Its Application}, 11:\penalty0
  179--202, 2023.

\bibitem[Qi et~al.(2023)Qi, Shi, Qi, Ma, Yuan, Wu, and Shen]{Qi-2023-MS}
Meng Qi, Yuanyuan Shi, Yongzhi Qi, Chenxin Ma, Rong Yuan, Di~Wu, and Zuo-Jun
  Shen.
\newblock A practical end-to-end inventory management model with deep learning.
\newblock \emph{Management Science}, 69\penalty0 (2):\penalty0 759--773, 2023.

\bibitem[Qi et~al.(2025)Qi, Grigas, and Shen]{Qi-2025-OR}
Meng Qi, Paul Grigas, and Zuo-Jun Shen.
\newblock Integrated conditional estimation-optimization.
\newblock \emph{Operations Research}, 74\penalty0 (3):\penalty0 1604--1625,
  2025.

\bibitem[Rahimian and Pagnoncelli(2023)]{Rahimian-2023-SIAMOPT}
Hamed Rahimian and Bernardo Pagnoncelli.
\newblock Data-driven approximation of contextual chance-constrained stochastic
  programs.
\newblock \emph{SIAM Journal on Optimization}, 33\penalty0 (3):\penalty0
  2248--2274, 2023.

\bibitem[Ribeiro and Fanzeres(2026)]{Rafaela-2026-EJOR}
Rafaela Ribeiro and Bruno Fanzeres.
\newblock Integrated estimate-and-optimize decision trees learning for
  two-stage linear decision-making problems.
\newblock \emph{European Journal of Operational Research}, 329\penalty0
  (2):\penalty0 607--628, 2026.

\bibitem[Roberts and Rosenthal(2001)]{Roberts-2001-SS}
Gareth~O Roberts and Jeffrey~S Rosenthal.
\newblock Optimal scaling for various {Metropolis-Hastings} algorithms.
\newblock \emph{Statistical Science}, 16\penalty0 (4):\penalty0 351--367, 2001.

\bibitem[Rockafellar and Wets(1998)]{Rockafellar-2009-book}
R.~Tyrrell Rockafellar and Roger J.-B. Wets.
\newblock \emph{Variational Analysis}, volume 317 of \emph{Grundlehren der
  mathematischen Wissenschaften}.
\newblock Springer, Berlin, Heidelberg, 1998.
\newblock Corrected 3rd printing, 2009.

\bibitem[Ro{\v{c}}kov{\'a} and van~der Pas(2020)]{Rockova-2020-AOS}
Veronika Ro{\v{c}}kov{\'a} and St{\'e}phanie van~der Pas.
\newblock Posterior concentration for {B}ayesian regression trees and forests.
\newblock \emph{The Annals of Statistics}, 48\penalty0 (4):\penalty0
  2108--2131, 2020.

\bibitem[Sadana et~al.(2025)Sadana, Chenreddy, Delage, Forel, Frejinger, and
  Vidal]{EJORReview}
Utsav Sadana, Abhilash Chenreddy, Erick Delage, Alexandre Forel, Emma
  Frejinger, and Thibaut Vidal.
\newblock A survey of contextual optimization methods for decision-making under
  uncertainty.
\newblock \emph{European Journal of Operational Research}, 320\penalty0
  (2):\penalty0 271--289, 2025.

\bibitem[Satop{\"a}{\"a}(2022)]{Satop-2022-OR}
Ville~A Satop{\"a}{\"a}.
\newblock Regularized aggregation of one-off probability predictions.
\newblock \emph{Operations Research}, 70\penalty0 (6):\penalty0 3558--3580,
  2022.

\bibitem[Shapiro et~al.(2023)Shapiro, Zhou, and Lin]{Shapiro-2023-SIAMOPT}
Alexander Shapiro, Enlu Zhou, and Yifan Lin.
\newblock {Bayesian} distributionally robust optimization.
\newblock \emph{SIAM Journal on Optimization}, 33\penalty0 (2):\penalty0
  1279--1304, 2023.

\bibitem[Shen et~al.(2024)Shen, Jiao, Lin, Horowitz, and Huang]{Shen-2024-JMLR}
Guohao Shen, Yuling Jiao, Yuanyuan Lin, Joel~L Horowitz, and Jian Huang.
\newblock Nonparametric estimation of non-crossing quantile regression process
  with deep {ReQU} neural networks.
\newblock \emph{Journal of Machine Learning Research}, 25\penalty0
  (88):\penalty0 1--75, 2024.

\bibitem[Sim et~al.(2025)Sim, Tang, Zhou, and Zhu]{Sim-2025-OR}
Melvyn Sim, Qinshen Tang, Minglong Zhou, and Taozeng Zhu.
\newblock The analytics of robust satisficing: Predict, optimize, satisfice,
  then fortify.
\newblock \emph{Operations Research}, 73\penalty0 (5):\penalty0 2708--2728,
  2025.

\bibitem[South et~al.(2022)South, Riabiz, Teymur, and Oates]{South-2022-ARSA}
Leah~F South, Marina Riabiz, Onur Teymur, and Chris~J Oates.
\newblock Postprocessing of {MCMC}.
\newblock \emph{Annual Review of Statistics and Its Application}, 9:\penalty0
  529--555, 2022.

\bibitem[Spokoiny(2012)]{Spokoiny-2012-AOS}
Vladimir Spokoiny.
\newblock Parametric estimation. {F}inite sample theory.
\newblock \emph{The Annals of Statistics}, 40\penalty0 (6):\penalty0
  2603--2655, 2012.

\bibitem[Sriram et~al.(2013)Sriram, Ramamoorthi, and Ghosh]{Sriram-2013-BA}
Karthik Sriram, R.V. Ramamoorthi, and Pulak Ghosh.
\newblock Posterior consistency of {Bayesian} quantile regression based on the
  misspecified asymmetric {Laplace} density.
\newblock \emph{Bayesian Analysis}, 8\penalty0 (2):\penalty0 479--504, 2013.

\bibitem[Sun et~al.(2023)Sun, Liu, and Li]{Sun-2023-ICML}
Chunlin Sun, Shang Liu, and Xiaocheng Li.
\newblock Maximum optimality margin: A unified approach for contextual linear
  programming and inverse linear programming.
\newblock In \emph{Internat. Conf. Machine Learning}, volume 202, pages
  32886--32912. PMLR, 2023.

\bibitem[Wainwright(2019)]{Wainwright-2019-Book}
Martin~J. Wainwright.
\newblock \emph{High-Dimensional Statistics: A Non-Asymptotic Viewpoint}.
\newblock Cambridge Series in Statistical and Probabilistic Mathematics.
  Cambridge University Press, Cambridge, 2019.

\bibitem[Wang et~al.(2026)Wang, Srivastava, Hanasusanto, and
  Ho]{Wang-2026-MSOM}
Yijie Wang, Prateek~R. Srivastava, Grani~A. Hanasusanto, and Chin~Pang Ho.
\newblock On data-driven prescriptive analytics with side information: A
  regularized {Nadaraya--Watson} approach.
\newblock \emph{Manufacturing \& Service Operations Management}, 28\penalty0
  (3):\penalty0 841--859, 2026.

\bibitem[Wu et~al.(2018)Wu, Zhu, and Zhou]{Wu-2018-SIAMOPT}
Di~Wu, Helin Zhu, and Enlu Zhou.
\newblock A {Bayesian} risk approach to data-driven stochastic optimization:
  Formulations and asymptotics.
\newblock \emph{SIAM Journal on Optimization}, 28\penalty0 (2):\penalty0
  1588--1612, 2018.

\bibitem[Wu and Martin(2023)]{Wu-2023-BA}
Pei-Shien Wu and Ryan Martin.
\newblock A comparison of learning rate selection methods in generalized
  {Bayesian} inference.
\newblock \emph{Bayesian Analysis}, 18\penalty0 (1):\penalty0 105--132, 2023.

\bibitem[Yu and Moyeed(2001)]{Yu-2001-SPL}
Keming Yu and Rana~A Moyeed.
\newblock Bayesian quantile regression.
\newblock \emph{Statistics \& Probability Letters}, 54\penalty0 (4):\penalty0
  437--447, 2001.

\end{thebibliography}


\begin{thebibliography}{9}
\providecommand{\natexlab}[1]{#1}
\providecommand{\url}[1]{\texttt{#1}}
\expandafter\ifx\csname urlstyle\endcsname\relax
  \providecommand{\doi}[1]{doi: #1}\else
  \providecommand{\doi}{doi: \begingroup \urlstyle{rm}\Url}\fi

\bibitem[Alquier(2024)]{EC-Alquier-2024-BOOK}
Pierre Alquier.
\newblock User-friendly introduction to {PAC-Bayes} bounds.
\newblock \emph{Foundations and Trends in Machine Learning}, 17\penalty0
  (2):\penalty0 174--303, 2024.

\bibitem[Anthony and Bartlett(1999)]{EC-anthony-1999-book}
Martin Anthony and Peter~L. Bartlett.
\newblock \emph{Neural Network Learning: Theoretical Foundations}.
\newblock Cambridge University Press, Cambridge, UK, 1999.

\bibitem[Bertsimas and Kallus(2020)]{Bertsimas-2020-MS-EC}
Dimitris Bertsimas and Nathan Kallus.
\newblock From predictive to prescriptive analytics.
\newblock \emph{Management Science}, 66\penalty0 (3):\penalty0 1025--1044,
  2020.

\bibitem[Bissiri et~al.(2016)Bissiri, Holmes, and
  Walker]{EC-Bissiri-2016-JRSSB}
P.~G. Bissiri, C.~C. Holmes, and S.~G. Walker.
\newblock A general framework for updating belief distributions.
\newblock \emph{Journal of the Royal Statistical Society Series B: Statistical
  Methodology}, 78\penalty0 (5):\penalty0 1103--1130, 2016.

\bibitem[Boucheron et~al.(2013)Boucheron, Lugosi, and
  Massart]{Boucheron-2013-book}
Stéphane Boucheron, Gábor Lugosi, and Pascal Massart.
\newblock \emph{Concentration Inequalities: A Nonasymptotic Theory of
  Independence}.
\newblock Oxford University Press, 02 2013.

\bibitem[Elmachtoub and Grigas(2022)]{EC-SPO}
Adam~N. Elmachtoub and Paul Grigas.
\newblock Smart ``predict, then optimize''.
\newblock \emph{Management Science}, 68\penalty0 (1):\penalty0 9--26, 2022.

\bibitem[Mandi et~al.(2025)Mandi, Defresne, Berden, and
  Guns]{EC-Mandi-2025-NIPS}
Jayanta Mandi, Marianne Defresne, Senne Berden, and Tias Guns.
\newblock Feasibility-aware decision-focused learning for predicting parameters
  in the constraints.
\newblock In \emph{{Adv. Neural Inform. Processing Systems}}, volume~38, pages
  163285--163309, Red Hook, NY, 2025. Curran Associates, Inc.

\bibitem[van~der Vaart and Wellner(1996)]{vandervaart-1996-book}
Aad~W. van~der Vaart and Jon~A. Wellner.
\newblock \emph{Weak Convergence and Empirical Processes}.
\newblock Springer Series in Statistics. Springer New York, New York, NY, 1996.

\bibitem[Villani(2009)]{OptimalTransport}
Cédric Villani.
\newblock \emph{Optimal transport: old and new}.
\newblock Springer, Berlin, 2009.

\end{thebibliography}
\end{document}